\documentclass[11pt,a4paper]{amsart}

\usepackage{amssymb, amsmath, amsthm, amscd}

\usepackage{enumerate}

\usepackage{mathrsfs}

\usepackage{tikz-cd}

\usepackage{float}

\usepackage{changepage}

\usepackage{array}

\theoremstyle{plain} \newtheorem{theorem}{Theorem}[section]
\theoremstyle{plain} \newtheorem{corollary}[theorem]{Corollary}
\theoremstyle{plain} \newtheorem{proposition}[theorem]{Proposition}
\theoremstyle{plain}\newtheorem{lemma}[theorem]{Lemma}
\theoremstyle{plain}\newtheorem{conjecture}[theorem]{Conjecture}
\theoremstyle{definition} \newtheorem{definition}[theorem]{Definition}
\theoremstyle{definition}\newtheorem{example}[theorem]{Example}
\theoremstyle{remark}\newtheorem{remark}[theorem]{Remark}
\theoremstyle{definition}
\theoremstyle{remark}
\theoremstyle{definition}\newtheorem{question}[theorem]{Question}
\theoremstyle{definition}
\newcommand{\N}{{\mathbb{N}}}
\newcommand{\C}{{\mathbb{C}}}
\newcommand{\Z}{{\mathbb{Z}}}

\newcommand{\X}{\mathcal X}
\newcommand{\T}{\mathscr{T}(M)}
\newcommand{\TZ}{\mathscr{T}^{\Z}(M)}
\newcommand{\D}{\text{\rm Diff}^0(M)}
\newcommand{\DZ}{\text{\rm Diff}^\Z(M)}
\newcommand{\I}{\mathcal I(M)}
\newcommand{\Azero}{\text{\rm Aut}^0(X_0)}
\newcommand{\Aun}{\text{\rm Aut}^1(X_0)}
\newcommand{\A}{\text{\rm Aut}(X_0)}
\newcommand{\AZ}{\text{\rm Aut}^\Z(X_0)}

\newcommand{\git}{\mathbin{\!/\mkern-5mu/\!}}

\numberwithin{equation}{section}

\usepackage{hyperref}
\hypersetup{colorlinks = true, linkcolor=blue}

\title{Geography of the Teichm\"uller stack}
\author{Laurent Meersseman}
\date{\today}


\address{Laurent Meersseman\\
	Univ Angers, CNRS, Larema, SFR Mathstic\\
	F-49000 Angers, France\\ laurent.meersseman@univ-angers.fr}

\thanks{The author benefits from the support of the French government “Investissements d’Avenir” program integrated to France 2030, bearing the following reference ANR-11-LABX-0020-01.}

\begin{document}
\begin{abstract}
In this article, we describe the geography of the Teichm\"uller stack of \cite{LMStacks} and of one of its variants we introduce here, giving some answers to questions as: which points are orbifold points? What are the different local models of special points?... We give a rough description in the general case, and we use the compacity of the cycle spaces to get a much more detailed picture in the K\"ahler setting. 
\end{abstract}

\maketitle

\section{Introduction.}
\label{intro}

Let $X_0$ be a compact complex manifold with underlying $C^\infty$ manifold denoted by $M$. For such a non-necessarily K\"ahler manifold, Hodge Theory does not apply and the set of complex structures close to $X_0$ is controlled by its Kuranishi space through Kodaira-Spencer theory of deformations. Thus to obtain the full moduli space of complex structures on $M$, it is theoretically enough to glue at most a countable number of Kuranishi spaces, or a suitable quotient of them. 

This process cannot however be realized in this degree of generality with classical, GIT, or orbifold quotients. General Artin Stacks are needed. In \cite{LMStacks} (see also \cite{Cortona} for a comprehensive presentation), under a very mild hypothesis, we build a moduli stack of complex structures on $M$, resp. a Teichm\"uller stack of $M$, describing explicitly the gluing process involved. Rather than gluing Kuranishi spaces, we glue Kuranishi stacks. These stacks are, roughly speaking, the quotient of Kuranishi base spaces by the automorphism group of the base points.  The moduli and Teichm\"uller stacks thus obtained are an analytic enriched version of the topological moduli and Teichm\"uller spaces.

This being said, the next step consists of analyzing the geometric structure of these stacks, understanding why they are in general neither analytic spaces nor orbifolds, classifying the different types of points, giving adequate local models of the special points and establishing a cartography of them. This is what we begin to do in this paper. 

A first very well known obstruction for the moduli/Teichm\"uller space being locally an analytic space at some complex manifold $X_0$, is the fact that the dimension of the automorphism group of $X_0$ may differ from the dimension of the automorphism group of close complex manifolds. This dimension is an upper semicontinuous function of the Kuranishi space for the Zariski topology \cite{Grauert}\footnote{\label{ftgrauert} Indeed, upper semicontinuity is proven in \cite{K-Susc} for the ordinary topology on a smooth base, and the general case can be deduced from the proof of Grauert's direct image theorem in \cite{Grauert}.}, so the geography of the points where it jumps is clear. They are located on a strict analytic subspace. Moreover, the corresponding Kuranishi spaces have a foliated structure described in \cite{ENS} from which local models can be derived. All this analysis can be transposed to the moduli/Teichm\"uller stack, cf. \S \ref{localstack}.

Another well known fact is that the action of the mapping class group on the Teichm\"uller space (whose quotient is the moduli space) can be very wild with dense or locally dense orbits. This is the case for complex tori of dimension at least $2$, for K3 surfaces, ... So we avoid this wildness by considering here only the Teichm\"uller stack $\T$ and not the moduli stack. We also introduce a variant of it, the $\Z$-Teichm\"uller stack $\TZ$, that basically satisfies all the properties of $\T$.  

We show in this paper that there exists another subtler phenomenon. In the brief account of the construction of the Teichm\"uller stack at the beginning of this introduction, we said that we glue Kuranishi stacks to obtain it. This is a slight simplification for it is not always true that the Teichm\"uller stack is locally isomorphic at some $X_0$ to the Kuranishi stack of $X_0$. Points where this does not happen are called exceptional or $\Z$-exceptional in the case of $\TZ$. They exhibit a different local model and some strange properties that are analyzed in \S \ref{secexc}. Their geography is also not so clear than that of jumping points.

It is not easy to find exceptional or $\Z$-exceptional points - none of the classical examples admits one. Indeed, this work was strongly delayed because, for a long time, we did not have any example. We finally build a non-K\"ahler example of a $\Z$-exceptional point that is presented in \S \ref{secexample} and Theorem \ref{thmexampleexc}. We note that it is not exceptional.
In the K\"ahler case, that is when we restrict to the open\footnote{\label{fnoteKahler} By a classical result of Kodaira-Spencer, K\"ahlerianity is a stable property through small deformations.} substack of K\"ahler points of the Teichm\"uller stack, we show in Theorem \ref{2ndmainthm} that the closure of exceptional points as well as that of $\Z$-exceptional points form a strict analytic substack, making use of the compacity of the cycle spaces. It is important to stress that all the arguments using compacity of cycle spaces break completely when $X_0$ is neither K\"ahler nor in Fujiki class $(\mathscr{C})$, so that it is natural to expect a dichotomy between the K\"ahler and the non-K\"ahler cases. Pushing forward this analysis, we state in the very polarized Conjecture \ref{mainconj} that there does not exist neither exceptional nor $\Z$-exceptional K\"ahler points; whereas such points can be dense in a connected component of non-K\"ahler points of the Teichm\"uller or the $\Z$-Teichm\"uller stack. Theorems \ref{thmexampleexc} and \ref{2ndmainthm} cited above are the best results we obtained on exceptional points. But they left wide open several assertions of Conjecture \ref{mainconj} starting from the existence of an exceptional point as well as several important associated questions, especially whether the set of exceptional points is closed.

 As a consequence of this dichotomy, the local structure of $\T$ or $\TZ$ at a general non-K\"ahler point may be much more complicated than at a K\"ahler point. An example of this phenomenon for $\TZ$ is given in Corollary \ref{corexampleAut1Z}. This is somewhat surprising since, at the level of the Kuranishi space (and Kuranishi stack), there is no difference: the Kuranishi space of a K\"ahler, even of a projective, manifold can exhibit all the pathologies (for example not irreducible \cite{Horikawa}, not reduced \cite{Mumford}, arbitrary singularities \cite{Vakil}) the Kuranishi space of a non-K\"ahler, non class $(\mathscr{C})$ one can have. Moreover, due to the possibly wild action of the mapping class group, there is no difference between them at the level of the Riemann moduli stack. {\sl This difference only appears when considering the Teichm\"uller stack}. Its full complexity is only seen at non-K\"ahler non class $(\mathscr{C})$ points {\sl hence its geometry cannot be fully understood without dealing with such manifolds}. We hope to understand this better in the future.

Going back to the paper, we introduce the notion of normal and $\Z$-normal points in \S \ref{secnormalpoints}. They almost coincide with non-jumping, non-exceptional points. They form an open substack of $\T$ and $\TZ$. We show in Theorem \ref{thmnormal} that both can be reduced to an \'etale stack, that is a stack locally isomorphic to an at most discrete quotient of an analytic space. These normal Teichm\"uller stacks are quite easy to handle. Once again, and for the same reasons, more can be said in the K\"ahler case: the normal Teichm\"uller stacks are then orbifolds.

At the end of the day, we obtain the following rough description of the geography of the Teichm\"uller stacks in the general case (see Theorem \ref{thmstructuregeneral} for a precise statement): 
\begin{enumerate}[---]
	\item Jumping points are the most pathological points but are quite well understood and form a strict analytic substack of $\T$ and $\TZ$.
	\item Exceptional points exhibit a different strange behaviour but also quite well understood, falling into three types. Their existence and geography, as well as those of the different types, are however unclear.
	\item Normal points form an open substack associated to the \'etale normal Teichm\"uller stack that is easy to handle.
	\item Complementary points (if exist) have a non-reduced Kuranishi space with some special property.
\end{enumerate}
In the K\"ahler case\footnote{ Some - but not all - statements are valid under the weaker Fujiki class $(\mathscr{C})$ hypothesis.}, making use of the compacity of the cycle spaces allows us to get a much more detailed picture of the geography of $\T$ and $\TZ$ (see Theorem \ref{thmstructureKahler} for a precise statement): 
\begin{enumerate}[---]
	\item The closure of exceptional points is a strict analytic substack of $\T$ and of $\TZ$.
	\item The normal Teichm\"uller stack is an orbifold.
\end{enumerate}

\smallskip

Here is an outline of the paper. The main protagonists, that is the Teichm\"uller and Kuranishi stacks, are introduced in $\S$\ref{stack}-\ref{localstack}. The material comes essentially from \cite{LMStacks} but with some slight differences and additions. Notably, we show in Theorem \ref{thmuniversal} that the germ of Kuranishi stack at a point has a universal property. This generalizes the semi-universality property of the Kuranishi space. At a rough level, this is folklore (see for example \cite{VerbitskySurvey}), but we never saw a precise statement of such a property, probably because the stack setting developed in \cite{LMStacks} is necessary to a get clear formulation. We introduce the $\Z$-Teichm\"uller stack in \S \ref{secZteich}. Every statement about the Teichm\"uller stack given in this paper holds with obvious changes for the $\Z$-Teichm\"uller stack. In the sequel, to avoid redundancies, we often content ourselves with stating the full results and definitions only for the Teichm\"uller stack and  with briefly indicating the changes needed for the $\Z$-Teichm\"uller stack. The main interest of introducing $\TZ$ is the already cited Theorem \ref{thmexampleexc} of a $\Z$-exceptional point. It is not an example of an exceptional point but it is a strong evidence that such points exist.
Section \S \ref{localTeich} begins the local analysis of $\T$ and $\TZ$.  It culminates with Theorem \ref{finitethm} showing that the natural morphism from the Kuranishi stack to the Teichm\"uller stack is \'etale. This gives a complete solution to the problem of comparing these two stacks at a point. Section \S \ref{secconjecture} begins with Definition \ref{defexcpoint} of exceptional points and then states the main Conjecture \ref{mainconj} about them.
In \S \ref{localTeich2}, we define and study normal points and the normal Teichm\"uller stack. The main result is the already mentionned Theorem \ref{thmnormal}. Of interest also is the characterization of points with Kuranishi stacks being orbifolds in \S \ref{Kurorbifold}. We then swith to the analysis of exceptional points in \S \ref{secexc}. We introduce the cycle spaces that are related to $\T$ and $\TZ$ but the most important results are Theorem \ref{thmexampleexc} on a $3$-fold that is a $\Z$-exceptional point and the associated Theorem \ref{thmexampleAut1Z} showing that the subgroup of the automorphism group of these $3$-folds formed by elements inducing the identity in cohomology with coefficients in $\Z$ is infinite discrete. Such a phenonmenon cannot occur on Kähler manifolds, so we obtain here points whose $\Z$-Teichm\"uller stack is not locally isomorphic to that of any Kähler manifold, see Corollary \ref{corexampleAut1Z}. The next two sections \S \ref{secfinite}-\ref{secdistri} focus on the K\"ahler case. We first draw all the consequences of the compacity of cycle spaces in the K\"ahler setting, as Theorem \ref{thmfinite2} showing that the \'etale morphism from the Kuranishi stack to the Teichm\"uller stack is indeed finite, allowing to characterize easily points where $\T$ is an orbifold from the analogous results proven for the Kuranishi stack. Then we show in Theorem \ref{2ndmainthm} that the closure of exceptional points is an analytic substack of $\T$ and $\TZ$.  
As an interlude, \S \ref{secvarexc} gives somes variations about exceptionality, for example introducing the notion of exceptional pairs and showing that non-separated pairs of points are exceptional pairs in \S \ref{secNH}. Going back to the K\"ahler setting, we investigate in \S \ref{secpathologies} the structure of $\T$ from the side of pathological families. One of the main interests in using stacks is to give a dictionnary between properties of $\T$ as a moduli space and properties of families of compact complex manifolds diffeomorphic to $M$. In this way, jumping points are related to jumping families that is families with all fibers biholomorphic except for one. We define several types of pathologies and analyze then from both the family and the moduli space point of view. The philosophy developed in \S \ref{secpathologies} is that, at least in the K\"ahler case, pathologies only may occur at jumping and/or exceptional points so occur on a strict analytic substack. We also comment in \S \ref{seclocaliso} on a false statement of \cite{ENS}. Finally, we gather all the previous results to state in \S \ref{seccarto} the two main results (Theorems \ref{thmstructuregeneral} and \ref{thmstructureKahler}) on the cartography of $\T$ and $\TZ$ in the general and in the K\"ahler case. Many additional comments are included. We also revisit this geography through the concept of holonomy points that sheds a different light on our results. Jumping points are points with continuous holonomy, normal points have at most discrete holonomy, and exceptional points are points a neighborhood of which is not controlled by the holonomy group. Section \S \ref{seccarto} ends with some remarks on Teichm\"uller stacks and GIT quotients.

 The notions of exceptional vs. normal points are reminiscent from
Catanese question in \cite{Catsurvey}, see also \cite{Catsurvey2}, on conditions under which the Teichm\"uller and Kuranishi spaces are locally homeomorphic. This question has played a central role in the genesis of this work. We turned it into giving conditions under which the Teichm\"uller and Kuranishi stacks are locally isomorphic. We first thought that this is always the case, before the concept of exceptional points emerges.

I am indebted to An-Khuong Doan, Julien Grivaux and Etienne Mann for illuminating discussions on some parts of this work.

\section{The Teichm\"uller Stack: basic facts}
\label{stack}
We recollect some facts about the Teichm\"uller stack of a connected, compact oriented $C^\infty$ manifold $M$ admitting complex structures. We refer to \cite{LMStacks} for more details.

The general idea is the following. From the one hand, the Teichm\"uller stack is the category of analytic families of compact complex manifolds diffeomorphic to $M$ (in $\D$) together with a functor that sends a family to its base. From the other hand, it has an atlas (which is not unique), that is an analytic space with a smooth and surjective mapping to the Teichm\"uller stack. Roughly speaking, the Teichm\"uller stack appears thus as a quotient of the atlas, say $T$. The morphism from the atlas $T$ to the Teichm\"uller stack $\T$ is given by the choice of a family of compact complex manifolds above $T$, thanks to Yoneda's Lemma. It is smooth and surjective if the projections $T\times_{\T}T\to T$ are smooth and surjective morphisms between analytic spaces. This allows to form the analytic groupoid $T\times_{\T}T\rightrightarrows T$ that contains all the information needed to reconstruct the Teichm\"uller stack as a category of families. Indeed, starting with $T$, one recovers a stack isomorphic to $\T$ through a process called stackification. And any other choice of an atlas of $\T$ gives an analytic groupoid that is Morita equivalent to that associated to $T$.

Let us begin with the categorical viewpoint.
Let $\mathfrak S$ be the category of analytic spaces and morphisms endowed with the euclidian topology. Given $S\in \mathfrak S$, we call {\it $M$-deformation over $S$} a proper and smooth morphism $\mathcal X\to S$ whose fibers are compact complex manifolds diffeomorphic to $M$. As $C^\infty$-object, such a deformation is a bundle over $S$ with fiber $M$ and structural group $\text{Diff}^+(M)$ (diffeomorphisms of $M$ that preserve its orientation). It is called {\it reduced} if the structural group is reduced to $\text{Diff}^0(M)$. In the same way, a morphism of reduced $M$-deformations $\mathcal X$ and $\mathcal X'$ over an analytic morphism $f: S\to S'$ is a cartesian diagram
$$
\begin{tikzcd}
\mathcal X \arrow[d]\arrow[dr, phantom, "{\scriptscriptstyle \square}", very near start]\arrow[r]&\mathcal X'\arrow[d]\\
S \arrow[r,"f"]&S'
\end{tikzcd}
$$
such that $\mathcal X$ and $f^*\mathcal X'$ are isomorphic as $\text{Diff}^0(M)$-bundles over $S$.

The Teichm\"uller stack $\T$ is the stack over the site $\mathfrak{S}$ whose objects are reduced $M$-deformations and morphisms are morphisms of reduced $M$-deformations. The natural morphism $\T\to\mathfrak{S}$ sends a reduced $M$-deformation onto its base and a morphism of reduced $M$-deformation to the corresponding morphism between their bases. Alternatively, $\T$ can be seen as a $2$-functor from $\mathfrak S$ to the category of groupoids such that
\begin{enumerate}
\item[i)] $\T(S)$ is the groupoid of isomorphism classes of reduced $M$-deformations over $S$.
\item[ii)] $\T(f)$ is the pull-back morphism $f^*$ from  $\mathscr{T}(M)(S')$ to $\mathscr{T}(M)(S)$.
\end{enumerate}

A point $X_0:=(M,J_0)$ is an object of $\mathscr{T}(M)(pt)$ that is a complex structure on $M$ up to biholomorphisms smoothly isotopic to the identity. 

\begin{remark}
	\label{important}
	In the definition of analytic stack used in \cite{LMStacks}, we did not impose that the diagonal is representable, see the discussion in \S 2.4 of \cite{LMStacks}. However, this is indeed true in full generality, hence the definition of an analytic stack as the stackification over $\mathfrak S$ of a smooth analytic groupoid given in \cite[\S 2.4]{LMStacks} is equivalent to the definition of an analytic stack as a stack over $\mathfrak S$ with representable diagonal, see \cite[\S 2.4]{Boivin}.
\end{remark}

Let us switch to the atlas point of view. Roughly speaking, $\T$ can be considered as an analytic version of the quotient $\mathcal I(M)/\text{Diff}^0(M)$. Here, $\mathcal I(M)$ is the set of integrable complex operators on $M$ compatible with its orientation (o.c.), that is
\begin{equation}
\label{I}
\mathcal I(M)=\{J\ :\ TM\longrightarrow TM\mid J^2\equiv -Id,\ J\text{ o.c.},\ \ [T^{1,0}, T^{1,0}]\subset T^{1,0}\}
\end{equation}
for
$$
T^{1,0}=\{v-iJv\mid v\in TX\}.
$$
and $\text{Diff}^0(M)$ is the group of diffeomorphisms of $M$ which are $C^\infty$-isotopic to the identity. It acts on the right on $\mathcal I(M)$ through
	\begin{equation}
	\label{actionDiff}
	J\cdot f:=df^{-1}\circ J\circ df
\end{equation}
Hence the mapping
\begin{equation}
	\label{ordrefiso}
	x\in X_{J\cdot f}\longmapsto f(x)\in X_J
\end{equation}
is an isomorphism (note the order: $f$ sends $J\cdot f$ to $J$).

In \cite{LMStacks}, a finite-dimensional atlas $T$ of (a connected component of)  $\T$ is described under the hypothesis that the dimension of the automorphism group of the complex manifolds encoded in $\T$ is bounded. Basically, it is given by a (at most countable) disjoint union of Kuranishi spaces\footnote{ Indeed, for technical reasons, a fatting process is used to ensure that all components of $T$ have the same dimension.}.  The construction of the Kuranishi space of a compact complex manifold is recalled in Section \ref{Kurfamily}. The morphism from $T$ to $\T$ is given by the choice of a family of complex manifolds above $T$ by Yoneda's Lemma. Here we take the Kuranishi families, see \S \ref{Kurfamily}. Then, to finish the construction, we need to compute the fiber product $T\times_{\T}T$ and check both projections to $T$ are smooth and surjective, forming in this way an analytic groupoid that encodes completely $\T$. This is the crux of \cite{LMStacks}. In this paper, we will not make use of this global atlas since we analyse the local models, hence we will skip its precise construction.

Two points have to be emphasized here. Firstly, as a stack, $\T$ also encodes the isotropy groups of the action. Recall that the isotropy group at $X_0$ is the group
\begin{equation}
\label{Aut1def}
\text{Aut}^1(X_0):=\text{Aut}(X_0)\cap \text{Diff}^0(M).
\end{equation}
which may be different from $\Azero$, the connected component of the identity of the automorphism group $\A$, see \cite{Aut1} and \cite{CatAut}. 
Secondly, 
 since $\D$ acts on the (infinite-dimensional) analytic space $\I$ preserving its connected components and its irreducible components, we may speak in this way of connected components and irreducible components of $\T$.  Indeed Kuranishi's Theorem tells us that the set $\mathcal I(M)$ is locally the product of a finite-dimensional local analytic section $K_0$ to the action with an infinite-dimensional manifold, cf. Section \ref{Douady}. Hence, locally, the irreducible components of $\I$ correspond to those of the finite-dimensional space $K_0$.

We finish this part with some terminology. By Teichm\"uller {\slshape space}, we mean the topological space obtained by endowing the quotient of \eqref{I} by \eqref{actionDiff} with the quotient topology. It is of course different from the Teichm\"uller {\slshape stack} $\T$, although related as follows. Given $T\to\T$ an atlas and $T_1\rightrightarrows T$ the associated groupoid, the Teichm\"uller space is homeomorphic to the quotient of $T$ by the the equivalence relation induced by $T_1$. We also introduce the notion of open, resp. Zariski open, resp. analytic substack of $T\to\T$, that plays an important role in the sequel. It is obtained as the stack induced from $T\to\T$ by an open, resp. Zariski open subset, resp. analytic subspace of $T$. In other words, an analytic stack $A\to\mathscr{A}$ is an open, resp. Zariski open, resp. analytic substack of $T\mapsto\T$ if there exists an open, resp. Zariski open subset $B$, resp. analytic subspace $B$ of $T$ such that $A \rightarrow \mathscr{A}$  is isomorphic to the analytic groupoid $B \times_{\mathscr{T}(M)} B \rightrightarrows B$.

\begin{remark}
	\label{rkconcomp}
	Such an analytic substack $A\to\mathscr{A}$ is said to be {\slshape proper} or {\slshape strict} if $B$ can be chosen so that its is an analytic subspace of positive codimension in the reduction of $T$.
	
	Note also that we admit analytic subspaces $B$ with a countable number of connected components, cf. Example \ref{exHopf3}.
\end{remark}

\section{The Kuranishi stacks}
\label{localstack}
In the first two subsections, we review the construction of the Kuranishi family first from classical deformation theory point of view, then from Kuranishi-Douady's point of view.

We then review the construction of the Kuranishi stack(s) introduced in \cite{LMStacks}. They play a fundamental role in the local theory. Especially we prove in Subsection \ref{Kurstack} that they enjoy the universal property the Kuranishi space does not fulfill. 

It is worth pointing out that the classical point of view (which presents Kuranishi family from a formal/algebraic point of view leaving aside the analytic details of the construction) is not enough for our purposes. This is indeed an infinitesimal point of view and even if it gives complete equations for the Kuranishi space, it fails in describing the properties of the structures close to the base complex structure. Kuranishi-Douady's point of view allows to pass from the infinitesimal point of view to a local one.  
\subsection{The Kuranishi family}
\label{Kurfamily}

The Kuranishi family $\pi : \mathscr{K}_0\to K_0$ of $X_0$ is a semi-universal deformation of $X_0$. It comes with a choice of a marking, that is of an isomorphism $i$ between $X_0$ and the fiber $\pi^{-1}(0)$ over the base point $0$ of $K_0$. The semi-universal property means that
\begin{enumerate}[i)]
\item Every marked deformation $\mathcal{X}\to B$ of $X_0$ is locally isomorphic to the pull-back of the Kuranishi family by a pointed holomorphic map $f$ defined in a small neighborhood of the base point of $B$ with values in a neighborhood of $0$ in $K_0$.
\item Neither the mapping $f$ nor its germ at the base point are unique; but its differential at the base point is.
\end{enumerate}

Two such semi-universal deformations of $X_0$ are isomorphic up to restriction to a smaller neighborhood of their base points. Hence the germ of deformation $(\mathscr{K}_0,\pi^{-1}(0))\to (K_0,0)$ is unique. This explains why we talk of the Kuranishi family, even if, in many cases, we work with a representative of the germ rather than with the germ itself.

The Zariski tangent space to the Kuranishi space $K_0$ at $0$ identifies naturally with $H^1(X_0,\Theta_0)$, the first cohomology group with values in the sheaf $\Theta_0$ of germs of holomorphic tangent vector fields of $X_0$. Indeed, $K_0$ is locally isomorphic to an analytic subspace of $H^1(X_0,\Theta_0)$ whose equations coincide at order 2 with the vanishing of the Schouten bracket.

The groups $\A$, $\Azero$ and $\Aun$ act on this tangent space.
However, this infinitesimal action cannot always be integrated in an action of the automorphism groups of $X_0$ onto $K_0$, see \cite{Doan}. Still there exists an action of each $1$-parameter subgroup and all these actions can be encoded in an analytic groupoid and thus in a stack. To do this, we need to know more about the complex properties of the structures encoded in a neighborhood of $0$ in $K_0$.

\subsection{Kuranishi-Douady's presentation}
\label{Douady}

Let $V$ be an open neighborhood of $J_0$ in $\I$. Complex structures close to $J_0$ can be encoded as $(0,1)$-forms $\omega$ with values in $T^{1,0}$ which satisfy the equation
\begin{equation}
\label{integ}
\bar\partial \omega+\dfrac{1}{2}[\omega,\omega]=0 
\end{equation}
Choose an hermitian metric and let $\bar\partial^*$ be the $L^2$-adjoint of $\bar\partial$ with respect to this metric. Let $U$ be a neighborhood of $0$ in the space of global smooth sections of $(T^{0,1})^*\otimes T^{1,0}$. Set
\begin{equation}
\label{K0}
K_0:=\{\omega\in U\mid \bar\partial \omega+\dfrac{1}{2}[\omega,\omega]=\bar\partial^*\omega=0\}
\end{equation} 
 Let $W$ an open neighborhood of $0$ in the vector space of vector fields $L^2$-orthogonal to the vector space of holomorphic vector fields $H^0(X_0,\Theta_0)$. In Douady's setting \cite{Douady}, Kuranishi's Theorem states the existence of a local isomorphism between $\I$ at $J_0$ and the product of $K_0$ with $W$ such that every plaque $\{pt\}\times W$ is sent through the inverse of this isomorphism into a single local $\D$-orbit. To be more precise, up to restricting $U$, $V$ and $W$, the Kuranishi mapping
\begin{equation}
\label{Kurmap}
(\xi,J)\in W\times K_0\longmapsto J\cdot e(\xi)\in V
\end{equation}
is an isomorphism of infinite-dimensional analytic spaces. As usual, we use the exponential map associated to the chosen metric in order to define the map $e$ which gives a local chart of $\D$ at $Id$. And $\cdot$ denotes the natural right action \eqref{actionDiff} of $\D$ onto $\I$.

\begin{remark}
	\label{rkproductchart}
	In the sequel, we will always work with an open set $V$ of $\I$ which is a product through \eqref{Kurmap}. Especially, in all statements, the expression ``Reducing $V$ if necessary'' or ``For $V$ small enough'' must be understood as taking a smaller $V$ but that still satisfies \eqref{Kurmap} for a smaller $K_0$.
\end{remark}

\subsection{Automorphisms and jumping points}
\label{secjumping}
Let us now analyze how automorphisms of $X_0$ are related to the local geography of the Teichm\"uller stack and determine a first obstruction for $\T$ to be locally isomorphic to an analytic space or to an orbifold.

We begin with a construction in a slightly more general setting. Let $f$ be an element of $\D$ such that $J_0\cdot f$ belongs to $V$, e.g. $f$ is an automorphism of $X_0$. Composing the inverse of \eqref{Kurmap} with the projection onto $K_0$ gives a retraction map $\Xi : V\to K_0$. Then, define
\begin{equation}
	\label{Uf}
	U_f=\{J\in K_0\mid J\cdot f\in V\}
\end{equation} 
This is an open set since $V$ is open. And it contains $J_0$. Now
\begin{equation}
\label{Autaction}
\text{Hol}_f\ :\ J\in U_f\subset K_0\longmapsto \Xi(J\cdot f)\in K_0
\end{equation}
is a well defined analytic map that must be thought of as the action of $f$ onto $K_0$.  Note however that composition does not work in general, that is $\text{Hol}_{f\circ g}$ may be different from $\text{Hol}_{g}\circ \text{Hol}_f$. In particular, there is no well defined action of $\Azero$ onto $K_0$, see \cite{Doan} for a counterexample.

Let $(F_t)_{t\in[0,1]}$ be a continuous path in $\D$ joining $f$ to the identity. Assume that $J_0\cdot F_t$ belongs to $V$ for all $t\in [0,1]$, e.g. $f\in\Azero$ and $(F_t)$ is any continuous path in $\Azero$ joining $f$ to the identity. We call {\slshape compatible} such a path. Then, $U_{F_t}$ and $\text{Hol}_{F_t}$ are well defined for all $t$ and all $U_{F_t}$ contain $J_0$.

\begin{lemma}
	\label{lemmaomegaf}
	The intersection $\cap_{t\in [0,1]}U_{F_t}$ contains an open neigborhood of $J_0$.
\end{lemma} 

\begin{proof}
	For every $t\in [0,1]$, choose some relatively compact open set $V_t\Subset U_{F_t}$. By continuity, there exists an open interval $I_t$ containing $t$ such that $t'\in I_t$ means $V_t\subset U_{F_{t'}}$. By compacity, there exist $t_1,\hdots,t_k$ in $[0,1]$ such that
	\begin{equation*}
		[0,1]=I_{t_1}\cup\hdots\cup I_{t_k}
	\end{equation*}
	Hence,
	\begin{equation*}
		\forall t\in [0,1],\qquad \cap_{1\leq i\leq k}V_{t_i}\subset U_{F_t}
	\end{equation*}
	and we are done, since the lefthand term is a non-empty open set containing $J_0$ as a finite intersection of such open sets.
\end{proof}
Thanks to Lemma \ref{lemmaomegaf}, we may define the open set
\begin{equation}
	\label{Omegaf}
	\Omega_f:=\bigcup_{(F_t)_{t\in [0,1]}}\text{Int}\Big (\bigcap_{t\in [0,1]}U_{F_t}\Big )
\end{equation}
where the union is taken on all compatible paths $(F_t)_{t\in[0,1]}$ in $\D$.

By definition, $J\in\Omega_f$ if and only if there exists a compatible path $(F_t)$
in $\D$ such that $J\cdot F_t\in V$ for all $t\in [0,1]$.

Define now the function
\begin{equation}
	\label{h0}
	t\in K_0\longmapsto h^0(t):=\dim (\text{Aut}(X_t))\in\N
\end{equation}
It is a well known fact that \eqref{h0} may not be constant - the best we can say is that it is upper semicontinuous for the Zariski topology of $K_0$, see \cite{Grauert} and footnote \ref{ftgrauert}. When it is constant and $K_0$ is reduced, a classical Theorem of Wavrik \cite{Wavrik} asserts the Kuranishi space is universal, i.e. $f$ is unique in the setting of Subsection \ref{Kurfamily}. But we also have

\begin{lemma}
	\label{lemmaHolWavrik}
	Assume that {\rm \eqref{h0}} is constant and $K_0$ is reduced. Then $\text{\rm Hol}_f$ is equal to the identity of $K_0$ for any $f\in\Azero$.
\end{lemma}

We will come back to this in Section \ref{localTeich2} when defining and studying normal points.

\begin{proof}
	Let $f\in\Azero$ and let $J\in\Omega_f$. We assume that $J$ is different from $J_0$ and we let $(F_t)_{t\in [0,1]}$ be a compatible path in $\Azero$ with $J\cdot F_t$ in $V$ for all $t$. By definition, $\text{Hol}_{F_t}(J)$ is well defined for all $t\in [0,1]$, drawing a continuous path in $K_0$ between $J$ and $\text{Hol}_f(J)$ all of whose points encode the same manifold $X_J$ up to biholomorphism $C^\infty$-isotopic to the identity. Assume this path is non-constant. Then we can find distinct points of $K_0$ encoding $X_J$ through a biholomorphism arbitrary close to the identity in $\D$-topology. Since \eqref{h0} is constant, this would contradict Theorem 1 of \cite{Kur3}. Hence the path has to be constant so that $\text{Hol}_f(J)=J$ for all $J\in\Omega_f$. Since $K_0$ is reduced, this is enough to conclude that $\text{Hol}_f$ is the identity on $\Omega_f$, thus on $U_f$ by analyticity. But $U_f$ must then be equal to $K_0$.
\end{proof}
In the general case, that is when \eqref{h0} is not constant, the line of arguments used in the proof of Lemma \ref{lemmaHolWavrik} can be expanded to show that $K_0$ acquires a stratified foliated structure as defined and analyzed in \cite[\S 3]{ENS}. Firstly one decomposes $K_0$ into strata $(K_0)_a$ with function \eqref{h0} bounded above by $a$. Then each difference $S_a:=(K_0)_a\setminus (K_0)_{a-1}$ admits a holomorphic foliation with non singular leaves. The highest $S_a$ contains at least the base point $J_0$ and its foliation is a foliation by points. The other $S_a$ admits positive-dimensional leaves. The leaves correspond to the connected components in $K_0$ of the following equivalence relation: $J\equiv J'$ if and only if both operators belong to the same $\D$-orbit.

In other words, when the function \eqref{h0} is not constant, some positive-dimensional connected submanifolds of $K_0$ encode the same complex manifold up to biholomorphism $C^\infty$-isomorphic to the identity. Hence the Teichm\"uller space is not homeomorphic to $K_0$. Indeed, it is the leaf space of this stratified foliated structure and is usually non-Hausdorff. We thus set

\begin{definition}
	\label{defjumpingpoint}
	We say that $X_0$ is a {\slshape jumping point} of $\T$ if \eqref{h0} is not locally constant.
\end{definition}

Jumping points are the first obstruction for $\T$ to be locally isomorphic to an analytic space or an orbifold and correspond to points where the Teichm\"uller space has very bad properties. However, as recalled above, from the one hand they form a strict analytic substack of $\T$; and from the other hand, they correspond to points with non trivial foliated structure. Describing them geometrically boils down to describing the foliated
structure of $K_0$. The first draft is done in \cite{ENS} but there is still much to do. One of the most challenging question is the following.

\begin{question}
	\label{questionseparatrix}
	When \eqref{h0} is non constant, does the foliated structure always admit a separatrix?
\end{question}

By separatrix, we mean a positive-dimensional leaf that contains $0$ in its closure. The existence of a separatrix at a jumping point implies that there exists a jumping family  based at that point, see Example \ref{exjumping} and Section \ref{secpathologies}.


\begin{remark}
	\label{rkleavesconn}
	Given an automorphism $f$ of $X_0$, observe that $\text{Hol}_f$ respects the foliation  of $K_0$. Moreover,
	for $f$ in $\Azero$ close to the identity, then $\text{Hol}_f$ fixes each leaf of the foliation of $K_0$. However, a general element of $\Azero$ may send a leaf to a different leaf. This is due to the fact that the restriction to $V$ may disconnect the $\Azero$-orbits in $\mathcal{I}(M)$.
%
\end{remark}

\subsection{The Kuranishi stacks}
\label{Kurstack}
 The Kuranishi stacks encode the maps \eqref{Autaction} in an analytic groupoid. The first step to do this consists in proving that there is an isomorphism
 \begin{equation}
 \label{Gdecompoglobal}
 (\xi,g)\in W\times \text{\rm Aut}^0(X_0)\longmapsto g\circ e(\xi)\in \mathcal D_0
 \end{equation}
 with values in a neighborhood $\mathcal D_0$ of $\Azero$ in $\D$, see \cite[Lemma 4.2]{LMStacks}. 
 
 Let now $\text{Diff}^0 (M,\mathscr K_0)$ denote the set of $C^\infty$ diffeomorphisms from $M$ to a fiber of the Kuranishi family $\mathscr K_0\to K_0$. This is an infinite-dimensional analytic space\footnote{ Strictly speaking, we have to pass to Sobolev $L^2_l$-structures for a big $l$ to have an analytic space, and $\text{Diff}^0 (M,\mathscr K_0)$ is the subset of $C^\infty$ points of this analytic set. In the sequel, we automatically make this slight abuse of terminology, cf. Convention 3.2 in \cite{LMStacks}.}, see \cite{Douady}. Here by $(J,F)\in \text{Diff}^0 (M,\mathscr K_0)$, we mean that we consider $F$ as a diffeomorphism from $M$ to the complex manifold $X_J$.
 
 \begin{definition}
 	\label{defadm}
 	Given $(J,F)$ an element of $\text{Diff}^0 (M,\mathscr K_0)$, we say it is $(V,\mathcal D_0)$-admissible if there exists a finite sequence $(J_i,F_i)$ (for $0\leq i\leq p$) of $\text{Diff}^0 (M,\mathscr K_0)$ such that
 	\begin{enumerate}[i)]
 		\item $J_0=J$ and $J_{i+1}=J_i\cdot F_i$ is in $K_0$ for all $0\leq i\leq p$, adding the convention $J_{p+1}:=J\cdot F$.
 		\item $F=F_0\circ\hdots\circ F_p$.
 		\item Each $F_i$ belongs to $\mathcal D_0$ as well as $F_i^{-1}$.
 		\item \label{additional} $J_i$ belongs to $\Omega_{F_i}$ for all $0\leq i\leq p$ and to $\Omega_{F_{i-1}^{-1}}$ for all $1\leq i\leq p+1$.
 	\end{enumerate}	
 \end{definition}
 
\begin{remark}
	\label{rkadmissible}
	It could seem more natural to speak of $(K_0,\mathcal D_0)$-admissible, but, by Remark \ref{rkproductchart}, changing $V$ is equivalent to changing $K_0$ when $\mathcal D_0$ is fixed. We keep this terminology to be coherent with that of \cite{LMStacks}. It should be pointed however that the previous definition is a bit more restrictive than that of \cite{LMStacks}. We shall see that the additional point \ref{additional}) plays a crucial role in the proof of Lemma \ref{lemma1main}. 
\end{remark}
Notice that, given $(J,F)$ and $(J\cdot F,F')$ both $(V,\mathcal D_0)$-admissible, then $(J,F\circ F')$ is also $(V,\mathcal D_0)$-admissible, as well as $(J\cdot F, F^{-1})$. We set then
 \begin{equation}
 \label{Azero}
 \mathcal A_0=\{(J,F)\in \text{Diff}^0 (M,\mathscr K_0)\mid (J,F)\text{ is }(V,\mathcal D_0)\text{-admissible}\}
 \end{equation}
This set encodes identifications between structures in $K_0$ that are given by composing diffeomorphisms in the neighborhood $\mathcal D_0$.\\
We also consider the two maps from $\mathcal A_0$ to $K_0$
 \begin{equation}
 \label{st}
 s(J,F)=J\qquad\text{ and }\qquad t(J,F)=J\cdot F
 \end{equation}
 and the composition and inverse maps
 \begin{equation}
 \label{multgroupoid}
 m((J,F),(J\cdot F, F'))=(J,F\circ F'),\qquad i(J,F)=(J\cdot F, F^{-1})
 \end{equation} 
 With these structure maps, the groupoid $\mathcal A_0\rightrightarrows K_0$ is an analytic groupoid \cite[Prop. 4.6]{LMStacks} whose stackification over $\mathfrak{S}$ is called {\slshape the Kuranishi stack} of $X_0$. We denote it by $\mathscr A_0$. Note that it depends indeed of the particular choice of $V$.
 
 As a category, its objects are still reduced $M$-deformations over bases belonging to $\mathfrak S$. However, the allowed complex structures are those encoded in $V$; and the allowed families are those obtained by gluing pull-back families of $\mathscr K_0\to K_0$ with respect to $(V,\mathcal D_0)$-admissible diffeomorphisms. In the same way, morphisms are those induced by $(V,\mathcal D_0)$-admissible diffeomorphisms. Hence, not only the complex fibers of the families have to be isomorphic to those of $\mathscr K_0$, but gluings and morphisms of families are restricted.
 
  Of course, the same construction can be carried out for the automorphism groups $\Aun$, resp. $\A$, with the following modifications. In \eqref{Gdecompoglobal}, $\Azero$ is replaced with $\Aun$, resp. $\A$, defining a neighborhood $\mathcal D_1$ of $\Aun$ in $\D$, resp. $\mathcal D$ of $\A$ in $\text{Diff}^+(M)$. This allows to speak of $(V,\mathcal D_1)$-admissible, resp. $(V,\mathcal D)$-admissible diffeomorphisms. But in the $\mathcal{A}$ and $\mathcal{A}_1$ cases, the sets $\Omega_{F_i}$, resp. $\Omega_{F_{i-1}^{-1}}$, in point \ref{additional} must be replaced with $U_{F_i}$, resp. $U_{F_{i-1}^{-1}}$ in the definition of admissibility, since $\A$ and $\Aun$ may contain elements that are not connected to the identity. Then, replacing $\mathcal D_0$ with $\mathcal D_1$, resp. $\mathcal D$ in \eqref{Azero} we obtain the analytic groupoid $\mathcal A_1\rightrightarrows K_0$, resp. $\mathcal A\rightrightarrows K_0$. Its stackification over $\mathfrak S$ gives a stack $\mathscr A_1$, resp. $\mathscr{A}$. The previous description of $\mathscr A_0$ as a category applies to $\mathscr{A}_1$, resp. $\mathscr A$ with the obvious changes. We also call them {\slshape Kuranishi stacks}.
 
Before analyzing more thoroughly these Kuranishi stacks, we would like to say a little more about automorphisms and $K_0$.	
Given $f\in \A$, define $\text{Hol}_f$ as in \eqref{Autaction} and set
\begin{equation}
	\label{sigmaf}
	\sigma_f : J\in U_f\subset K_0\longmapsto (J,f\circ e(\chi(J))\in\mathcal{A}
\end{equation}
where $\chi$ is an analytic mapping from $U_f\subset K_0$ to $W$ with $\chi(0)=0$ defined by
\begin{equation}
	\label{chichecks}
	\text{Hol}_f(J)=\Xi(J\cdot f)=(J\cdot f)\cdot e(\chi(J))
\end{equation}
The map $\sigma_f$ is a local analytic section of the source map $s : \mathcal{A}\to K_0$ defined on $U_f$. It satisfies
\begin{equation}
	\label{tsigma}
	t\circ\sigma_f=\text{Hol}_f
\end{equation}
Moreover, when $f$ is an element of $\Azero$, 
it has values in $\mathcal A_0$, 
when restricted to $\Omega_f$.

Finally, we prove that $\mathscr{A}_0$ contains the connected component of the identity of the automorphism group of every fiber.

\begin{lemma}
	\label{lemmaautcontained}
	Given any $J\in K_0$ and any $f\in\text{\rm Aut}^0(X_J)$, then $(J,f)$ belongs to $\mathcal{A}_0$.
\end{lemma}

\begin{proof}
	Let $J\in K_0$ and let $f\in\text{Aut}^0(X_J)$. Assume that $f$ is sufficiently close to the identity to belong to $\mathcal{D}_0$. Then, using \eqref{Gdecompoglobal}, we find $g\in\Azero$ such that $J$ belongs to $U_g$ and
	\begin{equation*}
		(J,f)=\sigma_g(J)
	\end{equation*} 
	This proves that $(J,f)$ belongs to $\mathcal{A}_0$ as soon as $f$ is small enough. Since any element in $\text{Aut}^0(X_J)$ is a finite composition of small elements, this is still true for any $f\in\text{Aut}^0(X_J)$.
\end{proof}
\noindent Connexity is crucial here. Lemma \ref{lemmaautcontained} does not hold true for $\text{Aut}^1(X_J)$.
\subsection{Universality of the Kuranishi stacks}
\label{universal}
Recall that Kuranishi's Theorem asserts the existence of a semi-universal deformation for any compact complex manifold. This is however not a universal deformation when the dimension of the automorphism group varies in the fibers of the Kuranishi family, i.e. in the setting of section \ref{Kurfamily}, the germ of mapping $f$ is not unique. Replacing the Kuranishi space with the Kuranishi stack allows to recover a universality property. 

To do that, we need to germify the Kuranishi stacks. We replace our base category $\mathfrak{S}$ with the base category $\mathfrak{G}$ of germs of analytic spaces. We turn $\mathfrak G$ into a site by considering the trivial coverings. Hence each object of $\mathfrak G$ has a unique covering and there is no non trivial descent data.

We then germify the groupoids. Starting with $\mathcal A\rightrightarrows K_0$, resp. $\mathcal A_0\rightrightarrows K_0$ and $\mathcal A_1\rightrightarrows K_0$, and using $s$ and $t$ as defined in \eqref{st}, we germify $K_0$ at $0$, $\mathcal{A}$, resp. $\mathcal{A}_0$ and $\mathcal{A}_1$, at the fiber $(s\times t)^{-1}(0)$ and germify consequently all the structure maps. We thus obtain the groupoids $(\mathcal A,(s\times t)^{-1}(0))\rightrightarrows (K_0,0)$, resp. $(\mathcal A_0,(s\times t)^{-1}(0))\rightrightarrows (K_0,0)$ and $(\mathcal A_1,(s\times t)^{-1}(0))\rightrightarrows (K_0,0)$.

Finally, we stackify $(\mathcal A,(s\times t)^{-1}(0))\rightrightarrows (K_0,0)$, resp. $(\mathcal A_0,(s\times t)^{-1}(0))\rightrightarrows (K_0,0)$ and $(\mathcal A_1,(s\times t)^{-1}(0))\rightrightarrows (K_0,0)$, over $\mathfrak{G}$. We denote the corresponding stacks by $(\mathscr{A},0)$, resp. $(\mathscr{A}_0,0)$ and $(\mathscr{A}_1,0)$.

The objects of $(\mathscr{A},0)$ over a germ of analytic space $(S,0)$ are germs of $M$-deformations $p : \mathcal X\to S$ with fiber at the point $0$ of $S$ isomorphic to $X_0$. We denote them by $(\mathcal X,p^{-1}(0))\to (S,0)$. The morphisms over some analytic mapping $f: S\to S'$ are germs of morphisms between $M$ deformations $(\mathcal X,p^{-1}(0))\to (S,0)$ and $(\mathcal X',{p'}^{-1}(0))\to (S',0')$ over $f$. Note that $f(0)=0'$.

\begin{remark}
	\label{rkimportant}
	It is crucial to notice that we deal with germs of {\slshape unmarked} deformations. There is obviously a distinguished point (since we deal with germs), but there is no marking of the distinguished fiber.
\end{remark}

The following theorem shows that $(\mathscr{A},0)$ contains indeed {\slshape all} such germs of $M$-deformations and of morphisms between $M$-deformations. It is folklore although we never saw a paper stating this in a precise way.

\begin{theorem}
	\label{thmuniversal}
	The stack $(\mathscr{A},0)$ is the stack $\mathscr{M}$ over $\mathfrak G$ whose objects are the germs of $M$-deformations of $X_0$ and whose morphisms are the germs of morphisms between $M$-deformations.
\end{theorem}

\begin{proof}
	Since the site $\mathfrak{G}$ does not contain any non-trivial covering, there is no gluings of families, and the torsors associated to $(\mathscr{A},0)$ are just given by the pull-backs of the germ of Kuranishi family $(\mathscr{K}_0,\pi^{-1}(0))\to (K_0,0)$. Kuranishi's Theorem implies that the natural inclusion of $(\mathscr{A},0)$ in the stack $\mathscr{M}$ is essentially surjective.
	
	Morphisms over the identity of some germ $(S,0)$ of analytic space are thus given by morphisms $F$ between two germs of families $(f^*\mathscr{K}_0,\pi^{-1}(0))\to (B,0)$ and $(g^*\mathscr{K}_0,\pi^{-1}(0))\to (B,0)$ for $f$ and $g$ germs of analytic mappings from $(B,0)$ to $(K_0,0)$. Hence $F$ restricted to the central fiber $X_0\simeq \pi^{-1}(0)$ is an automorphism of the central fiber that is an element of $\A$. But $(s\times t)^{-1}(0)$ is isomorphic to $\A$ so such a morphism $F$ is induced by an analytic mapping from $(B,0)$ to $(\mathcal A,(s\times t)^{-1}(0))$ that we still denote by $F$ which satisfies $s\circ F=f$ and $t\circ F=g$. This shows that the natural inclusion of $(\mathscr{A},0)$ in the stack $\mathscr{M}$ is fully faithful. 
\end{proof}

This must be thought of as the good property of universality. Indeed, the failure of universality in Kuranishi's theorem comes from the existence of automorphisms of the Kuranishi family fixing the central fiber but not all the fibers. Imposing a marking is an artificial and incomplete solution to this problem because it only kills automorphisms inducing a non-trivial automorphism on the central fiber. Now, the stack $(\mathscr{A},0)$ is universal for germs of $M$-deformations of $X_0$, because, thanks to Theorem \ref{universal}, any such germ $(\mathcal X,X_0)\to (B,0)$ is induced by an analytic map from its base to $(\mathscr{A},0)$, yielding a diagram
\begin{equation}
	\label{CDpullbackgerm}
	\begin{tikzcd}
		(\mathcal X,X_0)\arrow[r,"F"]\arrow[d] &(\mathscr{K}_0,\pi^{-1}(0))\arrow[d,"\pi"]\\
		(B,0)\arrow[r,"f"'] &(K_0,0)
	\end{tikzcd}
\end{equation}
 Moreover the full map $(f,F)$ is unique up to unique isomorphism of family encoded in $(\mathscr{A},0)$. So thinking of $(\mathscr{A},0)$ as the quotient of the Kuranishi space by the automorphisms of the Kuranishi family, not only the map $f$ but also the full map $(f,F)$ is unique; and this occurs with no extra condition. 

In the same way, we have 
\begin{corollary}
	\label{thmuniversal1}
	The stack $(\mathscr{A}_1,0)$ is the stack over $\mathfrak G$ whose objects are the germs of reduced $M$-deformations and whose morphisms are the germs of morphisms between reduced $M$-deformations.
\end{corollary}
Here $C^\infty$-markings of the $M$-deformations, that is the choice of a $C^\infty$ diffeomorphism from $M$ to the central fiber, can be used to characterize reduced families. Morphisms are required to induce on $M$ a diffeomorphism isotopic to the identity through the markings. Here again, we may rephrase this Corollary as: the stack $(\mathscr{A}_1,0)$ is universal for germs of reduced $M$-deformations of $X_0$.

We also have

\begin{corollary}
	\label{thmuniversal0}
	The stack $(\mathscr{A}_0,0)$ is the stack over $\mathfrak G$ whose objects are the germs of $0$-reduced $M$-deformations and whose morphisms are the germs of morphisms between $0$-reduced $M$-deformations.
\end{corollary}
In other words, the stack $(\mathscr{A}_0,0)$ is universal for germs of $0$-reduced $M$-deformations of $X_0$.
A $0$-reduced $M$-deformation is just a marked family. We use a different terminology because morphisms are different. A morphism of marked families is required to induce on $X_0$ the identity through the markings, whereas a morphism of $0$-reduced $M$-deformation is required to induce on $X_0$ an element of $\Azero$ through the markings.

With this difference in mind, it is interesting to compare Corollary \ref{thmuniversal0} with the already cited classical statement of universality of \cite{Wavrik}, see also \cite{Kur3}. When the dimension of the automorphism group is constant on the fibers of the Kuranishi family and the Kuranishi space is reduced, every element of $\Azero$ extends as an automorphism of the Kuranishi family inducing the identity on $K_0$. Imposing a marking of the families prevents from reparametrizing with an element of $\A/\Azero$, yielding unicity of the pull-back morphism $f$ and universality in the classical sense. However, universality in the stack sense of Corollary \ref{thmuniversal0} is
\begin{enumerate}[i)]
	\item more general because it does not need extra hypotheses.
	\item more natural because the good condition to impose on the central fiber is to authorize reparametrizations by an element of $\Azero$, and not to prevent any reparametrization as the classical marking does.
	\item more precise because it gives unicity of the full mapping $(f,F)$ of \eqref{CDpullbackgerm}, that is it keeps track of the automorphisms of the Kuranishi family, even if they induce the identity on the base. 
\end{enumerate}

\section{The $\Z$-Teichm\"uller stack}
\label{secZteich}
Motivated by \cite{CatAut}, we introduce now the $\Z$-Teichm\"uller stack as a new stack intermediary between the Teichm\"uller stack and the moduli stack. It will play an important role when analyzing exceptional points. All that has been said before on the Teichm\"uller stack can be easily adapted to the $\Z$-Teichm\"uller stack.

\subsection{Definition and basic facts}
\label{subsecZdef}

Let $\DZ$ be the subgroup of $\text{Diff}^+(M)$ of diffeomorphisms that induce the identity on the singular cohomology groups $H^*(M,\Z)$. Recall that the $C^\infty$-type of a $M$-deformation, resp. a reduced $M$-deformation, is a bundle over some base $S$ with fiber $M$ and structural group $\text{Diff}^+(M)$, resp. $\D$. In the same way, a $M$-deformation is called {\it $\Z$-reduced} if the structural group is reduced to $\DZ$. And a morphism of $\Z$-reduced $M$-deformations $\mathcal X$ and $\mathcal X'$ over an analytic morphism $f: S\to S'$ is a cartesian diagram
$$
\begin{tikzcd}
	\mathcal X \arrow[d]\arrow[dr, phantom, "{\scriptscriptstyle \square}", very near start]\arrow[r]&\mathcal X'\arrow[d]\\
	S \arrow[r,"f"]&S'
\end{tikzcd}
$$
such that $\mathcal X$ and $f^*\mathcal X'$ are isomorphic as $\DZ$-bundles over $S$.

The $\Z$-Teichm\"uller stack $\TZ$ is the stack over the site $\mathfrak{S}$ whose objects are $\Z$-reduced $M$-deformations and morphisms are morphisms of $\Z$-reduced $M$-deformations. The natural morphism $\TZ\to\mathfrak{S}$ sends a $\Z$-reduced $M$-deformation onto its base and a morphism of $\Z$-reduced $M$-deformation to the corresponding morphism between their bases. By a direct adaptation of \cite{LMStacks}, it is an analytic stack under the hypothesis that the $h^0$-function is bounded on the set $\I$. Indeed, the analytic atlas $T$ of the Teichm\"uller stack constructed in \cite{LMStacks} under the same hypothesis is also an analytic atlas of $\TZ$. The difference between the two cases occur when computing the fiber product $T\times_{\TZ}T$ but the projections are still smooth and surjective.
A point $X_0:=(M,J_0)$ is an object of $\TZ(pt)$ that is a complex structure on $M$ up to biholomorphisms inducing the identity in cohomology with $\Z$-coefficients\footnote{ Of course, there are natural variants of the $\Z$-Teichm\"uller stack by considering other cohomologies or homologies such as singular cohomology with coefficients in $\mathbb{Q}$.}. 

From the natural inclusions 
$\D\subset \DZ\subset \text{Diff}^+(M)$,
we deduce the natural inclusions
\begin{equation}
	\label{eqTeichinclusions}
	\T\hookrightarrow\TZ\hookrightarrow\mathscr{M}(M)
\end{equation}
meaning that an object, resp. a morphism of $\T$ is also an object, resp. a morphism of $\TZ$ and that an object, resp. a morphism of $\TZ$ is also an object, resp. a morphism of $\mathscr{M}(M)$. The isotropy group of $X_0$ as a point of $\TZ$ is the group
\begin{equation}
	\label{AutZdef}
	\text{Aut}^\Z(X_0):=\text{Aut}(X_0)\cap \DZ.
\end{equation}
which contains both $\Azero$ and $\Aun$ and is contained in $\A$. Note that all inclusions may be strict \cite{CatAut}, showing in particular that the first inclusion map of \eqref{eqTeichinclusions} may also be strict.

\subsection{The $\Z$-Kuranishi stack}
\label{subsecZkur}
 The construction of \S \ref{Kurstack} can be carried out for the automorphism groups $\AZ$ with the following modifications. In \eqref{Gdecompoglobal}, $\Azero$ is replaced with $\AZ$ defining a neighborhood $\mathcal D_\Z$ of $\AZ$ in $\DZ$. This allows to speak of $(V,\mathcal D_\Z)$-admissible diffeomorphisms. The sets $\Omega_{F_i}$, resp. $\Omega_{F_{i-1}^{-1}}$, in point \ref{additional} must be replaced with $U_{F_i}$, resp. $U_{F_{i-1}^{-1}}$ in the definition of admissibility, since $\AZ$ may contain elements that are not connected to the identity. Then, replacing $\mathcal D_0$ with $\mathcal D_\Z$ in \eqref{Azero} we obtain the analytic groupoid $\mathcal A_\Z\rightrightarrows K_0$. Its stackification over $\mathfrak S$ gives a stack $\mathscr A_\Z$ that we call the {\slshape $\Z$-Kuranishi stack} of $X_0$.
 
 We may now deduce another Corollary to Theorem \ref{thmuniversal} that refers to $\Z$-reduced $M$-deformations. As in \S \ref{secuniversal}, we germify the analytic groupoid $\mathcal A_\Z\rightrightarrows K_0$ and denote its stackification over $\mathfrak G$ by $(\mathscr{A}_\Z,0)$. We then have

\begin{corollary}
	\label{thmuniversalZ}
	The stack $(\mathscr{A}_\Z,0)$ is the stack over $\mathfrak G$ whose objects are the germs of $\Z$-reduced $M$-deformations and whose morphisms are the germs of morphisms between $\Z$-reduced $M$-deformations.
\end{corollary}
Here $C^\infty$-markings of the $M$-deformations, that is the choice of a $C^\infty$ diffeomorphism from $M$ to the central fiber, can be used to characterize $\Z$-reduced families. Morphisms are required to induce on $M$ a diffeomorphism inducing the identity in cohomology through the markings. Here again, we may rephrase this Corollary as: the stack $(\mathscr{A}_\Z,0)$ is universal for germs of $\Z$-reduced $M$-deformations of $X_0$.

\section{Local structure of the Teichm\"uller stacks}
\label{localTeich}
A neighborhood of $X_0$ in $\mathscr{T}(M)$ consists of $M$-deformations all of whose fibers are close to $X_0$, that is can be encoded by structures $J$ living in a neighborhood $V$ of $J_0$ in $\mathcal I(M)$. As in \cite{LMStacks}, we shall denote it by $\mathscr{T}(M, V)$. The corresponding neighborhood of $\TZ$ is denoted by $\mathscr{T}^\Z(M, V)$. From now on, we assume that $V$ is open, connected and small enough to come equipped with a Kuranishi mapping \eqref{Kurmap}.

\subsection{Atlas}
\label{atlas}
The main difficulty to construct an atlas in \cite{LMStacks} was to describe all the morphisms between the different Kuranishi spaces involved to compute the fiber product. Here, in the local case, we just need to use one Kuranishi space and family as atlas and it is straightforward to give the associated groupoid for $\mathscr{T}(M, V)$. Just consider
\begin{equation}
	\label{atlasneigh}
	\mathcal T_V:=\{(J,f)\in\text{Diff}^0(M,\mathscr K_0)\mid J\cdot f\in K_0\}
\end{equation}
and the groupoid $\mathcal T_V\rightrightarrows K_0$ with structure maps as in \eqref{st} and \eqref{multgroupoid}. And consider
\begin{equation}
	\label{Zatlasneigh}
	\mathcal T^\Z_V:=\{(J,f)\in\text{Diff}^\Z(M,\mathscr K_0)\mid J\cdot f\in K_0\}
\end{equation}
and $\mathcal T^\Z_V\rightrightarrows K_0$ for a neighborhood of $X_0$ in $\TZ$.

Observe that \eqref{atlasneigh}, resp. \eqref{Zatlasneigh} is very close to the groupoid $\mathcal{A}_1\rightrightarrows K_0$ of the Kuranishi stack $\mathscr{A}_1$, resp. of $\mathcal{A}_\Z\rightrightarrows K_0$. Indeed the points of $\mathscr{T}(M, V)$, resp. $\mathscr{T}^\Z(M, V)$ are exactly the same than those of $\mathscr{A}_1$, resp. $\mathscr{A}_\Z$, but $\mathscr{A}_1$, resp. $\mathscr{A}_\Z$, have less morphisms, hence also less descent data and thus less objects. To understand how to pass from $\mathscr{A}_1$ to $\mathscr{T}(M,V)$, resp. from $\mathscr{A}_\Z$ to $\mathscr{T}^\Z(M, V)$, we need to understand and encode the "missing" morphisms. 

\subsection{Target Germification}
\label{proj}
As in the case of the Kuranishi stacks, we would like to germify $ V$ and consider only complex structures belonging to the germ of some point $J_0$ in $ V$. This process is different from the germification process of section \ref{universal} which was about germifying the base category and thus the base of $M$-deformations. Here we still want to consider $M$-deformations over any analytic bases, but need to germify the set of possible fibers. Hence we need a target germification process, as opposed to the source germification process used in Section \ref{universal}. To avoid cumbersome notations and an unreasonable use of resp., we only describe the process for $\mathscr{T}(M,V)$ and let the reader add a $\Z$ at each step in the case of $\mathscr{T}^\Z(M, V)$.

To do that, we look at sequences of stacks $\mathscr{T}(M, V_n)$ for $( V_n)$ an inclusion decreasing sequence of neighborhoods of a fixed point $J_0$ with $ V_0= V$. Corresponding to a nesting sequence 
\begin{equation}
	\label{nestingseq}
	\hdots\subset  V_n\subset\hdots\subset  V\subset\mathscr{I}
\end{equation}
we obtain the sequence
\begin{equation}
	\label{nestingsstack}
	\begin{tikzcd}
	\hdots \arrow[r,hook]&\mathscr{T}(M, V_n)\arrow[r,hook]&\hdots\arrow[r,hook]&\mathscr{T}(M, V)
	\end{tikzcd}
\end{equation}
We consider sequences of $M$-deformations over the same base $(\mathcal{X}_n\to B)$ such that $\mathcal{X}_n$ is an object of $\mathscr{T}(M, V_n)$ for some decreasing sequence \eqref{nestingseq}. We identify two such sequences $({\mathcal{X}'}_n\to B)$ and $(\mathcal{X}_n\to B)$ if the families ${\mathcal{X}'}_n\to B$ and $\mathcal{X}_n\to B$ are isomorphic as objects of $\mathscr{T}(M, V)$ for every large $n$. 
Here are some examples of such sequences
\begin{enumerate}[i)]
	\item Start with a $M$-deformation $\mathcal{X}\to \mathbb{D}$ over the disk with central fiber isomorphic to $X_0$. Then consider the pull-back sequence $(\lambda_n^*\mathcal{X}\to \mathbb{D})$ where $(\lambda_n)$ is a sequence of homotheties with ratio decreasing from $1$ to $0$.
	\item Start with a fiber bundle $E\to B$ with fiber $X_0$ and structural group $\Aun$ and a $M$-deformation $\pi : \mathcal{X}\to B\times\mathbb{D}$ which coincides with the bundle $E$ over $B\times\{0\}$. Then pick up some sequence $(x_n)$ in the disk which converges to $0$. Then consider the sequence of families $(\pi^{-1}(B\times\{x_n\})\to B)$.
\end{enumerate}
Morphisms from $({\mathcal{X}'}_n\to B)$ to $(\mathcal{X}_n\to B)$ are sequences $(f_n)$ with $f_n$ a family morphism over $B$ from ${\mathcal{X}'}_n$ to $\mathcal{X}_n$ for every $n$. Once again, we identify two such sequences $(f_n)$ and $(g_n)$ if there exists some integer $k$ such that $f_n=g_n$ as morphisms of $\mathscr{T}(M, V)$ for $n\geq k$.
\vspace{5pt}\\
We call the resulting category the {\slshape target germification} of $\mathscr{T}(M)$ at $J_0$ and denote it by $(\mathscr{T}(M),J_0)$. Observe that this is not a stack but rather a projective limit of stacks.

\subsection{The natural morphism from Kur to Teich}
\label{secfiniteetale}
We want to analyse the structure of the analytic space $\mathcal T_V$ defined in \eqref{atlasneigh} and compare it with $\mathcal{A}_1$. 

We already observed in Section \ref{atlas} that there is a natural inclusion of groupoids of $\mathcal{A}_1$ into $\mathcal T_V$. It comes from the fact that $\mathcal T_V$ encodes every morphism between fibers of the Kuranishi family, whereas $\mathcal{A}_1$ encodes {\slshape some} morphisms between fibers of the Kuranishi family. This inclusion is just the description at the level of atlases of the natural inclusion of $\mathscr{A}_1$ into $\mathscr{T}(M,V)$: $\mathscr{A}_1$-objects, resp. $\mathscr{A}_1$-morphisms, inject in $\mathscr{T}(M,V)$-objects, resp. $\mathscr{T}(M,V)$-morphisms. So our final goal here is to give the structure of this inclusion.

There exists also a natural inclusion of $\mathscr{A}_0$ into $\mathscr{T}(M,V)$. We first relate the morphisms encoded in $\mathcal T_V$ to those encoded in $\mathcal{A}_0$. Set

\begin{definition}
	\label{defshom}
	Let $(J,f)\in \mathcal T_V$ and let $(J,g)\in \mathcal T_V$. Then they are {\slshape $s$-homotopic} if there exists a compatible path $(F_t)$ in $\D$ such that $t\mapsto (J,F_t)$ joins $(J,f)$ to $(J,g)$ in $\mathcal T_V$. 
\end{definition}

In other words, $(J,f)$ to $(J,g)$ are $s$-homotopic if they belong to the same connected component of the $s$-fiber of $\mathcal{T}_V\rightrightarrows K_0$ at $J$.

\begin{lemma}
	\label{lemma1main}
	Let $(J,f)\in \mathcal T_V$ and let $(J,g)\in \mathcal T_V$. Then these two elements are $s$-homotopic if and only if $(J\cdot f,f^{-1}\circ g)$ belongs to $\mathcal{A}_0$.
\end{lemma}

\begin{proof}
	Assume $(J\cdot f,f^{-1}\circ g)$ belongs to $\mathcal{A}_0$, that is $(J\cdot f,f^{-1}\circ g)$ is $(V, \mathcal{D}_0)$-admissible. Then we may decompose it as
	\begin{equation*}
		f^{-1}\circ g=h_1\circ h_2\circ\hdots\circ h_p
	\end{equation*}
	with each $h_i\in \mathcal{D}_0$; and
	\begin{equation*}
		J_{i+1}=J_i\cdot h_i\qquad\text{ for }\qquad i=1,\cdots, p-1
	\end{equation*}
	belongs to $K_0$ with $J_1:=J\cdot f$. 
	We claim that $(J_p,Id)$ and $(J_p,h_p)$
	stay in the same connected component of $\mathcal{T}_V$.
	Indeed, it follows from point v) of Definition \ref{defadm} that $J_p$ belongs to $\Omega_{h_p}$. Hence there exists a compatible path $(H_t)$ joining $h_p$ to the identity and $(J_p, H_t)$ is a continuous path in $\mathcal{T}_V$ joining $(J_p,Id)$ and $(J_p,h_p)$. Since this path has fixed first coordinate, we may compose on the left by $f\circ h_1\circ h_2\circ\hdots\circ h_{p-1}$ and obtain a continuous path between $(J,g\circ h_{p-1})$ and $(J,g)$. Repeating the process, we connect $(J,f)$ to $(J,g)$.
	
	

	Conversely, let $(J,f)$ and $(J,g)$ be $s$-homotopic. Then, there exists an isotopy $(J,f_t)$ joining these two points in $\mathcal{T}_V$. But then we may find by compacity $t_0=0<t_1<\hdots <t_k=1$ such that
	\begin{equation*}
		(J_i,h_i):=(J\cdot f_{t_i},f_{t_i}^{-1}\circ f_{t_{i+1}})
	\end{equation*}
	satisfies that $h_i$ and $h_i^{-1}$ are sufficiently small to belong to an open neighborhood of the identity in $\mathcal{D}_0$ that maps every $J\cdot f_{t_j}$ and $J_0$ inside $V$. As a consequence, we have $J\cdot f_{t_i}$ in $\Omega_{h_{i}}$ and in $\Omega_{h_{i-1}^{-1}}$. Hence, 
	\begin{equation*}
		(J_0, h_0\circ\hdots\circ h_k)=(J\cdot f, f^{-1}\circ g)
	\end{equation*}
	is $(V,\mathcal{D}_0)$-admissible as needed.
\end{proof}

We are now in position to state and prove our first main result.

\begin{theorem}
	\label{finitethm}
	The natural inclusion of $\mathscr{A}_0$ into $\mathscr{T}(M,V)$, resp. of $\mathscr{A}_1$ into $\mathscr{T}(M,V)$, is an \'etale morphism of analytic stacks.

\end{theorem}

Let us make a few comments before proving Theorem \ref{finitethm}. First of all, the statement may be a bit misleading for readers used to the classical notion of \'etale morphism of analytic space. Given a discrete group $G$ acting holomorphically onto an analytic space $X$, then the morphism $X\to [X/G]$, with $[X/G]$ the quotient stack, is \'etale even if it has dense orbits or infinite stabilizers.

Then, by \'etale morphism of analytic stacks, we mean that, given any $B\in \mathfrak{S}$ and any morphism $u$ from $B$ to $\mathscr{T}(M,V)$, the fiber product
\begin{equation}
	\label{fb}
	\begin{tikzcd}
		B\times_{u}\mathscr A_0 \arrow[d,"f_1"']\arrow[r,"f_2"]\arrow[dr,phantom,"\scriptstyle{\square}",very near start, shift right=0.5ex]&\mathscr{A}_0\arrow[d,"\text{inclusion}"]\\
		B\arrow[r,"u"]&\mathscr{T}(M,V)
	\end{tikzcd}
	\quad\text{ resp.}
	\begin{tikzcd}
		B\times_{u}\mathscr A_1 \arrow[d,"f_1"']\arrow[r,"f_2"]\arrow[dr,phantom,"\scriptstyle{\square}",very near start, shift right=0.5ex]&\mathscr{A}_1\arrow[d,"\text{inclusion}"]\\
		B\arrow[r,"u"]&\mathscr{T}(M,V)
	\end{tikzcd}
\end{equation}
satisfies
\begin{enumerate}[i)]
	\item $B\times_{u}\mathscr A_0$, resp. $B\times_{u}\mathscr A_1$, is a $\C$-analytic space.
	\item The morphism $f_1$ is an \'etale morphism between  $\C$-analytic spaces.
\end{enumerate}

In other words, point i) means that the natural inclusion is a representable morphism. Hence it may enjoy any property preserved by arbitrary base change and local at target that a classical morphisms between analytic spaces may enjoy. Then point ii) means that, amongst all these properties, we prove that the natural inclusion is \'etale. We note that this corresponds to the "strong" notion of \'etale in the literature on algebraic stacks, e.g. in the local structure theorem of \cite{AHR}, the constructed étale morphism is not representable in general so a weaker notion of étale morphisms of algebraic stacks is used.

Last but not least, Theorem \ref{finitethm} must be understood geometrically as follows. A family $\mathcal X\to B$ with all fibers belonging to $V$ can be decomposed as local pull-backs of $\mathscr{K}_0$ glued together through a cocycle of morphisms $(u_{ij})$ in $\mathcal T_V$. It is an object of $\mathscr{A}_1$ if and only we may find an equivalent cocycle living in $\mathcal A_1$. This is completely similar to the process of reduction of the structural group of a fiber bundle. Theorem \ref{finitethm} says that, given such a family, there exists at most a discrete set of non-equivalent reductions. Assume $X_0$ is rigid with Kuranishi space being a reduced point. Then, any reduced $M$-family is indeed a locally trivial holomorphic bundle with fiber $X_0$ and structural group $\Aun$. There is no difference with families that are objects of $\mathscr{A}_1$, however objects of $\mathscr{A}_0$ are bundles with fiber $X_0$ and structural group $\Azero$ this time. So, in this particular case, Theorem \ref{finitethm} really describes the set of non-equivalent reductions of the structural group of such a bundle from $\Aun$ to $\Azero$. And this set can be easily determined by passing to the associated principal bundles and making use of the following observation. Given a principal $\Aun$-bundle $E$ over some base $B$, let $\Azero$ act on the fibers of $E$. The quotient $E'$ has fibers $\Aun/\Azero$ and is trivializable if and only if $E$ admits a $\Azero$-reduction. Hence the set we are looking for is the set of trivializations of $E'$ and identifies with the set of holomorphic maps from $B$ to the discrete $\Aun/\Azero$ set. The proof given below in the general case follows the same strategy.

We give another geometric interpretation of Theorem \ref{finitethm} in Section \ref{secuniversal}.

\begin{proof}
	By Yoneda's lemma, a morphism $u : B\to\mathscr{T}(M,V)$ corresponds to a family $\mathcal X\to B$. The fiber product $B\times_{u}\mathscr A_0$ encodes the isomorphisms of $\mathscr{T}(M,V)$ 
	\begin{equation}
		\begin{tikzcd}[column sep=small]
			\label{objects}
			\mathcal X \arrow[rr,"\alpha"] \arrow[rd]&&\mathcal X'\arrow[ld]\\
			&B&
		\end{tikzcd}
	\end{equation}
	between families $\mathcal X\to \mathcal X'$ over $B$ (with $\mathcal X'$ in $\mathscr{A}_0$)
	modulo isomorphisms $\beta$ over $B$
	\begin{equation}
		\begin{tikzcd}[row sep=small]
			\label{morphisms}
			&\mathcal X'\arrow[dd, "\beta"]\\
			\mathcal X \arrow[ur,"\alpha"] \arrow[rd, "\alpha'"']&\\
			&\mathcal X''
		\end{tikzcd}
	\end{equation}
	belonging to $\mathscr{A}_0$.
	
	Assume $B$ connected. Decompose $B$ as a union of connected open sets $B_1\cup\hdots\cup B_k$ in such a way that the family $\mathcal X$, resp. $\mathcal X'$, 
	is locally isomorphic above $B_i$ to $u_i^*\mathscr{K}_0$ for some $u_i : B_i\to K_0$, resp. to $(u'_i)^*\mathscr{K}_0$ for some $u'_i : B_i\to K_0$. These local models are glued through a cocycle $u_{ij}:B_i\cap B_j\to \mathcal T_V$, resp. $u'_{ij}:B_i\cap B_j\to \mathcal A_0\subset \mathcal T_V$, 
	satisfying $s(u_{ij})=u_i$ and $t(u_{ij})=u_j$, resp. $s(u'_{ij})=u'_i$ and $t(u'_{ij})=u'_j$, to obtain a family isomorphic to $\mathcal X$, resp. $\mathcal X'$.
	
	In these models, up to passing to a finer covering, an isomorphism \eqref{objects} corresponds to a collection $F_i:B_i\to \mathcal T_V$ fulfilling
	\begin{enumerate}[i)]
		\item $s\circ F_i=u_i$ and $t\circ F_i=u'_i$
		\item $m(u_{ij},F_j)=m(F_i,u'_{ij})$
	\end{enumerate} 
	Then $\mathcal X''$ corresponds to a cocycle $u_i'' : B_i\to K_0$ and $\alpha'$ to a collection $F'_i:B_i\to\mathcal T_V$ satisfying similar relations.\\
%
	Let $\beta$ be $\alpha'\circ\alpha^{-1}$. This is a morphism of $\mathscr{T}(M,V)$ which is given in our localisation by the collection 
	\begin{equation}
		\label{beta}
		G_i:=m(i(F_i),F'_i): B_i\longrightarrow \mathcal{T}_V
	\end{equation}
We want to know when $\beta$ is a morphism of $\mathcal{A}_0$, that is when $G_i$ has image in $\mathcal{A}_0$ for all $i$. 

Since the $B_i$ are connected, the image of each map $F_i$, $F'_i$ is included in a single connected component of the space $\mathcal{S}_i$ of $s$-sections of $B_i\times_{u_i} \mathcal T_V$ above $B_i$. By Lemma \ref{lemma1main}, $F_i$ and $F'_i$ land in the same connected component of $\mathcal{S}_i$ if and only if $G_i$ lands in $\mathcal{A}_0$. 

Choose a point $b_i$ in each $B_i$. Then $\alpha$ and $\alpha'$ are equivalent through \eqref{morphisms} if and only if $(b_i,F_i(b_i))$ and $(b_i,F'_i(b_i))$ belong to the same connected component of $\mathcal{S}_i$ for all $i$. 

Now, assume that $(b_1,F_1(b_1))$ and $(b_1,F'_1(b_1))$ belong to the same connected component of $\mathcal{S}_1$. Given $i\not = 1$ and taking $c\in B_1\cap B_i$, it follows from the compatibility relations that 
	\begin{equation}
	\label{compatibility1}
	F_i(c)=m(u_{i1},m(F_1, u'_{1i}))(c)
\end{equation}
and
\begin{equation}
	\label{compatibility2}
	F'_i(c)=m(u_{i1}, m(F'_1, u''_{1i}))(c)
\end{equation} 
But $u'_{1i}$ and $u''_{1i}$ are mappings with values in $\mathcal A_0$, hence, applying once again Lemma \ref{lemma1main}, we deduce that $m(F_1, u'_{1i})(c)$ and $m(F'_1, u''_{1i})(c)$ belong to the same connected component of $\mathcal{S}_1$, say $S$, and finally $F_i(c)$ and $F'_i(c)$ to the same connected component of $\mathcal{S}_i$ since both lie in the image of $S$ by $m(u_{i1},-)$. And so do $F_i(b_i)$ and $F'_i(b_i)$ for all $i$ by connectedness of $B$.

As a consequence, $\alpha$ and $\alpha'$ are equivalent through \eqref{morphisms} if and only if $(b_i,F_i(b_i))$ and $(b_i,F'_i(b_i))$ belong to the same connected component of $\mathcal{S}_i$ {\slshape for some} $i$. 
	
%
%
	
	Therefore, the fiber product $B\times_u\mathscr{A}_0$ identifies with a disjoint union of copies of $B$. On the points of $B\times_u\mathscr{A}_0$\footnote{ that is, for objects above some point $b\in B$.}, this identification is given by the map
	\begin{equation}
		\label{RepFP}
		(b,F(b))\in (B\times_{u}\mathscr A_0)\longmapsto (b,\sharp (F_i(b_0))\in B\times\sharp S_0
	\end{equation}
	Here $S_0$ is the $s$-fiber of $\mathcal{T}_V$ above $u_i(b_0)$, the set $\sharp S_0$ is the set of connected components of $S_0$ and the $\sharp$ application maps an element of $S_0$ to the connected component of $S_0$ which contains it; the mapping $F_i$ is defined as above as a local expression for $\alpha$ satisfying \eqref{objects}, the point $b_0$ and the index $i$ are fixed with $b_0$ belonging to $B_i$. 
	It follows from what preceeds that the quantity $F_i(b_0)$ depends only on the class of $\alpha$ modulo \eqref{morphisms}, showing that \eqref{RepFP} is an isomorphism onto its image. Its image is $B\times E_0$, for $E_0$ a subset of $\sharp S_0$ which can be a strict subset because some connected components of $S_0$ may not compatible with any cocycle of the family $\mathcal X$. Indeed, it may be empty, corresponding to a family $\mathcal X$ that is not isomorphic to any family of $\mathscr{A}_0$, i.e. corresponding to $B\times_u\mathscr{A}_0$ empty. 
	Finally, when it is not empty, $f_1$ can be rewritten as the natural projection map
	\begin{equation}
		\label{finite1}
		B\times \{1,\hdots ,g(\mathcal X)\}\longrightarrow B
	\end{equation}
	for $g(\mathcal X)$ a number $\N^*\cup\{+\infty\}$ that depends on $\mathcal X$, as the notation suggests, and is the number of connected components of $S_0$ that can be attained through \eqref{RepFP}. This proves that the inclusion of $\mathscr{A}_0$ in $\mathscr{T}(M,V)$ is an \'etale morphism.

	Let us deal now with the $\mathscr{A}_1$ case. We thus consider diagrams \eqref{objects} and \eqref{morphisms} with $\mathcal X'$ and $\beta$ in $\mathcal{A}_1$. We define $u_i$, $u'_i$, $u''_i$, $F_i$, $F'_i$ and $G_i$ as before. We want to know when $G_i$ has image in $\mathcal{A}_1$ for all $i$. Using \eqref{compatibility1} and \eqref{compatibility2}, we obtain that $G_i$ is in $\mathcal{A}_1$ for all $i$ if and only if $G_1$ is in $\mathcal{A}_1$. 
	
	Then, by connexity of $B_1$, this occurs if and only if $G_1(b_1)$ belongs to $\mathcal{A}_1$ for some fixed $b_1\in B_1$. Indeed, letting $b$ be another point of $B_1$, we have that $G_1(b_1)$ and $G_1(b)$ belong to the same connected component of $\mathcal S_1$, hence are related through an element of $\mathcal{A}_0$ by Lemma \ref{lemma1main}. Since the composition of an element of $\mathcal{A}_1$ with an element of $\mathcal{A}_0$ belongs to $\mathcal{A}_1$, we deduce that $G_1(b)$ is also in  $\mathcal{A}_1$ for all $b\in B_1$.
	
	This is equivalent to $F_1(b_1)$ and $m(F'_1(b_1),h)$ are $s$-homotopic for some $h\in\mathcal{A}_1$ with $s(h)=t(F_1(b_1))$. Say that $F_1(b_1)$ and $F'_1(b_1)$ are $(s,1)$-homotopic when this is true. In other words, we define on $S_0$ the following equivalence relation: $(J,f)$ and $(J,g)$ are $(s,1)$-homotopic (where $J=u_1(b_1)$) if and only if $(J\cdot f, f^{-1}\circ g)$ belongs to $\mathcal{A}_1$. It is is straightforward to check that it only depends on the connected component of $(J,f)$ and $(J,g)$ in $S_0$, that is the $(s,1)$-homotopy descends as an equivalence relation on $\sharp S_0$.

%
%
	
	
	Thus, one eventually finds that the fiber product $B\times_u\mathscr{A}_1$ identifies with the disjoint union of an at most countable number of copies of $B$, say $g_1(\mathcal X)$, through the map
	\begin{equation}
		\label{RepFP1}
		(b,F(b))\in (B\times_{u}\mathscr A_1)\longmapsto (b,\sharp_1 F_i(b_0))\in B\times\sharp_1 S_0
	\end{equation}
	Here we use the same notations and conventions as in \eqref{RepFP}, $\sharp_1 S_0$ is the set of $(s,1)$-homotopy classes of $\sharp S_0$, and the $\sharp_1$ application maps an element of $\mathcal T_V$ to the element of $\sharp_1$ which contains it. 
	Finally $f_1$ can be rewritten as the natural projection map
	\begin{equation}
		\label{finite2}
		B\times \{1,\hdots ,g_1(\mathcal X)\}\longrightarrow B
	\end{equation}
	for $g_1(\mathcal X)$ the number of $\sharp\Aun$-orbits of connected components of $\mathcal T_V$ that can be attained through \eqref{RepFP1}. 
\end{proof}
 Of course, a similar statement holds for the $\Z$-Teichm\"uller stack. We have indeed
 \begin{corollary}
 	\label{corZfinite}
 	In the following commutative diagram of natural inclusions,
 	\begin{equation}
 		\label{CDZfinite}
 		\begin{tikzcd}
 			\mathscr{A}_0\arrow[d,equal]\arrow[r,hook] &\mathscr{A}_1\arrow[d,hook]\arrow[r,hook] &\mathscr{T}(M,V)\arrow[d,hook]\\
 			\mathscr{A}_0\arrow[r,hook]&\mathscr{A}_\Z\arrow[r,hook]&\mathscr{T}^\Z(M,V)
 		\end{tikzcd}
 	\end{equation}
 	every arrow is an étale morphism of analytic stacks.
 \end{corollary}
 
 \begin{proof}
 	The top line is given by Theorem \ref{finitethm} and the bottom line is proven in the same way. Then, take the fibered product of \eqref{CDZfinite} with some $u: B\to \mathscr{T}^\Z(M,V)$ obtaining
 	\begin{equation}
 		\label{CDZfinite2}
 		\begin{tikzcd}
 			B\times_{u}\mathscr{A}_0\arrow[d,equal]\arrow[r] &B\times_{u}\mathscr{A}_1\arrow[d]\arrow[r] &B\times_{u}\mathscr{T}(M,V)\arrow[d]\\
 			B\times_{u}\mathscr{A}_0\arrow[r]&B\times_{u}\mathscr{A}_\Z\arrow[r]&B
 		\end{tikzcd}
 	\end{equation}
 	that is a diagram of analytic spaces with all horizontal lines being étale. Thus the vertical ones are also étale and we are done.
 \end{proof}
 Since being étale is a local at base property, we also have
  \begin{corollary}
 	\label{corZfinite2}
 	The natural inclusion of $\T$ in $\TZ$ is étale analytic.
  \end{corollary}

 \section{The main conjecture}
 \label{secconjecture}
 \subsection{Exceptional and $\Z$-exceptional points}
 \label{subsecexc}

 In view of Theorem \ref{finitethm} and Corollary \ref{corZfinite}, it is natural to single out the following case.
\begin{definition}
	\label{defexcpoint}
	We say that $X_0$ is an {\slshape exceptional point of the Teichm\"uller stack $\T$} or simply that $X_0$ is {\slshape exceptional} if there is no neighborhood $V$ of $X_0$ such that the \'etale morphism $\mathscr{A}_1\to\mathscr{T}(M,V)$ is an isomorphism.
	
	Analogously, we say that $X_0$ is  an {\slshape exceptional point of the $\Z$-Teichm\"uller stack $\TZ$}, or simply that $X_0$ is {\slshape $\Z$-exceptional} if there is no neighborhood $V$ of $X_0$ such that the \'etale morphism $\mathscr{A}_\Z\to\mathscr{T}^\Z(M,V)$ is an isomorphism.
\end{definition}

The idea behind this definition is of course that these \'etale morphisms  should be isomorphisms at a generic point for a sufficiently small $V$. The situation is however much more complicated. It turns out that it strongly depends on the existence of a Kähler metric on the manifold $X_0$.

\subsection{Conjecture on exceptional points}
\label{subsecconjex}
We now state and discuss the main conjecture on exceptional points that will occupy ourselves in Sections \ref{secexc} to \ref{secdistri}.

\begin{conjecture}
	\label{mainconj}
	{\bf (Main conjecture on exceptional points).}
	\begin{enumerate}[{\rm \bf I.}]
		\item Let $X_0$ be K\"ahler. Then $X_0$ is neither exceptional in $\T$ nor $\Z$-exceptional in $\TZ$.
		\item There exist some exceptional and $\Z$-exceptional (non-K\"ahler) points. They may even be dense in a connected component of $\T$ or $\TZ$. 
		\item Every exceptional, resp. $\Z$-exceptional point, is vanishing or wandering.
	\end{enumerate}
\end{conjecture}

As we already did in the introduction \S \ref{intro}, we emphasize that classical deformation theory from an analytic point of view is rather insensible to K\"ahlerianity and that the dichotomy in Conjecture \ref{mainconj} is only seen on the Teichm\"uller and $\Z$-Teichm\"uller stacks, not on the moduli stack. Especially, only in the context of K\"ahler manifolds acted on by $\D$ or $\DZ$ can we use results on the compacity of cycle spaces. 

Point I of Conjecture \ref{mainconj} is the most optimistic. Basic reasons to believe it include the already cited compacity results in the K\"ahler setting such as Lieberman's Theorem recalled in \S \ref{secfinite}; the fact that submanifolds of K\"ahler manifolds represent non-trivial cohomology classes rigidifying $\Aun$ and $\AZ$; and the difficulties in finding an example of an exceptional point even for non-K\"ahler manifolds. But this is far from giving a strong evidence for I. It also supposes that the $\D$ and $\DZ$-orbits of $\I$ have a very simple topology with no holonomy phenomenon. We shall prove in Section \ref{secdistri} a weaker result: exceptional K\"ahler points, resp. $\Z$-exceptional K\"ahler points, if exist, form a strict analytic substack of $\T$, resp. $\TZ$, see Theorem \ref{2ndmainthm}.

Point II of Conjecture \ref{mainconj} seems more plausible although not easier to prove. Basic reasons to believe it boil down to the fact that all the techniques used to prove Theorem \ref{2ndmainthm} break down completely in the non-K\"ahler setting. Also we give in \S \ref{secexample} an example of $\Z$-exceptional points with a property of local density, see Theorem \ref{thmexampleexc}. We are unable however to show the existence of exceptional points.

Point III of Conjecture \ref{mainconj} is more technical but explains the difference of behaviour between the K\"ahler and the non-K\"ahler case stated in points I and II. It makes reference to the classification of exceptional points into exceptional, vanishing and wandering points introduced in \S \ref{secintrocycles}. It implies ppint I since vanishing and wandering points do not exist in the K\"ahler context, see Corollary \ref{4thcor}. Its general meaning is the following. Vanishing, resp. wandering points, encompass non compactness of a component of cycles, resp. non finiteness of the number of components.

Indeed, Conjecture \ref{mainconj} sheds some light on the dichotomy between K\"ahler and non-K\"ahler points, pushes forward it to the extreme form of a $0-1$ conjecture. We hope it will serve as a challenging problem and a source of motivation for studying these questions. 

Before looking with more care at exceptional points, we will focus on the points where the Teichm\"uller stack is locally simpler. 

\section{Structure of the Teichm\"uller stack of normal points}
\label{localTeich2}
Section \ref{secconjecture} introduces the notion of exceptional points. At such a point, the local Teichm\"uller stack includes morphisms that are not close to automorphisms of the central fiber, adding complexity. But jumping points are also bad points, where the Teichm\"uller space is usually locally non-Hausdorff at $X_0$. Roughly speaking, normal points are points of $\T$ that are neither jumping nor exceptional.
\subsection{Normal points}
\label{secnormalpoints}
We start with the following Lemma
\begin{lemma}
	\label{lemmah0constant}
	Assume {\rm\eqref{h0}} is locally constant at $X_0$. If $V$ is small enough, every automorphism of $\Azero$ extends as an automorphism of the whole Kuranishi family.
\end{lemma}

\begin{proof}
	Assume $K_0$ reduced. Recall that the Lie algebra of $\Azero$ is the cohomology group $H^0(X_0,\Theta_0)$. Since \eqref{h0} is locally constant at $0$, the projection
	\begin{equation}
		\bigcup_{t\in K_0}H^0(X_t,\Theta_t)\longrightarrow K_0
	\end{equation} 
	is a locally trivial fiber bundle for $V$ and $K_0$ small enough. Hence we may extend every element of $H^0(X_0,\Theta_0)$ as a holomorphic vector field tangent to the fibers of the Kuranishi family. Taking the exponential and composing, this means that every element of $\Azero$ extends as an automorphism of the whole Kuranishi family as wanted.
	
	If $K_0$ is not reduced, then the previous argument shows that every automorphism of $\Azero$ extends as an automorphism above $K_0^{red}$, the reduction of $K_0$. Now, given $f\in\Azero$ and $G$ an extension of $f$
	\begin{equation}
		\label{extensionf}
		\begin{tikzcd}
			\mathscr{K}_0^{red}\arrow[r,"G"]\arrow[d]&	\mathscr{K}_0^{red}\arrow[d]\\
			K_0^{red}\arrow[r,"g"'] &K_0^{red}
		\end{tikzcd}
	\end{equation}
	then $g$ is defined on $K_0$ through
	\begin{equation}
		g(J):=J\cdot G_J
	\end{equation}
	Then, since $\mathscr{K}_0\to K_0$ is a smooth morphism, $\mathscr{K}_0$ is locally isomorphic to $K_0\times\C^n$. On such an open subset of $K_0\times\C^n$, we define $G$ as
	\begin{equation}
		(J,z)\longmapsto (g(J),G_J(z))
	\end{equation}
	and we are done.
\end{proof}
Notice that, when $K_0$ is reduced, the proof of Lemma \ref{lemmah0constant} shows that every automorphism of $\Azero$ extends as an automorphism of the nearby fibers of the Kuranishi space of $X_0$. Indeed the extension at $J$ is given by the function $f\circ e(\chi(J))$ appearing in \eqref{sigmaf}, which is an automorphism of $X_J$ by Lemma \ref{lemmaHolWavrik}. As a consequence the source and target morphisms of $\mathcal A_0\rightrightarrows K_0$ are equal. If $K_0$ is not reduced, then the extensions descend as the identity on the reduction of $K_0$ but not always as the identity of $K_0$. 

 This motivates the following definition.
\begin{definition}
	\label{defnormal}
	A point $X_0$ of the Teichm\"uller stack is a {\sl normal point} if
	\begin{enumerate}[\rm i)]
		\item choosing $V$ small enough, every automorphism of $\Azero$ extends as an automorphism of the whole Kuranishi family which descends as the identity on $K_0$.
		\item It does not belong to the closure of the set of exceptional points.
	\end{enumerate}
\end{definition}

Replacing exceptional with $\Z$-exceptional in Definition \ref{defnormal} gives the notion of $\Z$-normal points.

As a consequence of point i) in Definition \ref{defnormal}, the source and target maps of $\mathcal A_0\rightrightarrows K_0$ can be assumed to be equal at a normal point. Indeed, it is exactly the analytic space and morphism
\begin{equation}
	\label{Namba}
	\mathscr{N}:=\bigcup_{t\in K_0}\text{Aut}^0(X_t)\longrightarrow K_0
\end{equation} 
constructed by Namba\footnote{To be precise, the Namba space is the union of the full groups $\text{Aut}(X_t)$, hence \eqref{Namba} is an open subset of it.} in \cite{Namba} when $K_0$ is reduced. This can be proven as follows. By Lemma \ref{lemmaautcontained}, there is a set theoretic inclusion of $\mathscr{N}$ into $\mathcal{A}_0$. The equality $s=t$ yields the reverse inclusion. Taking account that the topology in $\mathscr{N}$ is that of uniform convergence, this bijection is indeed a homeomorphism. Finally, both $\mathcal{A}_0$ and $\mathscr{N}$ being smooth over $K_0$ reduced - for $\mathscr{N}$, this is true because \eqref{h0} is constant - with same fibers, this homemorphism can be turned into an analytic isomorphism. 

Recall also that point i) in Definition \ref{defnormal} is stronger than \eqref{h0} being locally constant in the non-reduced case.

\begin{example}
	\label{fakeorbifold}
	Consider the case of compact complex tori. Then $\Azero$ is not trivial, since it contains the translations. So neither $\mathscr{T}(M,V)$ nor $\mathscr{A}_1$ is an orbifold, since their isotropy groups are not finite. However, if we forget about the stack structure, the Teichm\"uller space is naturally a complex manifold. Indeed, roughly speaking, the stack is obtained from this complex manifold by attaching a group of translations to each point. This is an example of a stack represented by an analytic groupoid with $s$ and $t$ equal. More precisely, \eqref{Namba} is the universal family of tori, see \cite{LMStacks}, Example 13.1. 
\end{example}

At a normal point, resp. a $\Z$-normal point, the local Teichm\"uller stack and the Kuranishi stack $\mathscr{A}_1$, resp. the local $\Z$-Teichm\"uller stack and $\mathscr{A}_\Z$, coincide. Lemma \ref{lemmah0constant} and \eqref{Namba} are however not enough to describe both of them, since $\Aun$, resp. $\AZ$, may have several connected components.

\subsection{Kuranishi stack as an orbifold}
\label{Kurorbifold}
Before analyzing the Teichm\"uller stack of normal points, we investigate the important case when the Kuranishi stack(s) is (are) an orbifold. Here by an orbifold, we mean a stack given as the global quotient of an analytic space by an holomorphic action of a finite group with a point fixed by the whole group. We include non-effective actions.
We have
\begin{theorem}
	\label{Kurstackorbifold}
	The following two statements are equivalent
	\begin{enumerate}[\rm i)]
		\item There exists some open neighborhood $V'\subset V$ of $X_0$ such that the Kuranishi stack $\mathscr{A}$ restricted to $V'$ is an orbifold 
		\item $\A$ is finite
	\end{enumerate}
\end{theorem}
\begin{remark}
	\label{precisestatement}
	In order to endow the Kuranishi stack $\mathscr{A}$ with a structure of an orbifold, we need to start with an open set $V$ stable under the action of the automorphisms. This explains the restriction to some $V'$ in the statement. In the proof $V'$ is constructed as such a stable open set.
\end{remark}

\begin{proof}
	Since the isotropy group of $X_0$ is $\A$, the condition is obviously necessary. So let us assume that $\A$ is finite. We start with an arbitrary atlas $\mathcal A\rightrightarrows K_0$. We assume that the $\A$ version of \eqref{Gdecompoglobal} is valid. 
	
	We show that we may choose $ V'\subset V$ so that the corresponding atlas $\mathcal A'\rightrightarrows K_0\cap V'$ of $\mathscr{A}$ is Morita equivalent to the translation groupoid $\A\times K_0\cap V'\rightrightarrows K_0\times V'$.

	The proof of Theorem \ref{Kurstackorbifold} consists in the following three lemmas.
	
	\begin{lemma}
		\label{lemma1}
		For all $f\in\A$, the map $\sigma_f : U_f\subset K_0\to \mathcal A$ constructed in {\rm\eqref{sigmaf}} is the unique (up to restriction) extension of $f$.
	\end{lemma}
	
	By extension of $f$, we mean a section $F$ of $s$ defined in a neighborhood of $0$ and such that $F(0)=f$.
	
	\begin{proof}[Proof of Lemma \ref{lemma1}]
		
		The map $\sigma_f$ is obviously an extension of $f$ as desired. 
		
		Let now $G$ be another extension of $f$. Then, for all $J\in K_0$ close to $0$, we have a decomposition
		\begin{equation}
			\label{extensionbis}
			G(J)=f\circ e(\eta(J))
		\end{equation}
		using \eqref{Gdecompoglobal}. Here the factor in $\A$ is constant equal to $f$ since $\A$ is discrete.
		
		We have
		\begin{equation}
			J\cdot G(J)=\Xi(J\cdot G(J))=\Xi (J\cdot f)=\text{Hol}_f(J)
		\end{equation}
		so the mapping $\eta$ also satisfies \eqref{chichecks}. 
		But since \eqref{Kurmap} is an isomorphism, \eqref{chichecks} is uniquely verified and $\eta=\chi$. Thus $G=\sigma_f$ on a neighborhood of $J_0$ in $K_0$.		
	\end{proof}
	As a consequence, we have
	\begin{lemma}
		\label{lemma1bis}
		For all $f\in\A$, and $g\in\A$, we have  $\sigma_{g\circ f}=m(\sigma_{g}, \sigma_{f})$ on a neighborhood of $J_0$.
	\end{lemma}
	
	\begin{proof}[Proof of Lemma \ref{lemma1bis}]
		Define 
		\begin{equation}
			m(\sigma_g,\sigma_f)\ :\ J\longmapsto m(\sigma_g(J),\sigma_f(J\cdot\sigma_g(J)))
		\end{equation} 
		This is an extension of $g\circ f$, and thus by Lemma \ref{lemma1} is equal to $\sigma_{g\circ f}$ on a neighborhood of $J_0$.
	\end{proof}
	
	Let $U$ be the intersection of all $U_f$ for $f$ in the finite group $\A$. Then all $\text{Hol}_f$ are defined on $U$ with values in $K_0$. Redefine $K_0$ as the intersection of all $\text{Hol}_f(U)$ for $f$ in the finite group $\A$.  Observe that
	\begin{equation}
		\text{Hol}_g(\cap_f (\text{Hol}_f(U)))=(\cap_f (\text{Hol}_{f\circ g}(U)))=\cap_f (\text{Hol}_f(U))
	\end{equation}
	because the intersection of the $\text{Hol}_f(U)$ is included in $U$ and because of Lemma \ref{lemma1bis}. Then all $\text{Hol}_f$ map bijectively $K_0$ to $K_0$. Associated to this new $K_0$ and to \eqref{Gdecompoglobal} is some $V'\subset V$.
	Set 
	\begin{equation}
		\label{Ext}
		\mathscr{E}xt=\{\sigma_f : K_0\to \mathscr{A}\mid f\in\A\}
	\end{equation}
	We have
	\begin{lemma}
		\label{lemma2}
		$(\mathscr{E}xt,\circ)$ is a group isomorphic to $\A$.
	\end{lemma}
	
	\begin{proof}[Proof of Lemma \ref{lemma2}]
		By Lemma \ref{lemma1bis}, $m(\sigma_g,\sigma_f)$ is equal to $\sigma_{g\circ f}$ on a neighborhood of $J_0$, hence on $K_0$ by analyticity.
		
		The same arguments used in Lemma \ref{lemma1bis} show that if $g_1\circ\hdots\circ g_k=Id$, then the same relation holds for the $\sigma_{g_i}$'s.
	\end{proof}
	The space $K_0$ is invariant by the action of the group $(\mathscr{E}xt,\circ)$, which describes all the morphisms	of the Kuranishi stack. We may thus take as atlas for $\mathscr{A}$ the translation groupoid $\mathscr{E}xt\times K_0\rightrightarrows K_0$, or, equivalently the translation groupoid $\A\times K_0\rightrightarrows K_0$.
\end{proof} 

Replacing $\mathscr{A}$ with $\mathscr{A}_\Z$, resp. $\mathscr{A}_1$, resp. $\mathscr{A}_0$ and $\A$ with $\AZ$, resp. $\Aun$, resp. $\Azero$ yields the following immediate corollaries.

\begin{corollary}
	\label{Kur1stackorbifold}
	The following two statements are equivalent
	\begin{enumerate}[\rm i)]
		\item There exists some open neighborhood $V'\subset V$ of $X_0$ such that the Kuranishi stack $\mathscr{A}_1$ restricted to $V'$ is an orbifold 
		\item $\Aun$ is finite
	\end{enumerate}
\end{corollary}
then,
\begin{corollary}
	\label{KurZstackorbifold}
	The following two statements are equivalent
	\begin{enumerate}[\rm i)]
		\item There exists some open neighborhood $V'\subset V$ of $X_0$ such that the Kuranishi stack $\mathscr{A}_\Z$ restricted to $V'$ is an orbifold 
		\item $\AZ$ is finite
	\end{enumerate}
\end{corollary}
and finally,
\begin{corollary}
	\label{Kur0stackorbifold}
	The following two statements are equivalent
	\begin{enumerate}[\rm i)]
		\item There exists some open neighborhood $V'\subset V$ of $X_0$ such that the Kuranishi stack $\mathscr{A}_0$ restricted to $V'$ is an orbifold 
		\item  There exists some open neighborhood $V'\subset V$ of $X_0$ such that the Kuranishi stack $\mathscr{A}_0$ restricted to $V'$ is an analytic space
		\item $\Azero$ is reduced to the identity.
	\end{enumerate}
\end{corollary}

\begin{proof}
	Just notice that $\Azero$ is finite if and only if it is reduced to the identity.
\end{proof}

\subsection{The \'etale Teichm\"uller stack of normal points}
\label{secTSnormal}
We are now in position to analyse the structure of the Teichm\"uller stack restricted to the set of normal points. First note the following result.

\begin{proposition}
	\label{propnormalopen}
	The subset $\mathcal{N}\subset\mathcal{I}$ of normal points is an open set of $\mathcal{I}$.
\end{proposition}

\begin{proof}
	Let $J_0$ be a normal point of $\mathcal I$. Then $\mathcal{A}_0\rightrightarrows K_0$ have equal source and target morphisms. Taking $K_0$ smaller if necessary, we may assume that $K_0$ does not meet the closure of exceptional points and that $K_0$ is complete at any point, see \cite{ENS}. Let now $J_1$ belong to $K_0$. Set $X_1=(M,J_1)$. Since \eqref{h0} is locally constant, Proposition 2 of \cite{Kur3} applies and the germ of $K_0$ at $J_1$ is universal for families above a reduced base, hence the reduction of $K_0$ at $J_1$ in $K_0$ and the reduction of the Kuranishi space $K_1$ of $X_1$ are locally isomorphic. It follows from completeness that $K_1$ injects in $K_0$ at $X_1$. Hence the Kuranishi stack $\mathcal{A}_0(X_1)\rightrightarrows K_1$ of $X_1$ injects in $\mathcal{A}_0\rightrightarrows K_0$ at $X_1$: any morphism of $\mathcal{A}_0(X_1)$ is an extension of an automorphism of $X_1$, so belong to $\mathcal{A}_0$ by Lemma \ref{lemmaautcontained} taking $K_1$ smaller if necessary. As a consequence its source and target morphisms are equal and $X_1$ is a normal point.
\end{proof}

Let $X_0$ be a normal point. Since the source and target maps of $\mathcal{A}_0\rightrightarrows K_0$ are equal, the multiplication of the groupoid induces a fibered action of $\mathcal A_0$ onto $\mathcal A_1$ that preserves the source and target maps of $\mathcal{A}_1\rightrightarrows K_0$. Given $(J,f)\in\mathcal{A}_0$ and $(J,g)\in\mathcal{A}_1$, we set
\begin{equation}
	\label{actionA0A1}
	(J,f)\cdot (J,g):=m((J,f),(J,g))=(J,f\circ g)
\end{equation}
We have
\begin{lemma}
	\label{lemmaquotientgroupoid}
	Assume $X_0$ is normal. Then, the quotient space $\mathcal{A}_1/\mathcal{A}_0$ is an analytic space and the morphism $s$, resp. $t :\mathcal{A}_1\to K_0$, descends as an \'etale morphism from $\mathcal{A}_1/\mathcal{A}_0$ to $K_0$.
\end{lemma}

\begin{proof}
	Let $(J,f)$ belong to $\mathcal{A}_1$. For $(J,f')$ close enough to $(J,f)$, the diffeomorphism $f'\circ f^{-1}$ belongs to $\text{Aut}^0(X_J)$ hence to $\mathcal{A}_0$ by Lemma \ref{lemmaautcontained} and $(J,f')$ equals $(J,f)$ in the quotient space. Choose a local $s$-, resp. $t$-section $\sigma$ from a neighborhood of $J$ in $K_0$ to a neighborhood of $(J,f)\in\mathcal{A}_1$ with $\sigma(J)=(J,f)$. Such local sections exist since $s$, resp. $t$ is a smooth morphism. Then the restriction of $s$, resp. $t$ to this section realizes a local isomorphism between $\mathcal{A}_1/\mathcal{A}_0$ and $K_0$.
	
	We are left with proving that $\mathcal{A}_1/\mathcal{A}_0$ is Hausdorff. Let $(J_n,f_n)$ converge to $(J,f)$ in $\mathcal{A}_1$ and $(J'_n,g_n)$ converge to $(J',g)$. Assume they are $\mathcal{A}_0$-equivalent. By definition, we thus have $J_n=J'_n$ and $g_n\circ f_n^{-1}$ is an element of $\text{Aut}^0(X_{J_n})$ for all $n$. Since $K_0$ is Hausdorff, passing to the limit gives $J=J'$. Moreover all $g_n\circ f_n^{-1}$ are holomorphic extensions at $J_n$ of some automorphism of $\Azero$ by Lemma \ref{lemmah0constant}, hence the limit $g\circ f^{-1}$ still belongs to $\Azero$. Hence a convergent sequence in $\mathcal{A}_1/\mathcal{A}_0$ has a unique limit, showing Hausdorffness.
\end{proof}
Hence we may define an \'etale quotient groupoid $\mathcal{A}_1/\mathcal{A}_0\rightrightarrows K_0$. Its stackification over the analytic site describes classes of reduced $(M,V)$-families up to $\mathcal{A}_0$-equivalence. 

This stack can be defined over the full open set $\mathcal N$ of normal points. In this context, a reduced $(M,\mathcal N)$-family is $\mathcal{A}_0$-equivalent to a trivial family if it can be decomposed as local pull-back families glued by a cocycle in $\mathcal{A}_0$, cf. the proof of Theorem \ref{finitethm}; and an isomorphism of a reduced $(M,\mathcal N)$-family is $\mathcal{A}_0$-equivalent to the identity, if it is given by local $\mathcal{A}_0$-sections once decomposed as local pull-back families. We set
\begin{definition}
	\label{defnormalstack}
	The stack over the analytic site of $\mathcal{A}_0$-equivalence classes of reduced $(M,\mathcal N)$-families is called {\slshape the normal Teichm\"uller stack} and denoted by $\mathscr{N}\mathscr{T}(M)$.
\end{definition}
\noindent and
\begin{definition}
	\label{defmcg1}
	We call {\slshape $\Aun$ mapping class group } the group
	\begin{equation}
		\label{map1}
		\text{Map}^1(X_0):=\Aun/\Azero
	\end{equation}
\end{definition}
We are in position to give some properties of $\mathscr{NT}(M)$

\begin{theorem}
	\label{thmnormal}
	The normal Teichm\"uller stack $\mathscr{NT}(M)$ satisfies the following properties
	\begin{enumerate}[\rm i)]
		\item It is an analytic \'etale stack with atlas an (at most) countable union of Kuranishi spaces.
		\item There is a natural morphism from $\mathscr{T}(M,\mathcal N)$ to $\mathscr{NT}(M)$. Moreover, these two stacks are associated to the same topological quotient space.
		\item Let $X_0$ be a normal point. The isotropy group of $\mathscr{NT}(M)$ at $X_0$ is the discrete group $\text{\rm Map}^1(X_0)$.
		\item Assume $\text{\rm Map}^1(X_0)$ is finite. Then, there exists some open neighborhood $V'\subset V$ of $X_0$ such that, reshaping $K_0$ to $V'$,  the group $\text{\rm Map}^1(X_0)$ acts holomorphically on $K_0$ fixing $J_0$ and $\mathscr{NT}(M)$ is locally isomorphic at $X_0$ to the orbifold $[K_0/\text{\rm Map}^1(X_0)]$.
		\item Especially, if every normal point belongs to Fujiki class $(\mathscr{C})$, then  the stack $\mathscr{NT}(M)$ is an orbifold at every point.
\end{enumerate}
\end{theorem}

\begin{remark}
	\label{notrepres}
	The natural morphism of point ii) is not representable since the isotropy groups of $\mathscr{T}(M)$ do not inject in those of $\mathscr{NT}(M)$.
\end{remark}

Hence, an open substack of the Teichm\"uller stack behaves as the quotient of an analytic space by a discrete equivalence relation; and as an orbifold around a point in Fujiki class $(\mathscr{C})$, that is bimeromorphic to a K\"ahler manifold, so in particular around any K\"ahler or projective point.

This \'etale normal Teichm\"uller stack is closer to the moduli space side than to the family side of the Teichm\"uller stack. It forgets the family automorphisms that induce the identity on the base but these automorphisms do not induce any identification of points in $\I$ so are not important when analyzing the structure of the set of $\D$-orbits in $\I$.

We may reformulate Theorem \ref{thmnormal} by saying that the restriction  of the topological Teichm\"uller space $\mathcal N/\D$ to the set of normal points acquires complex orbifold charts at the points with finite $\text{Map}^1(X_0)$ group.

Before giving the proof of this result, we give some examples.

\begin{example}
	\label{extorus2}
	We consider once again the case of complex tori, compare with Example \ref{fakeorbifold}. All points are normal and passing from the Teichm\"uller stack to the normal Teichm\"uller stack consists of forgetting about the translation group and replacing the universal family over the upper half plane $\mathbb{H}$ with $\mathbb{H}$. Hence the normal Teichm\"uller stack is the standard Teichm\"uller space $\mathbb{H}$ and the surjective morphism sends a family of complex tori $\pi : \mathcal X\to B$ to the morphism 
	\begin{equation}
		b\in B\longmapsto \tau(b)\in\mathbb{H}\qquad\text{ with }\qquad \pi^{-1}(b)\simeq\mathbb{E}_{\tau(b)}
	\end{equation}
\end{example} 

\begin{example}
	\label{exK3}
	We consider now the case of K3 surfaces. The function $h^0$ is constant equal to zero and the Kuranishi spaces of K3 can be glued together to form a 20-dimensional complex manifold called the moduli space of marked K3 surfaces. This is at the same time the Teichm\"uller stack and the normal Teichm\"uller stack. As for tori, all points are normal. However, this space is non-Hausdorff because a sequence of $\D$-orbits may accumulate onto two disjoint orbits. Non-separated pairs of points encode however the same manifold. In other words, they are in the same orbit of the mapping class group $\text{Diff}^+(M)/\D$. See \cite[\S 7.2]{Huybrechts} for more details about all this.
\end{example}

\begin{example}
	\label{exHopf2}
	We consider finally the case of (primary) Hopf surfaces. We concentrate on a single connected component of the Teichm\"uller stack, cf. the discussion in \cite{LMStacks}. These surfaces $X_g$ are quotient of $\C^2\setminus\{0\}$ by the group generated by a contracting biholomorphism $g$ of $\C^2$ fixing $0$. Here $h^0$ takes the values $2$, $3$ and $4$. Since the function $h^0$ is upper semi-continuous, normal points correspond to the smallest value, that is $2$. Hopf surfaces with $h^0$ equal to $2$ are those with $g$ either non-linearizable and conjugated to
	\begin{equation}
		\label{normalform}
		(z,w)\longmapsto (\lambda z+w^p,\lambda^pw)
	\end{equation}
	for $\lambda\in\mathbb D^*$ and $p>0$ or linearizable with non-resonant eigenvalues, i.e. the two eigenvalues are not of the form $(\lambda,\lambda^p)$ for some $p>0$, see \cite{Wehler} for more details about classification and properties of Hopf surfaces. As a consequence, a normal point $X_g$ is completely determined by the determinant and the trace of the linear part $g_{lin}$ of $g$ and (a connected component of) the normal Teichm\"uller stack is the bounded domain $\mathbb D^*\times\mathbb D$ in $\C^2$. Given a family of Hopf surfaces $\pi : \mathcal X\to B$, we associate to it the morphism 
	\begin{equation}
		b\in B\longmapsto \left (\det g_{lin}(b),\frac{1}{2}\text{Tr }g_{lin}(b)\right )\in\mathbb D^*\times\mathbb D
	\end{equation}
	with $\pi^{-1}(b)\simeq X_{g(b)}$. Note that the Teichm\"uller stack has a much more complicated structure, far from being a manifold and with non-Hausdorff Teichm\"uller space, that is analyzed in \cite{clement}, see also \cite{Cortona}, and that we recall and make use of in Example \ref{exHopf3}.
\end{example}
\begin{proof}
	Start with the atlas of $\mathscr{T}(M,\mathcal{N})$ constructed in \cite{LMStacks}. Since $h^0$ is contant along $\mathcal{N}$, observe that it is given as an (at most) countable union of Kuranishi spaces, there is no need to make use of the fattening process of \cite{LMStacks}. Let us denote the corresponding symmetry groupoid with $\mathcal{T}_1\rightrightarrows\mathcal{T}_0$. It can be chosen as follows. Cover $\mathcal N$ with a(n) (at most) countable union of Kuranishi charts $\Xi_i:V_i\to K_i$. Then set
	\begin{equation*}
		\mathcal T_0=\bigsqcup K_i
	\end{equation*}
	and
	\begin{equation*}
		\mathcal T_1=\{(J,f)\in\text{Diff}^0(M,\sqcup \mathscr{K}_i)\mid J\cdot f\in \sqcup K_i\}
	\end{equation*}
	where $\mathscr{K}_i\to K_i$ are the Kuranishi families (compare with \eqref{atlasneigh}). We may decompose it as
	\begin{equation}
		\label{T1dec}
		\bigsqcup \mathcal T_{1,j}:=\bigsqcup\{(J,f)\in\text{Diff}^0(M,\mathscr{K}_j)\mid J\cdot f\in \sqcup K_i\}
	\end{equation}
	Now observe that $\mathcal{A}_{0,j}$ act fiberwise on $\mathcal{T}_{1,j}$ through \eqref{actionA0A1}; and that Lemma \ref{lemmaquotientgroupoid} generalizes immediately  to this context. We may thus define a fiberwise 
	 $\sqcup\mathcal{A}_{0,i}$-action \eqref{actionA0A1} on $\mathcal{T}_1$ decomposed as in \eqref{T1dec}, obtaining an \'etale quotient groupoid, say $\mathcal{N}_1\rightrightarrows\mathcal{T}_0$. Its stackification over $\mathfrak{S}$ is the normal Teichm\"uller stack proving i).
	
	By construction, there is a natural quotient map from each component of $\mathcal{T}_1$ to a component of $\mathcal{N}_1$ preserving the structure maps, hence yielding a groupoid homomorphism from $\mathcal{T}_1\rightrightarrows\mathcal{T}_0$ to $\mathcal{N}_1\rightrightarrows\mathcal{T}_0$. It induces a natural morphism from $\mathscr{T}(M)$ to $\mathscr{NT}(M)$. Since the morphism $\mathcal{T}_1\to\mathcal{N}_1$ is surjective and commutes with the source and target morphisms, a couple of points of $\mathcal{T}_0$ is source and target of a morphism of $\mathcal{T}_1$ if and only if it is source and target of a morphism of $\mathcal{N}_1$. This proves that the two stacks are associated to the same topological quotient space.
	
	Still by construction, the isotropy group of a point $X_0$ is given by the $\Aun$-mapping class group of $X_0$. Assume now that it is finite. Since $X_0$ is normal, it is not exceptional and the Kuranishi stack $\mathscr{A}_1$ and the local Teichm\"uller stack $\mathscr{T}(M,V)$ coincides. 
	Choose now automorphisms $f_1,\hdots, f_k$ of $\Aun$ so that their classes in $\text{Map}^1(X_0)$ generate $\text{Map}^1(X_0)$. Observe that the $f_i$'s may not fulfill the $\text{Map}^1(X_0)$-relations their classes do; but the $\text{Hol}_{f_i}$'s do and the group generated by the $\text{Hol}_{f_i}$'s is isomorphic to $\text{Map}^1(X_0)$. We may thus argue as in the proof of Theorem \ref{Kurstackorbifold}, and redefine a smaller open set $V$ and a smaller $K_0$  such that every $\text{Hol}_{f_i}$ is an isomorphism of $K_0$. Therefore, $\mathcal{A}_1/\mathcal{A}_0\rightrightarrows K_0$ is Morita equivalent to the translation groupoid  $\text{Map}^1(X_0)\times K_0\rightrightarrows K_0$. The point $X_0$ is an orbifold point, proving iv).
	
	If $X_0$ is bimeromorphic to a K\"ahler manifold, then $\text{Map}^1(X_0)$ is finite by \cite{Fujiki} and we may apply iv).
\end{proof}

\begin{remark}
	\label{rkholonomygroupoid}
	In other words, the action of $\D$ restricted to the set of normal points is a bonafide foliation and the normal Teichm\"uller stack is nothing else than its \'etale holonomy groupoid, cf. \cite{MM} and \cite[\S 6]{LMStacks}. 
\end{remark}

All the previous considerations and results apply to the $\Z$-normal points. Proposition \ref{propnormalopen} is still true with normal replaced with $\Z$-normal and Lemma \ref{lemmaquotientgroupoid} applies with $\mathcal{A}_\Z$ instead of $\mathcal{A}_1$. We may thus define the stack of $\Z$-normal points as the stackification of the étale groupoid $\mathcal{A}_\Z/\mathcal{A}_0\rightrightarrows K_0$. This is the {\slshape $\Z$-normal Teichm\"uller stack} $\mathscr{N}\mathscr{T}^\Z(M)$. Calling $\AZ$ mapping class group and denoting by $\text{Map}^\Z(X_0)$ the quotient of $\AZ$ by $\Azero$, we have

\begin{corollary}
	\label{corZnormal}
	The $\Z$-normal Teichm\"uller stack $\mathscr{N}\mathscr{T}^\Z(M)$ satisfies all the properties listed in Theorem {\rm \ref{thmnormal}} with the following obvious changes in the statements: $\mathscr{N}\mathscr{T}(M)$ is replaced with $\mathscr{N}\mathscr{T}^\Z(M)$, normal with $\Z$-normal, $\T$ with $\TZ$, and $\text{\rm Map}^1(X_0)$ with $\text{\rm Map}^\Z(X_0)$.
\end{corollary}

In Examples \ref{extorus2}, \ref{exK3} and \ref{exHopf2}, points are normal if and only if they are $\Z$-normal and there is no difference between $\mathscr{N}\mathscr{T}^\Z(M)$ and $\mathscr{N}\mathscr{T}(M)$. This is often the case.

\section{Exceptional points}
\label{secexc}

We now switch to the analysis of exceptional and $\Z$-exceptional points. By definition, there exist morphisms belonging to the Teichm\"uller stack but not to the Kuranishi stack at an exceptional point. We first analyze which type of morphisms they are both from the point of view of sequences of isomorphisms in Section \ref{pointstarget} and from the point of view of cycle spaces in Section \ref{secintrocycles}. Finally we give an example of a $\Z$-exceptional point in Section \ref{secexample}. It should be noted that neither K3 surfaces, Hirzebruch surfaces, nor Hopf surfaces exhibit exceptional points. In fact, for a long time, we did not know any example until we finally manage to construct the one presented in \ref{secexample}. It is however only $\Z$-exceptional, not exceptional. By Proposition \ref{propcompareZextoex}, the converse is not possible.

\subsection{Sequences of isomorphic structures in the Kuranishi space}
\label{pointstarget}
Let $(f_n)$ be a morphism of $(\mathscr{T}(M),J_0)$, that is a sequence of morphisms from $({\mathcal{X}'}_n\to B)$ to $(\mathcal{X}_n\to B)$ as explained above in Section \ref{proj}. Since $\mathcal{T}_{V_n}\rightrightarrows K_0\cap V_n$ is an atlas for $\mathscr{T}(M,V_n)$, each $f_n$ is obtained by gluing a cocycle of morphisms in $\mathcal{T}_{V_n}$ over an open cover of $B$. Such morphisms are local morphisms of the Kuranishi family. We are not interested in morphisms which act on the base $B$ as the identity. For such morphisms restrict to automorphisms of the fibers, especially of the central fiber, and are thus already encoded in $\mathscr{A}_1$, at least for $n$ big enough. Now the existence of a morphism $f_n$ acting non-trivially on the base is subject to the existence of two isomorphic distinct fibers of the Kuranishi family, that is to the existence of two distinct points in $K_0$ encoding the same complex manifold up to isomorphism. In the same way, the existence of sequences of morphisms $(f_n)$ acting non-trivially on the base is subject to the existence of sequences $(x_n)$ and $(y_n)$ of points in $K_0$ such that
\begin{enumerate}[i)]
	\item Both sequences $(x_n)$ and $(y_n)$ converge to the base point of $K_0$.
	\item For every $n$, the fibers of the Kuranishi family above $x_n$ and $y_n$ are isomorphic.
\end{enumerate}
In particular, there exists a sequence $(\phi_n)$ of $\D$ such that
\begin{equation}
	\label{phin}
	\forall n,\qquad x_n\cdot\phi_n=y_n
\end{equation}
We assume that, for all $n\geq 0$, the points $x_n$ and $y_n$ belongs to $V_n$ and the morphism $\phi_n$ to $\mathcal T_{V_n}$, where $(V_n)$ is a nested sequence as in \eqref{nestingseq}. We also set $V=V_0$.

Now, we are looking for {\slshape missing} morphisms in the Kuranishi stacks. In other words, we are looking for such sequences $(x_n)$ and $(y_n)$ with the additional property that $(x_n,\phi_n)$ does not belong to $\mathscr{A}_1$, that is $(x_n,\phi_n)$ is not $(V_n,\mathcal D_1)$-admissible.

In the sequel, we assume that $\mathcal T_V$ is reduced, replacing it with its reduction if needed. Assume that the sequence $(\phi_n)$ belongs to a fixed irreducible component $\mathcal C_0$ of $\mathcal T_V$.
We shall see that only special components, that we call exceptional (see Definition \ref{defex}), may contain morphisms that are not in the Kuranishi stack $\mathscr{A}_1$. To do that, we need to better understand the different types of $\mathcal T_V$-components.

Let $A$ be a connected component of $\Aun$. Then $A$ belongs to an irreducible component $\mathcal C_0$ of $\mathcal T_V$. We have

\begin{lemma}
	\label{lemmaunique}
	The component $\mathcal C_0$ is unique, that is there does not exist another irreducible component of $\mathcal T_V$ that contains $A$.
\end{lemma}

\begin{proof}
	Let $\mathcal C_1$ be an irreducible component of $\mathcal T_V$ that contains $A$. Then it contains all the extensions of the automorphisms of $\Aun$. More precisely it contains all the morphisms of $\mathcal T_V$ that are in a sufficiently small neighborhood of $A$. But this is also the case of $\mathcal C_0$, so the irreducible components $\mathcal C_0$ and $\mathcal C_1$ intersects on an open set, hence are equal.
\end{proof}

We may thus define
\begin{definition}
	\label{defA1type}
	We call {\slshape $A^1$-component} the irreducible component of $\mathcal T_V$ that contains $A$.
\end{definition}

\begin{remark}
	\label{rkconnexiteAcomp}
	Notice from the proof that the $s$-image of an $A^1$-component covers a neighborhood of an irreducible component of $K_0$. Also, if the intersection of a $A^1$-component with $\Aun$ is not connected but has a finite number of connected components, restricting it to $V_n$ for $n$ big enough it becomes connected. In other words, if the $A^1$-component contains two connected components of the intersection, its restriction to $V_n$ for $n$ large splits into two disjoint $A^1$-components. 
\end{remark}

However $A^1$-components do not contain missing morphisms. Indeed,

\begin{lemma}
	\label{lemmaA1notex}
	Let $(\phi_n)$ be a sequence \eqref{phin}. Assume it is contained in a $A^1$-component. Then for $n$ big enough, the restriction of the $A^1$-component to $V_n$ only contains $(V_n,\mathcal D_1)$-admissible morphisms.
\end{lemma}

\begin{proof}
	All extensions of automorphisms of $A$ are $(V_n,\mathcal D_1)$-admissible for $n$ big enough by \eqref{Azero} (in fact by the $\Aun$ version of \eqref{Azero}). Since an $A^1$-component contains all the morphisms of $\mathcal T_V$ that are in a sufficiently small neighborhood of $A$, we are done.
\end{proof}

There exists another type of irreducible component of $\mathcal T_V$, namely 

\begin{definition}
	\label{defex}
	An irreducible component $\mathcal C_0$ of $\mathcal T_V$ is called {\slshape exceptional} if
	\begin{enumerate}[i)]
		\item For $n$ large, its restriction to $V_n$ does not intersect $\Aun$,
		\item It contains a sequence \eqref{phin}.
	\end{enumerate}  
\end{definition}

In other words, $(J_0,J_0)$ does not belong to the $(s\times t)$-image of an exceptional component restricted to $V_n$ for large $n$, but belongs to its closure.

We have

\begin{lemma}
	\label{lemmaex}
	If $\mathcal T_V$ contains an exceptional component, then $X_0$ is an  exceptional point of the Teichm\"uller stack.
\end{lemma}

\begin{proof}
	As discussed above, if $\mathcal T_V$ contains an exceptional component, then the associated sequence \eqref{phin} is a sequence of $\mathcal T_V$ but not of $\mathscr{A}_1$ and $X_0$ is exceptional.
\end{proof}

The converse to Lemma \ref{lemmaex} is not true in general. If the number of irreducible components of $\mathcal T_V$ is infinite, there may exist wandering sequences \eqref{phin}, that is sequences \eqref{phin} such that every irreducible component of $\mathcal T_V$ contains at most a finite number of terms of the sequence and is not exceptional.

To have a better understanding of the situation, we make use of cycle spaces.


\subsection{Morphisms of the Teichm\"uller stack and cycle spaces}
\label{secintrocycles}

We consider the Barlet space of (relative) $n$-cycles of $\mathscr{K}^{red}_0\times\mathscr{K}^{red}_0$ for $\mathscr{K}^{red}_0$ the reduction of the Kuranishi family. Hence a cycle is a finite sum of compact analytic subspaces of some $X_t\times X_s$, for $X_t$ and $X_s$ fibers of the Kuranishi family. We only consider in this space irreducible components  
\begin{enumerate}[i)]
	\item that only contains singular cycles with both projection maps onto $X_t$ and $X_s$ of degree one; we call them {\slshape completely singular components}
	\item that contains at least the graph of a biholomorphism between two fibers which induces the identity in cohomology with coefficients in $\Z$ ; we call them {\slshape regular components}.
\end{enumerate}
We denote by $\mathscr{C}$ the union of completely singular and regular components.
In the Barlet space of $n$-cycles of $X_0\times X_0$, we also consider only the union of completely singular and regular components, that we denote by $\mathscr{C}_0$. 

Assume $\mathcal T_V$ is reduced. Every irreducible component of $\mathcal T_V$ injects in an irreducible component of $\mathscr C$. Just send $(J,f)$ to its graph as a cycle of $\mathscr K_0\times\mathscr K_0$. Indeed $\mathcal T_V$ encodes the regular cycles of $\mathscr{C}$. So examining $\mathscr C$ gives more information about $\mathcal T_V$.
Especially, a sequence \eqref{phin}, where each $\phi_n$ is an isomorphism between the fiber $\pi^{-1}(x_n)$ and the fiber $\pi^{-1}(y_n)$ of the Kuranishi family, defines a sequence $(\gamma_n)$ of $\mathscr{C}$. Then, several cases may occur
\begin{enumerate}[i)]
	\item up to passing to a subsequence, $(\gamma_n)$ converges to the graph of an automorphism $g$ of $\Aun$.
	\item up to passing to a subsequence, $(\gamma_n)$ converges to a singular cycle in $\mathscr{C}_0$.
	\item up to passing to a subsequence, $(\gamma_n)$ lives in a single component of $\mathscr{C}$ but does not converge to a cycle in $\mathscr{C}_0$.
	\item every irreducible component of $\mathscr{C}$ contains at most a finite number of terms of $(\gamma_n)$.
\end{enumerate}

If $X_0$ is K\"ahler, it is well known that only cases i) and ii) can occur. It is important however to keep in mind that we also consider the general case. Here is a classical example of case iii).

\begin{example}
	\label{exwarningrelative}
	Let $B$ be the subset of matrices of type
	\begin{equation}
		\label{mat}
		\begin{pmatrix}
			\lambda_1 &1\\
			0 &\lambda_2
		\end{pmatrix}
		\qquad\text{ with }\qquad \lambda_i\in\mathbb D^*\text{ for }i=1,2
	\end{equation}
	We associate to $B$ the family of Hopf surfaces 
	\begin{equation}
		\label{famille}
		\mathcal X:=\left (\mathbb C^2\setminus\{(0,0)\}\times B\right )/\mathbb Z
	\end{equation}
	where $p\in\mathbb Z$ acts on $(v,A)$ through
	\begin{equation}
		\label{action}
		p\cdot (Z,A):=(A^pZ,A)
	\end{equation}
	Two Hopf surfaces $X_A$ and $X_{A'}$, corresponding to $A$ and $A'$ of type \eqref{mat}, are biholomorphic if and only if the matrices $A$ and $A'$ are conjugated, thus if and only if they have the same eigenvalues. In particular, denoting by $\check A$ the matrix obtained from $A$ by inverting its eigenvalues, i.e.
	\begin{equation}
		\label{AA'}
		A=\begin{pmatrix}
			\lambda_1 &1\\
			0 &\lambda_2
		\end{pmatrix}
		\qquad\qquad
		\check A=\begin{pmatrix}
			\lambda_2 &1\\
			0 &\lambda_1
		\end{pmatrix}
	\end{equation}
	then $X_A$ and $X_{\check A}$ are biholomorphic. Moreover, denoting by $B^*$ the subset of $B$ formed by matrices with distinct eigenvalues, the isomorphism of $\mathbb C^2\setminus\{(0,0)\}\times B^*$ given by
	\begin{equation}
		\label{liniso}
		(Z,A)\longmapsto \left (\begin{pmatrix}
			(\lambda_1-\lambda_2)^{-1} &1-(\lambda_2-\lambda_1)^{-2}\\
			1 &(\lambda_1-\lambda_2)^{-1}
		\end{pmatrix}Z,
		\check A\right )
	\end{equation}
	descends as an isomorphism of the family $\mathcal X^*\to B^*$ obtained by restricting $\mathcal X\to B$ to $B^*$.
	
	The graphs of biholomorphisms between $X_A$ and $X_{\check A}$ given by \eqref{liniso} do not converge as $\lambda_1-\lambda_2$ tends to zero. Indeed, these graphs lift as the restriction to $\C^2\setminus\{0\}\times\C^2\setminus\{0\}$ of complex linear planes in $\C^4$. As $\lambda_1-\lambda_2$ tends to zero, they converge to $\{0\}\times\C^2$ in $\C^4$. But this graph is not contained in $\C^2\setminus\{0\}\times\C^2\setminus\{0\}$.
\end{example}

Going back to our four cases, we see that
case i) corresponds to a converging sequence in a $A^1$-component. Case ii) may occur in a $A^1$-component or in an exceptional component. To distinguish the two cases, we set

\begin{definition}
	\label{defexcycle}
	A singular cycle of $\mathscr{C}_0$ is called {\slshape exceptional} if it is the limit of a sequence of regular cycles $(\gamma_n)$ of $\mathscr{C}$ which are graphs of a sequence \eqref{phin} lying in an exceptional component.
\end{definition}

Case iii) may also occur in a $A^1$-component or in an exceptional component. To distinguish the two cases, we set

\begin{definition}
	\label{defvanishsequence}
	A sequence \eqref{phin} satisfying {\rm iii)} and lying in an exceptional component is called a {\slshape vanishing} sequence.
\end{definition}

Finally, case iv) is covered by the following definition.

\begin{definition}
	\label{defwanderingpoint}
	A sequence \eqref{phin} such that every irreducible component of $\mathscr{C}$ contains at most a finite number of graphs of terms of \eqref{phin} is called a {\slshape wandering sequence}.
\end{definition}

We may sum up the previous discussion by the following converse to Lemma \ref{lemmaex}.

\begin{proposition}
	\label{propex}
	A point $X_0$ of the Teichm\"uller stack is exceptional if and only if one of the following statements is fulfilled:
	\begin{enumerate}[{\rm i)}]
		\item There exists an exceptional cycle in $\mathscr{C}_0$.
		\item There exists a vanishing sequence in $\mathcal{T}_V$.
		\item There exists a wandering sequence in $\mathcal{T}_V$.
	\end{enumerate}
\end{proposition}

In the case $X_0$ is K\"ahler, thanks to classical finiteness properties of the cycle spaces recalled in Section \ref{Lieb}, we will be able to show much more precise statements about exceptional points in Sections \ref{secfinite} and \ref{secdistri}.

We note that it is in no way an exceptional property for $\mathscr{C}_0$ to contain a connected component of singular cycles. For example, if $\Aun$ is reduced to zero, then the cycles $X_0\times pt+pt\times X_0$ form such a component. But they usually do not correspond to exceptional cycles.

\subsection{The case of $\Z$-exceptional points}
\label{subsecZexc}
The contents of \S \ref{pointstarget} and \S \ref{secintrocycles} can be easily adapted to the case of $\Z$-Teichm\"uller stack and $\Z$-exceptional points. We feel free to use the corresponding definitions and results in this context with the obvious changes of notations. The reader should not be affected.

Besides, we note the following interesting comparison Proposition.

\begin{proposition}
	\label{propcompareZextoex}
	Let $X_0$ be an exceptional point of $\T$. Then $X_0$ is also a $\Z$-exceptional point of $\TZ$.
\end{proposition}

Proposition \ref{propcompareZextoex} can be equivalently stated as: a $\Z$-normal point is normal. The example described in \S \ref{secexample} shows that the converse is false.

\begin{proof}
	Let $X_0$ be an exceptional point of $\T$. Let $(\phi_n)$ be an associated non-convergent sequence of $\D$ satisfying \eqref{phin}. Assume by contradiction that $X_0$ is not a $\Z$-exceptional point of $\TZ$. Then, up to passing to a subsequence, $(\phi_n)$ must belong to an $A^\Z$-component which is not an $A^1$-component. But this means that, for $n$ big enough, $\phi_n$ belongs to $\DZ$ and not to $\D$. Contradiction.
\end{proof} 

\subsection{An example of a $\Z$-exceptional point}
\label{secexample}
In this Section, we give an explicit example of a $3$-fold which is a $\Z$-exceptional point of its Teichm\"uller stack. We start with a Blanchard manifold \cite{Blanchard} as revisited in \cite{Sommese} and \cite{CataneseBlanchard}.

Let $a\geq 1$ be an integer and set $W:=\mathcal{O}(a)\oplus\mathcal{O}(a)\to\mathbb{P}^1$. Choose two holomorphic sections $\sigma_i$ ($i=0,1$) of $\mathcal{O}(a)$ with no common zeros. Then the sections of $W$ defined by
\begin{equation}
	\label{sections}
	(\sigma_0,\sigma_1),\qquad (i\sigma_0,-i\sigma_1),\qquad (-\sigma_1,\sigma_0),\qquad (-i\sigma_1,-i\sigma_0)
\end{equation}
trivialize $W$ as a $\mathbb{R}^4$-bundle. In each fiber of $W$, these four sections generates an integer lattice. We denote by $\Gamma(t)$ the lattice above $t\in\mathbb{P}^1$. The group $\mathbb{Z}^4$ therefore acts on $W$ by translations along these lattices. The resulting quotient $3$-fold $\pi : \mathcal{X}\to\mathbb{P}^1$ is a deformation of complex $2$-tori above the projective line, with lattice $\Gamma(t)$ above $t$.

Given a holomorphic section $\tau$ of $W$, we define an automorphism $\phi_\tau$ of $\mathcal{X}$ by translating in each fiber along $\tau$. We thus define a map 
\begin{equation}
	\label{taumap}
	\tau\in H^0(\mathbb{P}^1,W)\longmapsto \phi_\tau\in \text{Aut}^0(\mathcal{X})
\end{equation} 
Let $G$ be the additive subgroup of $H^0(\mathbb{P}^1,W)$ generated by the four sections \eqref{sections}. We have
\begin{lemma}
	\label{lemmaexampleex1}
	The map {\rm \eqref{taumap}} induces a monomorphism from $H^0(\mathbb{P}^1,W)/G$ to $\text{\rm Aut}^0(\mathcal{X})$.
\end{lemma}

\begin{proof}
	We have 
	\begin{equation*}
		\phi_\sigma\circ\phi_\tau=\phi_\tau\circ\phi_\sigma=\phi_{\sigma+\tau}
	\end{equation*}
	hence \eqref{taumap} is a group morphism. Its kernel is given by the linear combinations over $\mathbb Z$ of the four sections \eqref{sections} so is equal to $G$. 
\end{proof}
Choose now $a+1$ distinct points in the base $\mathbb{P}^1$. We denote them by $t_0,\hdots,t_a$. We assume that there is no automorphism of $\mathcal{X}$ permuting non-trivially the $t_i$-fibers. This is achieved for a generic choice of points since moving slightly the points $t_i$ yields that any pair of $t_i$-fibers are non-isomorphic. We fix one point $P_i$ in each $t_i$-fiber. We call $\hat{\mathcal{X}}$ the manifold obtained from $\mathcal{X}$ by blowing up these $a+1$ points. Define
\begin{equation}
	\label{Sigma}
	\Sigma:=\{\tau\in H^0(\mathbb{P}^1,W)\mid\tau(t_i)\in \Gamma(t_i) \text{ for } i=0,\hdots ,a\}
\end{equation}

We prove
\begin{lemma}
	\label{lemmaexampleex2}
	The map {\rm \eqref{taumap}} induces an isomorphism
	\begin{equation}
		\label{taumap2}
		[\tau]\in \Sigma/G\longmapsto \hat\phi_\tau\in \text{\rm Aut}^\Z(\hat{\mathcal{X}})
	\end{equation} 
	and we have
	\begin{equation}
		\label{autxhat}
		\text{\rm Aut}^\Z(\hat{\mathcal{X}})\simeq\mathbb{Z}^{4a}\quad \text{\rm  and }\quad 
		\text{\rm Aut}^1(\hat{\mathcal{X}})=\{Id\}.
	\end{equation}
\end{lemma}

\begin{proof}
	An automorphism $\hat\phi$ of $\hat{\mathcal{X}}$ corresponds to an automorphism $\phi$ of $\mathcal{X}$ that permutes the points $P_i$. Now, an automorphism of $\mathcal{X}$ lifts to an automorphism of $W$ so sends fibers to fibers, and especially $t_i$-fibers to $t_i$-fibers. Since we assume that there is no automorphism of $\mathcal{X}$ permuting non-trivially the $t_i$-fibers, then both $\phi$ and $\hat\phi$ induce the identity on the base and fix the points $P_i$. Now, $\phi$ inducing the identity in cohomology with coefficients in $\Z$ means that it induces a translation in every fiber hence lies in the image of \eqref{taumap}. Since it fixes the points $P_i$, it is the identity in every $t_i$-fiber, hence lies in the image of $\Sigma$. So we have an epimorphism $\tau\in\Sigma\mapsto \hat\phi_\tau\in \text{\rm Aut}^\Z(\hat{\mathcal{X}})$ where $\hat\phi_\tau$ is the automorphism of $\hat{\mathcal{X}}$ corresponding to $\phi_\tau$. As above in Lemma \ref{lemmaexampleex1}, its kernel is given by the linear combinations over $\mathbb Z$ of the four sections \eqref{sections} so is equal to $G$. This proves that \eqref{taumap2} is an isomorphism.
	
	The space $H^0(\mathbb{P}^1,W)$ has dimension $2a+2$, each section being given by a pair of elements of $\mathbb{C}[X]$ of degree $a+1$. To belong to $\Sigma$, every polynomial must satisfy $a+1$ equations. More precisely, given any $(a+1)$-uple $(Q_0,\hdots,Q_a)$ with each $Q_i$ belonging to $\Gamma(t_i)$, there exists a unique section of $W$ passing through $Q_i$ at $t_i$. It is given as a pair of Lagrange interpolation polynomials. Hence $\Sigma$ identifies with the product $\Gamma(t_0)\times\hdots\times \Gamma(t_a)$ so is isomorphic to $\mathbb{Z}^{4a+4}$. The action of $G$ is equivalent to a transitive action on $\Gamma(t_0)$ and $\Sigma/G$ identifies with the product $\Gamma(t_1)\times\hdots\times \Gamma(t_a)$ so is isomorphic to $\mathbb{Z}^{4a}$. The same occurs therefore for $\text{\rm Aut}^\Z(\hat{\mathcal{X}})$.
	
	Finally, we note that the universal covering $\hat W$ of $\hat{\mathcal{X}}$ is $W$ blown up at each vertex of the lattices $P_i+\Gamma(t_i)$. Hence any non-trivial element $\hat\phi$ of $\text{\rm Aut}^\Z(\hat{\mathcal{X}})$ lifts to an automorphism of $\hat W$ that permutes non-trivially the blown-up points. Such an automorphism is not $C^\infty$-isotopic to the identity yielding that $\hat\phi$ is not $C^\infty$-isotopic to the identity. This achieves the proof of \eqref{autxhat}.
\end{proof}

Choose a point $P$ in $\hat{\mathcal{X}}$. Denote by $\hat{\mathcal{X}}_P$ the blow up of $\hat{\mathcal{X}}$ at $P$. Let $s=\pi(P)$ and define 
\begin{equation}
	\label{Sigma'}
	\Sigma'_s=\{\tau\in\Sigma\mid \tau(s)\in\Gamma(s)\}
\end{equation}
 As an immediate corollary, we obtain the following important result.

\begin{theorem}
	\label{thmexampleAut1Z}
	The $3$-fold $\hat{\mathcal{X}}_P$ satisfies
	\begin{equation}
		\label{autHatX}
		\text{\rm Aut}^\Z(\hat{\mathcal{X}}_P)\simeq \Sigma'_{s}/G\quad\text{\rm and }\quad \text{\rm Aut}^1(\hat{\mathcal{X}}_P)=\{Id\}
	\end{equation}
	In particular,
	\begin{enumerate}[\rm i)]
		\item We have 
		\begin{equation}
				\label{autHatX1}
			\text{\rm Map}^\Z(\hat{\mathcal{X}}_P)\simeq\text{\rm Aut}^\Z(\hat{\mathcal{X}}_P)\simeq\mathbb{Z}^{4a}
		\end{equation}
		and
			\begin{equation}
			\label{autHatX1bis}
			\text{\rm Map}^1(\hat{\mathcal{X}}_P)\simeq\text{\rm Aut}^1(\hat{\mathcal{X}}_P)=\{Id\}
		\end{equation}
		 for $s=t_a$.
		\item We have 
		\begin{equation}
				\label{autHatX2}
			\text{\rm Map}^\Z(\hat{\mathcal{X}}_Q)\simeq\text{\rm Map}^1(\hat{\mathcal{X}}_Q)\simeq\text{\rm Aut}^\Z(\hat{\mathcal{X}}_Q)=\text{\rm Aut}^1(\hat{\mathcal{X}}_Q)=\{Id\} 
		\end{equation} for $\pi(Q)$ a generic small deformation of $s$ in $\mathbb{P}^1$.
	\end{enumerate}
\end{theorem}

\begin{proof}
	The proof follows from that of Lemma \ref{lemmaexampleex2}. An element of $\text{\rm Aut}^\Z(\hat{\mathcal{X}}_P)$ is given by a section $\tau$ of $\Sigma$ such that $\phi_\tau$ is the identity on the $s$-fiber, that is $\tau$ belongs to $\Sigma'_s$. This defines an epimorphism from $\Sigma'_s$ to $\text{\rm Aut}^\Z(\hat{\mathcal{X}}_P)$. As in the previous cases, the kernel is $G$ proving the first equality of \eqref{autHatX}.  The second one is obtained as in the proof of Lemma \ref{lemmaexampleex2} by arguing that a non-trivial element of $\text{\rm Aut}^\Z(\hat{\mathcal{X}}_P)$ lifts to an automorphism of the universal covering that permutes non-trivially blown-up points, so is not $C^\infty$-isotopic to the identity. This also proves \eqref{autHatX1bis}. If $s=t_a$, we have $\Sigma'_s=\Sigma$ since every $\phi_\tau$ with $\tau\in\Sigma$ induces the identity in the $t_a$-fiber and \eqref{autHatX1} is a direct application of Lemma \ref{lemmaexampleex2}. Finally, define the set 
	\begin{equation}
		\label{Gamma'}
		\Gamma'_s:=\{\tau(s)\in \pi^{-1}(s)\mid \tau \in\Sigma\setminus G\}
	\end{equation}
	We claim that, at a generic point, $\Gamma'_s$ does not intersect $\Gamma(s)$, so $\Sigma'_s$ is reduced to $G$ yielding \eqref{autHatX2}. Indeed, choose $4a$ sections $\tau_i$ ($i=1,\hdots, 4a$) which generates $\Sigma$ as a $\mathbb Z$-module together with the four sections \eqref{sections}. In particular $\Sigma/G$ identifies with the group generated by the $\tau_i$. It is enough to prove that no non-trivial linear combination over $\mathbb{Z}$ of the $\tau_i(s)$ belong to $\Gamma(s)$ for a generic $s$. Assume by contradiction that, for every $s$ in a small disk of $\mathbb{P}^1$ disjoint from the set of points $t_i$, we can find some $i$ with $\tau_i(s)$ belonging to $\Gamma(s)$. By continuity, we may assume that this is the same $i$ for all $s$, restricting our disk if necessary. But still by continuity this implies that $\tau_i$ coincides over a disk with an element of $G$, that is, is an element of $G$. Contradiction.
\end{proof}
Let us put some more context around Theorem \ref{thmexampleAut1Z}. In \cite{Aut1}, we gave examples of non-Kähler $3$-folds with non trivial finite $\Aun$-mapping class group. Examples of surfaces, including projective ones, having this property or having non trivial finite $\AZ$-mapping class group were given in \cite{CatAut}, see also \cite{MN} for $\AZ$. We ask in the last section of \cite{Aut1} for examples with non trivial {\slshape infinite} $\Aun$-mapping class group, noting that, by Lieberman's result \cite{Lieberman}, see also \S \ref{Lieb}, they must be non-Kähler. At that time, we were already motivated by questions about the geography of the Teichm\"uller stack and were looking for points with non-trivial holonomy group, that is points with non trivial $\Aun$-mapping class group that admits arbitrary small deformations with trivial $\Aun$-mapping class group, see \cite[\S 2]{Aut1} and \S \ref{secholonomy} for the notion of holonomy group in this context. The $3$-folds of \cite{Aut1} enjoy this property but they only have finite holonomy group.

The manifolds $\hat{\mathcal{X}}$ and  $\hat{\mathcal{X}}_P$ for $P$ on the $t_a$-fiber are examples with infinite discrete $\text{Aut}^\Z$ automorphism group, more precisely with $h^0$ equal to zero but with infinite $\AZ$-mapping class group. Moreover, it follows from Theorem \ref{thmexampleAut1Z} that an arbitrary small deformation of $\hat{\mathcal{X}}_P$ has no non-trivial automorphisms inducing the identity in cohomology with coefficients in $\Z$ so that such a point has infinite $\Z$-holonomy group. Notice in particular that this is not an orbifold point of the $\Z$-Teichm\"uller stack although it has no non-zero global holomorphic vector fields. Once again, there is no such Kähler points. This shows that these non-K\"ahler $\hat{\mathcal{X}}_P$ have a more complicated local $\Z$-Teichm\"uller stack than Kähler points. More precisely,

\begin{corollary}
	\label{corexampleAut1Z}
	The $\Z$-Teichm\"uller stacks of both $\mathbb{S}^2\times (\mathbb{S}^1)^4$ at a point $\mathcal{X}$ and of $\mathbb{S}^2\times (\mathbb{S}^1)^4\sharp\mathbb{P}^3\sharp\hdots\sharp\mathbb{P}^3$ at a point $\hat{\mathcal{X}}_P$ satisfying \eqref{autHatX1} are not locally isomorphic to the $\Z$-Teichm\"uller stack of any Kähler manifold. 
\end{corollary}

\begin{proof}
	The sections \eqref{sections} trivialize $W$ over the reals, hence the smooth model of $\mathcal{X}$ is $\mathbb{S}^2\times (\mathbb{S}^1)^4$ and that of $\hat{\mathcal{X}}_P$ is obtained from it by doing the connected sum with $\mathbb{P}^3$ for every blown-up point. Then, the stabilizer of $\mathcal{X}$ and of $\hat{\mathcal{X}}_P$ are infinite discrete by Lemma \ref{lemmaexampleex2} and Theorem \ref{thmexampleAut1Z}, whereas the stabilizer of a Kähler point has a finite number of connected component by Lieberman, proving the statement.
\end{proof}

But we have more.

\begin{theorem}
	\label{thmexampleexc}
	Let $P\in\hat{\mathcal{X}}$ such that $s=\pi(P)$ is generic. Then, the $3$-fold $\hat{\mathcal{X}}_P$ is an exceptional point of its $\Z$-Teichm\"uller stack.
\end{theorem}

\begin{proof}
	Let $P\in\hat{\mathcal{X}}$ such that $s=\pi(P)$ is different from $t_a$ and such that $\Gamma'_s$ does not intersect $\Gamma(s)$. We have already shown in the proof of Theorem \ref{thmexampleAut1Z} that a generic $s$ satisfies $\Gamma'_s\cap\Gamma(s)=\emptyset$. Choose a section $\tau$ in $\Sigma\setminus G$. Choose also a real linear map that sends the four sections \eqref{sections} evaluated at $s$ to the canonical basis of $\mathbb{R}^4$. Moving slightly $s$ if necessary, we may assume that the image of the vector $\tau(s)$ through this linear map has coordinates $\mathbb{Z}$-linearly independent and $\mathbb{Z}$-linearly independent with $1$, otherwise, arguing as above, $\tau$ would be in $G$. Then, the classical Kronecker's Theorem asserts that it generates a dense subgroup in the torus $\mathbb{R}^4/\mathbb{Z}^4$. In other words, given any point $Q$ in the $s$-fiber, the translates of $Q$ by the $\mathbb{Z}$-multiples of $\tau(s)$ form a dense orbit of the $s$-fiber. Choose $R$ in the $s$-fiber such that $R-P$ is not a $\mathbb{Z}$-multiple of $\tau(s)$. Note that $R$ can be chosen arbitrarily close to the initial point $P$. Consider the trivial deformation family $\hat{\mathcal{X}}\times \pi^{-1}(s)\to \pi^{-1}(s)$. It has a tautological section that sends a point $M$ in the $s$-fiber to $(M,M)$, the first $M$ being considered as living in the $3$-fold $\hat{\mathcal{X}}$. Blow up this section. This gives a family $\mathscr{Y}\to \pi^{-1}(s)$ whose fiber above a point $M$ is $\hat{\mathcal{X}}_M$. The section $\tau$ defines an automorphism of the trivial family
	\begin{equation}
		\label{autfamily}
		\begin{tikzcd}[column sep=large]
			\hat{\mathcal{X}}\times \pi^{-1}(s)\arrow[d]\arrow[r,"\ \phi_\tau\times\tau(s)\ "]&\hat{\mathcal{X}}\times \pi^{-1}(s)\arrow[d]\\
			\pi^{-1}(s)\arrow[r,"\tau(s)"]&\pi^{-1}(s)
		\end{tikzcd}
	\end{equation}
	where $\tau(s)$ means translation by $\tau(s)$ in the torus $\pi^{-1}(s)$. It prserves the tautological section, hence induces a family automorphism
	\begin{equation}
		\label{autfamily2}
		\begin{tikzcd}[column sep=large]
			\mathscr{Y}\arrow[d]\arrow[r,"\ \phi_\tau\times\tau(s)\ "]&	\mathscr{Y}\arrow[d]\\
			\pi^{-1}(s)\arrow[r,"\tau(s)"]&\pi^{-1}(s)
		\end{tikzcd}
	\end{equation}
	Especially, this shows that all the fibers above the points of the $\tau(s)$-transla\-tion orbit of $R$ are isomorphic through a map inducing the identity in cohomology. By density, we may in particular extract from this a sequence of points $(R_n)$ in $\pi^{-1}(s)$ that converges to $P$ together with a family $(\phi_n)$ of diffeomorphisms inducing the identity in cohomology with $\phi_n$ inducing a biholomorphism between the $R_n$-fiber and the $R_{n+1}$-fiber. Note that each $\phi_n$ is a translation along the fibers of $\hat{\mathcal{X}}\to\mathbb{P}^1$ by some section $\sigma_n$ in $\Sigma\setminus G$ and that $\Vert \sigma_n\Vert$ tends to infinity in $H^0(\mathbb{P}^1,W)$ as $n$ goes to infinity. Look at the corresponding sequence of graphs. Even if it converges, its limit cannot belong to the closure of the $\text{Aut}^\Z$-component of $\hat{\mathcal{X}}_P$ since this group is reduced to the identity by Theorem \ref{thmexampleAut1Z}. Hence, $\hat{\mathcal{X}}_P$ is a $\Z$-exceptional point of its Teichm\"uller stack.
\end{proof}

\begin{remark}
	\label{rkZextoex}
	As already mentioned, it follows from Theorem \ref{thmexampleAut1Z} that none of these points are exceptional. This shows in particular that the converse to Proposition \ref{propcompareZextoex} is false. But it also left wide open the existence part of Conjecture \ref{mainconj}, point II. 
\end{remark}

\begin{corollary}
	\label{corwandering}
	The $\Z$-exceptional points of Theorem {\rm \ref{thmexampleexc}} are wandering points.
\end{corollary}
	
\begin{proof}
	This comes from the fact that $\phi_n$ is induced by $\hat{\phi}_{r_n\tau}$ for some section $r_n\tau$ with $r_n$ going to infinity. Now, $\hat{\phi}_{p\tau}$ and $\hat{\phi}_{q\tau}$ for $p\not = q$ induce distinct actions on the lattice of exceptional divisors of the $s$-fiber of the universal covering $\hat{\mathcal{W}}_P$ of $\hat{\mathcal{X}}_P$, hence induce distinct actions on $H^2(\hat{\mathcal{W}}_P,\Z)$. It follows that $\hat{\phi}_{p\tau}$ and $\hat{\phi}_{q\tau}$ are not isotopic for $p\not = q$ and that the graphs of $\hat{\phi}_{r_n\tau}$ run through an infinite number of connected components of the cycle space $\mathscr{C}$ yielding wanderingness.
\end{proof}

\begin{remark}
	\label{rk}
	It follows from the proof of Theorem \ref{thmexampleexc} that, given any $P\in\pi^{-1}(s)$, the set of points of $\pi^{-1}(s)$ above which the fiber of the family $\mathscr{Y}\to \pi^{-1}(s)$ is isomorphic to $\hat{\mathcal{X}}_P$ through a biholomorphism inducing the identity in cohomology is a countable dense subset of the base $\pi^{-1}(s)$. This is a weak version of the situation discussed in \cite{LMStacks}, Remark 11.8 and Problem 11.9. 
\end{remark}
It is interesting to have a closer look at the repartition of these $\Z$-exceptio\-nal points. 

\begin{corollary}
	\label{thmrepartitionZex}
	The subset
	\begin{equation}
		\label{eqZexZ}
		E_{\hat{\mathcal{X}}}:=\{P\in \hat{\mathcal{X}}\mid \hat{\mathcal{X}}_P\text{ is }\Z\text{-exceptional}\}
	\end{equation}
	satisfies the following properties
	\begin{enumerate}[\rm i)]
		\item It is dense in $\hat{\mathcal{X}}$.
		\item If $P$ belongs to $E_{\hat{\mathcal{X}}}$, then every point of the fiber $\pi^{-1}(\pi(P))$ belongs to $E_{\hat{\mathcal{X}}}$.
	\end{enumerate}
\end{corollary}

\begin{proof}
	Point i) is a direct consequence of Theorem \ref{thmexampleexc}. We have shown that for a dense subset of points $s$ in $\mathbb{P}^1$, every point above $s$ is $\Z$-exceptional. Hence $E_{\hat{\mathcal{X}}}$ contains an union of $\pi$-fibers whose projection is dense in $\mathbb{P^1}$, hence is dense. To prove ii), consider, as a slight variation of the construction used in the proof of Theorem \ref{thmexampleexc}, the family $\mathscr{Z}\to \hat{\mathcal{X}}$ obtained by blowing up the trivial family $\hat{\mathcal{X}}\times\hat{\mathcal{X}}\to\hat{\mathcal{X}}$ along the section $P\in\hat{\mathcal{X}}\mapsto (P,P)\in\hat{\mathcal{X}}$. The fiber at $P\in\hat{\mathcal{X}}$ is $\hat{\mathcal{X}_P}$. 
	
	Let $P\in E_{\hat{\mathcal{X}}}$. Then, there exists an exceptional sequence $(Q_n,R_n,\phi_n)$ in $\hat{\mathcal{X}}\times\hat{\mathcal{X}}\times\DZ$ with $\phi_n$ inducing a biholomorphism between $\hat{\mathcal{X}}_{Q_n}$ and $\hat{\mathcal{X}}_{R_n}$. A direct adaptation of Theorem \ref{thmexampleAut1Z} shows that $\phi_n$ is given by a translation in the $\pi$-fibers along elements of $\Sigma$ with the factor of translation at $\pi(P)$ going to infinity. Thus, $Q_n$ et $R_n$ must belong to the same $\pi$-fiber as $P$ and, letting $P'$ belong to the same $\pi$-fiber as $P$, the sequence $(P'-P+Q_n,P'-P+R_n,\phi_n)$ is also exceptional, showing that $P'$ is also $\Z$-exceptional.
	
\end{proof}

\begin{remark}
	\label{rkdensityZ}
	We do not know if the density property of $\Z$-exceptional points in the family $\mathscr{Z}$ used in the proof  of Theorem \ref{thmrepartitionZex} is true for the Kuranishi family of one ot these points. If yes, this would give a local positive answer to part II of Conjecture \ref{mainconj} for $\Z$-exceptional points.
\end{remark}

\section{Finiteness properties of the local Teichm\"uller stack in the K\"ahler setting}
\label{secfinite}

\subsection{Consequences of Lieberman's compacity result}
\label{Lieb}
In this section, we recall and apply a basic result on cycle spaces in the K\"ahler case, which is due to Lieberman \cite{Lieberman}. We state the relative version, which is adapted to our purposes.

\begin{proposition}
\label{Liebprop}
Let $\pi_i : \mathcal X_i\to B_i$ be smooth morphisms with compact K\"ahler fibers over reduced analytic spaces $B_i$ for $i=0,1$. Let $\mathcal Z\to E$ be a continuous family of relative cycles of $\mathcal X_0\times \mathcal X_1\to B_0\times B_1$. Assume that the projection of $\mathcal{Z}$ is included in a compact of $B_0\times B_1$.  Assume moreover that all cycles of $\mathcal Z$ are smooth, i.e are graphs of a biholomorphism from some fiber $(X_0)_t$ onto some fiber $(X_1)_{t'}$. Assume finally that they are graphs of biholomorphisms that induce the identity in cohomology with coefficients in $\Z$. Then,
\begin{enumerate}[\rm i)]
\item $E$ has compact closure in the space of cycles of $\mathcal X_0\times \mathcal X_1\to B_0\times B_1$ hence only meets a finite number of irreducible components of this space.
\item Let $\mathcal C$ be such a component. Then $\mathcal C$ contains a Zariski open subset $\mathcal C_0$ all of whose members are graphs of a biholomorphism inducing the identity in cohomology with coefficients in $\Z$ between a fiber of $\mathcal X_0$ and a fiber of $\mathcal X_1$.
\end{enumerate}
\end{proposition}

\begin{proof}
i) Let $(\omega^i_t)_{t\in B_i}$ be a continuous family of K\"ahler forms on the $\pi_i$-fibers ($i=0,1$). Let $M$ be the smooth model of $X_0$ and let $(J^i_t)_{t\in B_i}$ be a continuous family of integrable almost complex operators on $M$ such that $(X_i)_t=(M,J^i_t)$. For every $e\in E$, call $f_e: M\to M$ the biholomorphism from some fiber $(X_0)_t$ onto some fiber $(X_1)_{t'}$ corresponding to the cycle $Z_e$. We compute the volume of these cycles using the $\omega_t$. We have
$$
\text{Vol}(Z_e)=\int_M \left (\omega^0_t+f_e^*\omega^1_{t'}\right )^n=\int_M\left (\omega^0_t+\omega^1_{t'}\right )^n
$$
since $f_e$ induces the identity in cohomology hence $f_e^*\omega^1_{t'}$ and $\omega^1_{t'}$ differs from an exact form. Since the projection of $\mathcal{Z}$ is included in a compact of $B_0\times B_1$, we obtain that the volume of the $Z_e$ is uniformly bounded. It follows from \cite[Theorem 1]{Lieberman} that $E$ has compact closure in the cycle space of $\mathcal X_0\times\mathcal X_1$. Hence $E$ only meets a finite number of irreducible components of this cycle space.\vspace{3pt}\\
ii) Consider the family of cycles $\tilde{\mathcal C}\subset \mathcal X_0\times\mathcal X_1\to \mathcal C$. Since this map is proper and surjective, it is smooth on a Zariski open subset. Since some fibers are non singular, the generic fiber is non singular. The cycles above $E$ are submanifolds of some $(X_0)_t\times (X_1)_{t'}$ with projections $pr_i$ being bijective onto both factors. Hence, on a Zariski open subset of $\mathcal C$, every cycle enjoys such properties. So is the graph of a biholomorphism between a fiber of $\mathcal X_0$ and a fiber of $\mathcal X_1$. Finally all these graphs are smoothly isotopic hence induce the identity in cohomology with coefficients in $\Z$ since at least one of them has this property. 
\end{proof}

\begin{remark}
	\label{rkclassC}
	There exist relative versions of Lieberman's result in the class $(\mathscr{C})$. However, they do not apply to smooth morphisms with class $(\mathscr{C})$ fibers but to morphisms that are equivalent to K\"ahler morphisms in some sense, see see \cite[Def. 2.3]{Fujiki} or \cite[Def 4.1.9]{Barlet}. For that reason, they are not suited to our purposes and we stick to the K\"ahler setting.
\end{remark}
Setting $X_0=(M,J_0)$, considering the $\DZ$-orbit $\mathcal O$ of $J_0$ in $\mathcal I(M)$ and viewing $K_0$ as a local transverse section, we obtain a first interesting Corollary. 

\begin{corollary}
	\label{interlemma}
	If $V$ is small enough then $K_0$  intersects $\mathcal O$ only at $J_0$. 
\end{corollary}

\begin{proof}
	We assume that $V$ is small enough so that $K_0$ only contains K\"ahler points, cf. footnote \ref{fnoteKahler}. We also assume that $K_0$ is reduced, replacing it with its reduction if necessary. Take $\X_1=\mathscr{K}_0$ and $\X_0=X_0$ seen as a family over the point $\{J_0\}$. Let $E'$ be the subset of $K_0$ corresponding to complex structures $J$ in the orbit $\mathcal O$. 
	Now, $\mathcal O$ intersects transversely $K_0$ at $J_0$ by \eqref{Kurmap} but also at any intersection point. Since $\DZ$ has a countable topology, this intersection contains at most a countable number of points. Since we are only interested in what happens close to $J_0$, we may replace $E'$ with its intersection with a compact neighborhood of $J_0$ in $K_0$. Then for each $J\in E'$, choose some element $f_J$ of $\DZ$ mapping $J_0$ onto it. Set $E=\{J_0\}\times E'$ and let $\mathcal Z$ be the cycles corresponding to the graphs of the $f_J$. Apply Proposition \ref{Liebprop}. We conclude that $E$ meets a finite number of irreducible components of the cycle space of $X_0\times\mathscr{K}_0$, say $\mathcal C_1$,..., $\mathcal C_p$.
	
	 Still by Proposition \ref{Liebprop}, it follows that a Zariski open subset of each $\mathcal C_i$ only contains graph of biholomorphisms between $X_0$ and some $X_J$ with $J\in E'$. Hence each of these components only contains cycles in a fixed product $X_0\times X_J$ and $E'$ is a finite subset. Reducing $V$ if necessary, we may assume that $E'$ is just $\{J_0\}$ as wanted.
\end{proof}

\begin{remark}
	\label{rkcountable}
	Recall that we work with the $\DZ$-orbit and not with the $\text{Diff}^+(M)$-orbit. Corollary \ref{interlemma} does not say that, in the K\"ahler setting, there is no infinite sequence of points in $K_0$ converging to $J_0$ and all encoding $X_0$. It just says that, if it happens, only a finite number of the involved biholomorphisms with $X_0$ induce the identity in cohomology.
\end{remark}

As for the intersection of the $\DZ$-orbit of an arbitrary $J\in K_0$ with $K_0$, we have
\begin{corollary}
	\label{interlemma2}
	Let $J\in K_0$. If $K_0$ is small enough, then $K_0$  intersects the $\DZ$-orbit of $J$ into a finite number of leaves of the foliation of $K_0$. 
\end{corollary}

\begin{proof}
	This is completely analogous to the proof of Corollary \ref{interlemma}. Take $\X_1=\mathscr{K}_0$ and $\X_0=X_J$ and apply Proposition \ref{Liebprop}. This proves that the set of graphs of biholomorphisms between $X_J$ and another fiber of $\mathscr{K}_0$ is the union of the Zariski open subsets of regular cycles of a finite number of irreducible components of the cycle space of relative cycles of $X_J\times\mathscr{K}_0$. By definition (see Subsection \ref{Douady}), they project exactly onto a finite number of leaves of the foliation of $K_0$. 
\end{proof}

\begin{remark}
	Reducing $K_0$ {\slshape as a neighborhood of $J$}, we may assume that this intersection consists of a single leaf. However, since some leaves may accumulate onto $J_0$, we may be end with a neighborhood that does no more contain $J_0$. Since we do not want to lose our base point, we cannot replace finite number by single in the statement of Corollary \ref{interlemma2}. 
\end{remark}

We now analyse further the structure of $\mathscr{C}$ and $\mathscr{C}_0$ as defined in \S \ref{secintrocycles} and compare them with $\mathcal T^\Z_V$ from the one hand and with $\AZ$ from the other hand. As before, we assume that $K_0$ is reduced, replacing it with its reduction if necessary. 

Let us begin with $\mathscr{C}_0$. By Proposition \ref{Liebprop} applied to $\mathcal X_0=\mathcal X_1=X_0$, it has only a finite number of irreducible components and it is compact. Moreover, if a component contains a non singular cycle, it must be the graph of an automorphism and a Zariski open subset of it contains such graphs, hence it contains a connected component of $\AZ$ as Zariski open set. So the picture to have in mind is the following.
\begin{enumerate}
	\item [i)] $\AZ$ consists of a finite number of connected components. It always includes $\Azero$ and the components of $\Aun$. All the connected components are connected components of $\A$ and are all isomorphic to $\Azero$.
	\item [ii)] Either these connected component are all compact and each of them forms an irreducible component of $\mathscr{C}_0$; or they admit an analytic compactification in $\mathscr{C}_0$ by adding an analytic space of strictly lower dimension of singular cycles, each compactified connected component of $\AZ$ becoming a compact irreducible component of $\mathscr{C}_0$.
	\item[iii)] There may exist in $\mathscr{C}_0$ a finite number of additional irreducible components which contain only singular cycles.
\end{enumerate} 
Notice that two such irreducible components may intersect. For example, two distinct connected components of $\AZ$ may intersect once compactified in $\mathscr{C}_0$.

As for $\mathscr{C}$, applying Proposition \ref{Liebprop} to $\mathcal X_0=\mathcal X_1=\mathscr{K}_0$ yields that it has only a finite number of irreducible components and that its restriction above any compact subset of $K_0\times K_0$ is compact. Moreover, if a component contains a non singular cycle, it must be the graph of a biholomorphism between two fibers of the Kuranishi family and a Zariski open subset of it contains such graphs, hence it contains an irreducible component of (the reduction of) $\mathcal{T}^\Z_V$ as Zariski open set. And if an irreducible component contains a graph of an automorphism of $\AZ$, then it must contain an irreducible component of the closure of $\AZ$ in $\mathscr{C}_0$. So the picture to have in mind is the following.
\begin{enumerate}
	\item [i)] The reduction of $\mathcal{T}^\Z_V$ consists of a finite number of irreducible components. Each of them either contains a connected component of $\AZ$ or does not contain any automorphism of $X_0$. 
	\item [ii)] Each of these irreducible components injects as a Zariski open subset of an irreducible component of $\mathscr{C}$.
	\item[iii)] There may exist in $\mathscr{C}$ a finite number of additional irreducible components which contain only singular cycles.
\end{enumerate} 
Notice that two such irreducible components may intersect, that is, two distinct connected components of $\mathcal{T}^\Z_V$ may intersect once compactified in $\mathscr{C}$.

Consider now a sequence \eqref{phin}. We may state our second Corollary.

\begin{corollary}
	\label{2ndcor}
	Assume $X_0$ is K\"ahler. Then,
	\begin{enumerate}[\rm i)]
		\item Up to passing to a subsequence, we may assume that the graphs of the $\phi_n$ belong to a fixed irreducible component $\mathcal C$ of $\mathscr{C}$ and converges to a cycle $\gamma$. 
		\item The cycle $\gamma$ belongs to $\mathcal C$.
		\item The cycle $\gamma$ and thus the component $\mathcal C$ are exceptional if and only if $\gamma$ does not belong to the closure of an $A^\Z$-component of $\mathscr{C}_0$.
		\item Let $\mathcal C_0$ be the connected component of the intersection $\mathcal C\cap \mathscr{C}_0$ containing $\gamma$. The cycle $\gamma$ and thus the component $\mathcal C$ are exceptional if and only if every irreducible component of $\mathcal C_0$ consist only of singular cycles.
	\end{enumerate}
\end{corollary}

\begin{proof}
	Since there are only finitely many components in $\mathscr C$, we may assume, up to passing to a subsequence, that the graphs of the $\phi_n$ belong to a fixed irreducible component $\mathcal C$. The component $\mathcal C$ restricted to any compact neighborhood of $(J_0,J_0)$ in $K_0\times K_0$ is compact, hence the sequence converges, up to passing to another subsequence, and the limit cycle $\gamma$ belongs to $\mathcal C$. This proves i) and ii). If $\gamma$ is also obtained as a limit of graphs of elements of $\AZ$, then the distance between $\phi_n$ and $\AZ$ tends to zero when $n$ goes to infinity. Hence, for $n$ big enough, $(x_n,\phi_n)$ is $(V,\mathscr{D}_1)$-admissible and so is analytically isotopic to an element of $\AZ$. Therefore $\mathcal C$ contains a Zariski open subset formed by the graphs of a connected component of $\AZ$ and their extensions, so is an $A^\Z$-component. Conversely, if $\gamma$ does not belong to the closure of an $A^\Z$-component, the Zariski open subset of regular cycles of $\mathcal C$ does not intersect $\AZ$. This proves iii). Then iii) implies that $\mathcal C_0$ contains only singular cycles, since an irreducible component of $\mathscr{C}_0$ which is not completely singular is the closure of an $A^\Z$-component. Conversely, if $\mathcal C_0$ is completely singular, then $\gamma$ is not in the closure of an $A^\Z$-component, otherwise, arguing as above, the corresponding  $A^\Z$-component would belong to $\mathcal C_0$. Hence, by iii), $\mathcal C$ is exceptional.
\end{proof}

Let $\mathscr{C}'$ be the union of irreducible components of $\mathscr{C}$ containing a sequence \eqref{phin}.

\begin{corollary}
	\label{3rdcor}
	Assume $X_0$ is K\"ahler. Then,
	\begin{enumerate}[\rm i)]
		\item The number of irreducible components of the reduction of $\mathcal{T}^\Z_V$ is finite.
		\item If $V$ is a sufficiently small neighborhood of $J_0$ in $\I$, then there exists a natural bijection between the set of irreducible components of the reduction of $\mathcal{T}^\Z_V$ and that of $\mathscr{C}'$. In particular, every component of $\mathscr{C}'$ is either an $A^\Z$-component or an exceptional one.
		\item If $V$ is a sufficiently small neighborhood of $J_0$ in $\I$, then the intersection of an irreducible component of $\mathscr{C}'$ with $\mathscr{C}_0$ is connected. 
		\item If $V$ is a sufficiently small neighborhood of $J_0$ in $\I$, and $V'\subset V$ contains also $J_0$, then the natural inclusion of $\mathcal T_{V'}$ in $\mathcal{T}^\Z_V$ is a bijection between the corresponding sets of irreducible components.
	\end{enumerate}
\end{corollary}

\begin{proof}
	Assume $\mathcal{T}^\Z_V$ is reduced. Since every irreducible component of $\mathcal{T}^\Z_V$ injects in an irreducible component of $\mathscr C$ and $\mathscr C$ has only a finite number of components by K\"ahlerianity, this proves i). By finiteness, for a sufficiently small neighborhood $V$ of $J_0$, the point $J_0$ is adherent to the $s$-image and the $t$-image of every such component. Hence they contain a sequence \eqref{phin}. So the irreducible components of  $\mathcal{T}^\Z_V$ are in fact in 1:1 correspondence with the irreducible components of $\mathscr{C}'$, proving ii). By compacity of the components of $\mathscr{C}_0$ and finiteness of their number, the intersection of an irreducible component $\mathcal{C}$ with $\mathscr{C}_0$ has at most a finite number of connected components. Taking $V$ smaller if needed, the irreducible component $\mathcal C$ disconnects as a finite union of irreducible components of $\mathscr{C}'$ each of them intersecting $\mathscr{C}_0$ in a single connected component. Finally iv) follows from ii) and from the finiteness of components. When restriction to a smaller $V'\subset V$, the number of components of both $\mathcal T_{V'}$ and $\mathscr{C}'$ decrease. By finiteness, for a small enough $V$, this number stays constant when restricting to smaller $V'\subset V$ and only counts the components that contain a sequence $\eqref{phin}$.
\end{proof}
As a consequence of Corollary \ref{3rdcor}, we do not need to consider the full target germification in the K\"ahler case. It is enough to look at $\mathscr{T}(M,V)$ for a fixed small enough $V$ since restricting to smaller neighborhoods of $0$ will not change the number of components of $\mathcal{T}_V$.

Moreover, we have

\begin{corollary}
	\label{4thcor}
	Assume $X_0$ is K\"ahler. Then,
	\begin{enumerate}[\rm i)]
		\item There is no wandering sequence \eqref{phin} with each $\phi_n$ belonging to a different component of $\mathcal{T}_V$.
		\item There is no vanishing cycle.
	\end{enumerate}
\end{corollary}
\noindent and
\begin{corollary}
	\label{5thcor}
	Assume $X_0$ is K\"ahler. Then, the following statements are equivalent
	\begin{enumerate}[\rm i)]
			\item $X_0$ is exceptional.
			\item There exists an exceptional component.
			\item There exists an exceptional cycle.
	\end{enumerate}
\end{corollary}

\begin{proof}[Proof of Corollaries \ref{4thcor} and \ref{5thcor}] This is essentially a reformulation of what preceeds.
Corollary \ref{4thcor} is a direct consequences of point i) of Corollary \ref{2ndcor}. Corollary \ref{5thcor} follows then from Proposition \ref{propex}. 
\end{proof}
In the sequel, we assume that $V$ is small enough so that 
\begin{enumerate}[i)]
	\item It only contains Kähler structures.
	\item Items ii), iii) and iv) of Corollary \ref{3rdcor} are valid.
\end{enumerate}
%
%
%
%

\begin{remark}
	\label{rkcompex}
	Let $\mathcal C_0$ be an exceptional component of $\mathscr{C}_0$. Then, although every cycle of $\mathcal C_0$ is singular, not every cycle of $\mathcal C_0$ is exceptional. First of all, we must remove the possibly non-empty intersection of $\mathcal C_0$ with the closure of the components of $\AZ$. But in the remaining Zariski open set $\mathcal U$ of $\mathcal C_0$, the exceptional cycles are those lying in the intersection with the finite union of irreducible components of $\mathscr{C}'$. This intersection forms an analytic subspace of $\mathcal U$ that can be strict.
\end{remark}

\begin{remark}
	\label{rksingular}
	A singular cycle of $\mathscr{C}_0$ is either the graph of a bimeromorphic mapping of $X_0$ or the sum of two cycles $\gamma_1+\gamma_2$, such that $\gamma_i$ maps bimeromorphically onto $X_0$ through the $i$-th projection of $X_0\times X_0$ to $X_0$. This may include the case $X_0\times \{*\}+\{*\}\times X_0$.
\end{remark}

\subsection{The morphism $\mathscr{A}_1$ into $\mathscr{T}(M,V)$ in the K\"ahler setting}
\label{secfinitemorphism}

The following result is a refinement of Theorem \ref{finitethm} in the K\"ahler case.
\begin{theorem}
	\label{thmfinite2}
	Assume that $X_0$ is K\"ahler. Then, the natural inclusion of $\mathscr{A}_0$ into $\mathscr{T}(M,V)$, resp. of $\mathscr{A}_1$ into $\mathscr{T}(M,V)$, as well as that of $\mathscr{A}_0$, resp. $\mathscr{A}_\Z$ into $\mathscr{T}^\Z(M,V)$, are finite and \'etale morphisms of analytic stacks.
\end{theorem}

\begin{proof}
	Just combine Theorem \ref{finitethm} and Corollary \ref{2ndcor}, noting that the cardinal of the fiber at $X_0$ of these \'etale morphisms is less than the number of connected components of $\mathcal{T}^\Z_V$. 
\end{proof}

	Let us compute the fibers of the inclusion of $\mathscr{A}_0$ into $\mathscr{T}(M,V)$, resp. of $\mathscr{A}_1$ into $\mathscr{T}(M,V)$ for $X_0$ K\"ahler. Observe that this means computing $g(\mathcal X)$, resp. $g_1(\mathcal X)$, for $\mathcal X$ a family $X_J\to\{J\}$ for $J\in K_0$. For $J=J_0$, it follows from Corollary \ref{interlemma} that there is no point $J$ in $K_0$ such that $X_J$ is $\D$-biholomorphic to $X_0$. Hence, we just have to consider isomorphisms of the family $\mathcal X$, that is automorphisms of $X_0$. We deduce from that the equalities
	\begin{equation}
		\label{calculgptbase}
		g(\mathcal X)=\text{Card }\sharp\Aun=\text{Card } (\Aun/\Azero)
	\end{equation} 
	and
	\begin{equation}
		\label{calculg1ptbase}
		g_1(\mathcal X)=\text{Card } (\Aun/\Aun)=1
	\end{equation} 
	Both computations \eqref{calculgptbase} and \eqref{calculg1ptbase} remain the same for $\mathcal X\to B$ a $\mathscr{A}_1$-family\footnote{ that is a family isomorphic to a family obtained by gluing local pull-backs of the Kuranishi family $\mathscr{K}_0\to K_0$ through a cocycle of morphisms belonging to $\mathcal A_1$.} with connected base and at least one fiber biholomorphic to $X_0$. Indeed, let $b\in B$ such that the $b$-fiber is isomorphic to $X_0$. Then, the family $\mathcal X\to B$ is locally isomorphic at $b$ to $u^*\mathscr{K}_0$ for $u$ a holomorphic map defined on a neighborhood of $b\in B$ with values in $K_0$. Every isomorphism of the family induces an automorphism of the $b$-fiber, that is of $X_0$. Using this point $b$ as $b_0$ in \eqref{RepFP} and \eqref{RepFP1}, and noting that, given any $J\in K_0$, there exists a morphism starting from $J$ in any connected component of $\mathcal A_1$, we are done. 
	
We go back to the case of $\mathcal X$ being a family $X_J\to\{J\}$ for $J\in K_0$. This time, we assume that $J$ is not $J_0$. The intersection of the $\D$-orbit of $J$ with $K_0$ may be positive dimensional, but it consists of a finite number of leaves of the foliation of $K_0$ by Corollary \ref{interlemma2}. Each leaf corresponds to a connected subset of $\mathcal T_V$. Now, since the maps $\text{Hol}_f$ defined in \eqref{Autaction} preserves the foliation of $K_0$, there is an action of $\sharp \Aun$ on the finite set of leaves corresponding to $J$. Then $g(\mathcal X)$ is given by the number of $\sharp\Azero$-orbits of leaves and $g_1(\mathcal X)$ by the number of $\sharp\Aun$-orbits of leaves (recall Remark \ref{rkleavesconn}). 

Let us say that a point $J\in K_0$ is {\slshape generic} if the finite set of leaves corresponding to $J$ has the same cardinal than $\sharp\mathcal T_V$. Equivalently, it is generic if every connected component of $\mathcal T_V $ contains a morphism with source $J$. Then $g(\mathcal X)$ is given by the number of $\sharp\Azero$-orbits in $\sharp\mathcal T_V$ and $g_1(\mathcal X)$ by the number of $\sharp\Aun$-orbits in $\sharp\mathcal T_V$, that is are maximal. This remains true for a family over a connected base with all fibers generic.

We recollect the previous computations in the following proposition. Of course similar calculations hold for $\mathscr{A}_\Z$-families.

\begin{proposition}
	\label{propfiber}
	Assume $X_0$ is K\"ahler. Let $\mathcal X\to B$ be a $\mathscr{A}_1$-family over a connected base. Then,
	\begin{enumerate}[\rm i)]
		\item Assume that at least one fiber of $\mathcal X$ is biholomorphic to $X_0$. Then \eqref{calculgptbase} and \eqref{calculg1ptbase} hold.
		\item Assume that all the fibers are generic. Then $g(\mathcal X)$ is given by the number of $\sharp\Azero$-orbits in $\sharp\mathcal T_V$ and $g_1(\mathcal X)$ by the number of $\sharp\Aun$-orbits in $\sharp\mathcal T_V$, that is are maximal.
	\end{enumerate}
\end{proposition}

%

\subsection{Universality of the Kuranishi stacks revisited}
\label{secuniversal}
By definition, the Teichm\"u\-ller stack $\mathscr{T}(M)$ is universal for reduced $M$-deformations, that is any such family is obtained through an analytic mapping from the base of the family to $\mathscr{T}(M)$ and this mapping is unique up to unique isomorphism. The same is true for $\mathscr{T}(M,V)$ and $V$-families, for $\mathscr{T}^\Z(M)$ and $\Z$-reduced $M$-deformations, for $\mathscr{T}^\Z(M,V)$ and $\Z$-reduced $V$-families. This is one of the main interests in using the stack formalism.

By germifying as in Section \ref{universal}, we may consider the germ of $\mathscr{T}$, resp. of $\TZ$, at a point $J\in\mathcal I(M)$. Then $(\mathscr{T}(M),J)$, resp. $(\mathscr{T}^\Z(M),J)$, is universal for germs of reduced, resp. $\Z$-reduced, $M$-deformations of $X_J$ (compare with Corollary \ref{thmuniversal1}).

Theorem \ref{thmfinite2} 
allows us to go beyond statements on germs and to prove similar results for the Kuranishi stack of a K\"ahler, non-exceptional point. In return, this gives a geometric interpretation of these results.

\begin{proposition}
	\label{propunivK}
	Assume $X_0$ K\"ahler. Then, the Kuranishi stack $\mathscr{A}_1$ is universal for reduced, $V$-families if and only if the base point $X_0$ is not exceptional.
	
	
	The previous statement still holds true replacing $\mathscr{A}_1$ with $\mathscr{A}_\Z$, reduced with $\Z$-reduced, $\Aun$ with $\AZ$.
\end{proposition}

 Now, Proposition \ref{propunivK} tells us that an exceptional point is a point whose Kuranishi stack lacks of universality. Firstly, this lack of universality is a lack of completeness. It is possible that some reduced, $V$-families do not belong to $\mathscr{A}_1$.  Secondly, it is a lack of unicity. Indeed, given $\mathcal{X}\to B$ a family belonging to $\mathscr{A}_1$, the number $g_1(\mathcal X)$ gives exactly the number of different ways for obtaining $\mathcal X$ from $\mathscr{A}_1$. It is interesting however to observe from Proposition \ref{propfiber} that $g_1(\mathcal X)$ is equal to one for all $\mathscr{A}_1$-families with a fiber biholomorphic to $X_0$. Hence we automatically have universality for this restricted type of families.
 
%
 
 \begin{proof}
 	Since the Teichm\"uller stack $\mathscr{T}(M,V)$ enjoys both universal properties of Proposition \ref{propunivK}, so does $\mathscr{A}_1$ when $X_0$ is K\"ahler  and not exceptional.
 	
 	If $X_0$ is exceptional, then some fibers of the morphism $\mathscr{A}_1\to\mathscr{T}(M,V)$ contain several points. In other words, the uniqueness property of universality is not true for some families $X_J\to \{J\}$.

 	The proof in the $\mathscr{A}_\Z$ case follows exactly the same line of arguments.
 \end{proof}

 	We go back to statements about germs. Given $J\in K_0$, the germ of Kuranishi stack $(\mathscr{A}_1,J)$ is universal for germs of reduced, $M$-deformations of $X_J$ if and only if the isotropy group at $J$ is $\text{\rm Aut}^1(X_J)$. Let us make precise that, in this statement, $(\mathscr{A}_1,J)$ is the Kuranishi stack of $X_0$ but germified at $X_J$, for $J\in K_0$. It is not the Kuranishi stack of $X_J$.
 	
 	This comes from the fact that, since $\mathscr{K}_0\to K_0$ is complete at any point $J\in K_0$, the only requirement to have universality of $(\mathscr{A}_1,J)$ is to have all the morphisms of germs of reduced $M$-deformations of $X_J$, that is that the isotropy group at $J$ of $\mathscr{A}_1$ is $\text{Aut}^1(X_J)$ (compare with the proof of Corollary \ref{thmuniversal1}).
 	
 	Let us now deal with the $\mathscr{A}_0$-case. Here there is no condition for the universality property at each point. To be more precise,
 	assume $X_0$ K\"ahler. If $V$ is small enough, then, for all $J\in K_0$, the germ of Kuranishi stack $(\mathscr{A}_0,J)$ is universal for germs of $0$-reduced, $M$-deformations of $X_J$.
 	
 	To prove this, we argue as above, and we obtain that, given $J\in K_0$, the germ of Kuranishi stack $(\mathscr{A}_0,J)$ is universal for germs of $0$-reduced, $M$-deformations of $X_J$ if and only if the isotropy group at $J$ is $\text{Aut}^0(X_J)$. But Lemma \ref{lemmaautcontained} shows that this is always the case.
 		

\subsection{The Teichm\"uller stack as an orbifold}

As second application of Theorem \ref{thmfinite2}, we deal with the orbifold case.

\begin{theorem}
	\label{mainorbifoldthm}
	Assume $X_0$ K\"ahler. Then, the following statements are equivalent
	\begin{enumerate}[\rm i)]
		\item The exists some open set $V$ of $\I$ such that $\mathscr{T}(M,V)$ is an orbifold.
		\item The exists some open set $V$ of $\I$ such that $\mathscr{A}_1$ is an orbifold and $X_0$ is not exceptional.
		\item $\Aun$ is finite and $X_0$ is not exceptional.
		\item $\Azero$ is trivial and $X_0$ is not exceptional.
	\end{enumerate}
\end{theorem}
\noindent Remark \ref{precisestatement} applies also here for the choice of $V$. We also have

\begin{corollary}
	\label{mainorbifoldthmZcase}
	Assume $X_0$ K\"ahler. Then, the following statements are equivalent
	\begin{enumerate}[\rm i)]
		\item The exists some open set $V$ of $\I$ such that $\mathscr{T}^\Z(M,V)$ is an orbifold.
		\item The exists some open set $V$ of $\I$ such that $\mathscr{A}_\Z$ is an orbifold and $X_0$ is not $\Z$-exceptional.
		\item $\AZ$ is finite and $X_0$ is not $\Z$-exceptional.
		\item $\Azero$ is trivial and $X_0$ is not $\Z$-exceptional.
	\end{enumerate}
\end{corollary}

\begin{proof}
	We only prove Theorem \ref{mainorbifoldthm}. Assume ii). 
	Then, $\mathscr{T}(M,V)$ is isomorphic to $\mathscr{A}_1$, so is an orbifold and i) is proved. Assume i). Then a finite group acts on $K_0$ stabilizing $J_0$ and $\mathcal T_V$ encodes the orbits of this action. As a consequence, the leaves of the foliation of $K_0$ are $0$-dimensional, hence, by Theorem 2 of \cite{ENS}, the dimension of the automorphism group of the fibers of the Kuranishi family is constant. Combining Proposition \ref{propfiber} and Remark \ref{rkleavesconn}, this implies that this group has the same cardinal as $\sharp\mathcal T_V$. Now, if $X_0$ is exceptional, the finite group $\sharp\mathcal T_V$ is not the stabilizer of the base structure $X_0$, since the exceptional components do not yield automorphisms at $0$. So $X_0$ is not exceptional, and ii) follows.
	This proves that i) and ii) are equivalent. Then, ii) and iii) are equivalent by Corollary \ref{thmuniversal1}. Finally, iii) and iv) are equivalent because of the fact that $\Azero$ has finite index in $\Aun$ in the K\"ahler case.
\end{proof}

\section{Distribution of exceptional points}
\label{secdistri}
Our next task is to understand how the exceptional points are distributed, that is in which sense they are rare. Let $V_K$ be the open set of K\"ahler points of $\I$ (recall footnote \ref{fnoteKahler}). We shall prove

\begin{theorem}
	\label{2ndmainthm}
	The closure of exceptional points of the Teichm\"uller stack $\mathscr{T}(M,V_K)$ of K\"ahler structures form a strict analytic substack $\mathscr{E}(M)$ of $\mathscr{T}(M, V_K)$. 
	
	The same is true for $\Z$-exceptional points that form  a strict analytic substack $\mathscr{E}^\Z(M)$ of $\mathscr{T}^\Z(M, V_K)$. 
\end{theorem}
\noindent and its immediate Corollary
\begin{corollary}
	\label{Cor2ndmain}
	Normal K\"ahler points fill a Zariski open substack of the Teichm\"uller stack $\mathscr{T}(M,V_K)$ of K\"ahler structures.
	
	$\Z$-normal K\"ahler points fill a Zariski open substack of the $\Z$-Teichm\"uller stack $\mathscr{T}^\Z(M,V_K)$ of K\"ahler structures.
\end{corollary}
Theorem \ref{2ndmainthm} and Corollary \ref{Cor2ndmain} form the best results we are able to prove in relation with main Conjecture \ref{mainconj}, point I. 

Note that the closure of exceptional points, resp. the set of normal points, form a strict analytic subspace, resp. a Zariski open set of any atlas of $\mathscr{T}(M,V_K)$. Since any Kuranishi space $K_0$ of a point in $V_K$ is a local atlas of $\mathscr{T}(M,V_K)$, this means that the closure of exceptional points, resp. the set of normal points, form a strict analytic subspace, resp. a Zariski open set of any $K_0$. 
\begin{remark}
	\label{Isubspace}
	The closure of exceptional points, resp. $\Z$-exceptional points, also form a strict analytic subspace of $V_K$.
\end{remark}

\begin{remark}
	\label{topgenericity}
	In particular, the set of normal points is dense in the K\"ahler Teichm\"uller space $V_K/\D$ and contains an open set. However, due to the non-Hausdorff topology this space may have, this may be a misleading statement. For example, if $M$ is $\mathbb S^2\times\mathbb S^2$, then the (K\"ahler) Teichm\"uller space of $M$, as a set, is $\mathbb Z$, a point $a\in\mathbb Z$ encoding the Hirzebruch surface $\mathbb F_{2a}$\footnote{ The surfaces $\mathbb F_{2a}$ and $\mathbb F_{-2a}$ are isomorphic, but not through a biholomorphism isotopic to the identity.}. Now, the topology to put on $\mathbb Z$ has for (non trivial) open sets $\{0\}$, $\{0,1\}$, $\{-1,0\}$ and so on, cf. \cite{LMStacks}, Examples 5.14 and 12.6. Hence $0$ is an open and dense subset of the Teichm\"uller space.
\end{remark}

\begin{proof}
We only prove the result for exceptional points. The proof can be easily modified to treat the case of $\Z$-exceptional points.

The atlas \eqref{atlasneigh} being an atlas of a neighborhood of $X_0$ in the Teichm\"uller stack, it contains all the information we need to decide which points close to $X_0$ are exceptional.  

Indeed, pick a point $X_J$ in $K_0$. Assume it is exceptional. By our assumptions on $V$\footnote{ see the end of Section \ref{Lieb}.}, every irreducible component of $\mathscr{C}$ contains a cycle of $X_0\times X_0$, hence the set $\mathscr{C}_J$ of cycles of $X_J\times X_J$ which are limits of cycles of $\mathscr{C}$ is included in $\mathscr{C}'$, the subset of components containing a sequence \eqref{phin}, since the other components are formed by $(V,\mathscr{D}_1)$-admissible morphisms and their degenerations. As a consequence, $J$ is exceptional if and only if there exists a component of $\mathscr{C}'$ whose subset of cycles above $J$ is non-empty and contains a connected component with only singular cycles by Corollary \ref{2ndcor}. Let $S$ be the analytic set of singular cycles of $\mathscr{C}'$. Let $p$ denote the projection of $\mathscr{C}'$ to $K_0\times K_0$. This is a proper map, since we are in the K\"ahler setting. We thus have that the set of exceptional points is equal to
\begin{equation}
	\label{EC}
		E=\bigcup_{\mathcal{C}\in\text{Irr }\mathscr{C}'}\{J\in K_0\mid p^{-1}(J,J)\cap\mathcal{C}\not =\emptyset
		\text{ and }(p^{-1}(J,J)\cap\mathcal{C})_0\subset S\}
\end{equation}
where $\text{Irr }$ denotes the set of irreducible components and $_0$ means that some connected component of $p^{-1}(J,J)\cap\mathcal{C}$ is included in $S$. We notice that, by our assumptions on $V$, the intersection $p^{-1}(J_0,J_0)\cap\mathcal{C}$ is connected. However, we cannot ensure that this is still true for any point $J$. Let $E^c$ be the closure of \eqref{EC}.

We claim that $E^c$ is an analytic subspace of $K_0$. To see that, we first embed $E$ in $K_0\times K_0$ through the diagonal embedding of $K_0$. We still call $E$ the image. Let $\mathcal C$ be a component of $\mathscr{C}'$ and let $E_{\mathcal C}$ denote the $\mathcal C$-component of \eqref{EC}. Let $(J_1,J_1)$ be a point of $K_0\times K_0$. Let $W$ be an open neighborhood of $(J_1,J_1)$ in $K_0\times K_0$. If $p^{-1}(J_1,J_1)\cap\mathcal{C}$ is empty, then $p^{-1}(W)\cap\mathcal{C}$ is empty for $W$ small enough. So let us assume that $p^{-1}(J_1,J_1)\cap\mathcal{C}$ is not empty. We decompose $p^{-1}(J_1,J_1)\cap\mathcal{C}$ into connected components $(p^{-1}(J_1,J_1)\cap\mathcal{C})_i$ for $i$ between $0$ and $k$. Let $A_i$ be open neighborhoods of $(p^{-1}(J_1,J_1)\cap\mathcal{C})_i$ in $\mathcal C$ such that the union of all $A_i$ is exactly $p^{-1}(W)\cap\mathcal{C}$. We assume that $(J_1,J_1)$ is generic in the sense that the intersection $A_i\cap p^{-1}(J,J)$ is still connected for $(J,J)\in W$ if non empty. Notice that the restriction of $p$ to some $A_i$ is still a proper map. We may decompose $E_{\mathcal C}\cap W$ as follows.
\begin{equation}
	\label{constructibleset}
	\begin{aligned}
		E_{\mathcal C}\cap W&=
		\bigcup_{0\leq i\leq k}\{
		(J,J)\in W\mid p^{-1}(J,J)\cap A_i\not =\emptyset\\
		&\qquad\qquad\qquad\text{ and }p^{-1}(J,J)\cap A_i\subset S
		\}\\
		&=\bigcup_{0\leq i\leq k} \Delta\cap (p(A_i)\setminus p(A_i\setminus S))
	\end{aligned}
\end{equation}
where $\Delta$ is the diagonal of $K_0\times K_0$. 
%
 Hence \eqref{constructibleset} is a constructible set so its closure in $W$ is an analytic set of $\Delta\cap W\cap\mathcal{C}$. But its closure is just $E^c_{\mathcal C}\cap W$. So we obtain a chart of analytic subspace for $E^c$ at every generic point. Assume now that $(J_1,J_1)$ is not generic. Then we may perform exactly the same construction but the resulting constructible set \eqref{constructibleset} may forget some exceptional points. Now, let $a$ be a positive integer. The set of points $(J,J)$ of $W$ where $p^{-1}(J,J)\cap A_i$ has exactly $a$ connected components is constructible, see \cite{Stacksproject}, Lemma 37.28.6 in an algebraic context. So is its pull-back by $p$ in $A_i$. Looking at its connected components, we may decompose $A_i$ into a finite set of constructible sets $A_{ij}$ on which $p$ has connected fibers. Then we obtain the correct decomposition
\begin{equation}
	\label{constructibleset2}
		E_{\mathcal C}\cap W=\bigcup_{i,j} \Delta\cap (p(A_{ij})\setminus p(A_{ij}\setminus S))
\end{equation}
making of $E\cap W$ a constructible set and of $E^c\cap W$ an analytic one.\\
We claim that $E^c$ is a {\slshape strict} analytic subspace of $K_0$\footnote{ As above in Section \ref{Lieb}, we replace $K_0$ with its reduction if necessary, so strict means that $E^c$ is not a whole irreducible component of $K_0$.}. Assume the contrary. Then there exists an irreducible component $\mathcal C$ of $\mathscr{C'}$ such that $E^c_{\mathcal C}$ is a union of irreducible components of $\Delta$. For simplicity, let us assume that $K_0$ and thus $\Delta$ are irreducible. By Corollary \ref{2ndcor}, for each $J\in K_0$, there exists a connected component of $\mathcal C\cap p^{-1}(J,J)$ that contains only singular cycles.
Now $\mathcal C$ contains a Zariski open subset of graphs of biholomorphisms between fibers of the Kuranishi family. Hence, $p(\mathcal C)$ is an analytic set of $K_0\times K_0$ strictly containing $\Delta$ and $p(\mathcal C)$ has dimension strictly greater than $n$, the dimension of $K_0$. It follows that, at a generic point $J$ of $K_0$, the intersection of $p(\mathcal C)$ with $\{J\}\times K_0$ is positive-dimensional. We may thus find a Zariski open subset of $K_0$, say $U$, such that, for any $J\in U$, the intersection $\mathcal C\cap p^{-1}(\{J\}\times K_0)$ contains a Zariski open subset of graphs of biholomorphisms. In other words, for those $J$ in $U$, there exists a path of biholomorphisms between $X_J$ and some $X_{J'_t}$ with $J'_t$ distinct from $J$. Hence $K_0$ has a non trivial foliated structure in the sense of \cite{ENS}. But this implies that the dimension of $\Azero$ jumps at $0$, that is is not constant in a neighborhood of $0$ in $K_0$. Since $\mathscr{K}_0\to K_0$ is complete at every point $J$ of $K_0$, denoting its Kuranishi space $K_J$, then the closure of the set of exceptional points in $K_J$ is also the full $K_J$. Hence the same argument tells that the dimension of the automorphism group also jumps at $J$ in $K_0$. But it cannot jump at every point of $K_0$. Contradiction. The set $E^c$ is a strict analytic subspace of $K_0$.

So we may define a strict analytic substack of $\mathscr T(M,V)$ as the stackification of the full subgroupoid of $\mathcal T_V\rightrightarrows K_0$ above $E^c\subset K_0$. Since the notion of exceptional point is an intrinsic notion, this substack is just a neighborhood of $X_0$ of an analytic substack of $\mathscr{T}(M,V_K)$. 
\end{proof}

\begin{remark}
	\label{exexceptional}
	If the intersection of an exceptional component of $X_0$ and $\mathscr{C}_J$ is non-empty but contains regular cycles, then the corresponding morphisms 
	 form a component of $\text{Aut}^1(X_J)$ which is not induced by $\Aun$.
		Inversely, the intersection of an $A^1$-component at $X_0$ with $\mathscr{C}_J$ may be exceptional at $J$. Finally an $A^1$-component at $X_0$ may not intersect $\mathscr{C}_J$ since it only contains morphisms that send $J$ to points that are distinct from $J$ and not adherent to it. 
\end{remark}

\section{Variations on exceptionality}
\label{secvarexc}

In this Section, we explore variants of exceptional points. To avoid redundancies, we only deal with exceptionality but this material can be immediatly adapted to $\Z$-exceptionality.

\subsection{Relative exceptionality}
\label{secrelative}
The notion of exceptional points and cycles introduced in Definitions \ref{defexcpoint} and \ref{defexcycle} is an intrinsic notion, in the sense that it depends only on the complex manifold $X_0$. In this Section, we elaborate on a relative version, which depends on a family with fiber $X_0$.

\begin{definition}
	\label{defrelex}
	Let $\mathcal X\to B$ be a reduced $M$-deformation over $B$ a reduced analytic space with fiber $X_0$ above $0\in K_0$. Then $X_0$ is {\slshape $\mathcal X$-exceptional} if there exists a singular cycle $\gamma$ of the cycle space $\mathscr{C}_0$ such that
	\begin{enumerate}[i)]
		\item The cycle $\gamma$ does not belong to the closure of a connected component of $\Aun$.
		\item The cycle $\gamma$ belongs to an irreducible component of $\mathscr{C}_{\mathcal X}$ whose generic member is non singular.
	\end{enumerate}
	The cycle $\gamma$ is called {\slshape $\mathcal X$-exceptional}.
\end{definition}
Here $\mathscr{C}_{\mathcal X}$ denotes the the union of completely singular and regular components\footnote{ in the sense of \S \ref{secintrocycles}.} of the Barlet space of relative cycles of $\mathcal X\times\mathcal X$.

A direct rephrasing of Corollary \ref{2ndcor} shows that a K\"ahler manifold $X_0$ is exceptional if and only if it is $\mathscr{K}_0^{red}$-exceptional. However, in the non-K\"ahler case, the notion of exceptional point is more general. Notice also that $X_0$ is $\mathcal X$-exceptional if and only if there exists a sequence \eqref{phin} with $(x_n)$ and $(y_n)$ sequences of $B$ converging to $0$ and the graphs of $\phi_n$ belonging to $\mathscr{C}_{\mathcal X}$ such that the limit cycle $\gamma$ is singular and does not belong to the closure of a connected component of $\Aun$.

Now, the notion of $\mathcal X$-exceptionality depends strongly on $\mathcal X$. When $\mathcal X$ is not the reduction of the Kuranishi family, it may have nothing to do with exceptionality. A simple example is given by a trivial family $X_0\times B\to B$. Then $X_0$ is never $X_0\times B$-exceptional. 

Consider a cartesian diagram
\begin{equation}
	\label{CD}
	\begin{tikzcd}
		\mathcal{X}' \arrow[r] \arrow[d] \arrow[dr, phantom, "{\scriptscriptstyle \square}", very near start]
		& \mathcal {X} \arrow[d] \\
		B' \arrow[r, "f"]
		&B
	\end{tikzcd}
\end{equation}
Then
\begin{lemma}
	\label{lemmaexCD}
	Let $b'\in B'$ and set $b=f(b')$. Assume that $X'_{b'}$, that is the fiber of $\mathcal{X}'\to B'$ over $b'$, is $\mathcal{X}'$-exceptional. Then $X_b$, the fiber of $\mathcal X\to B$ over $b$, is $\mathcal X$-exceptional.
\end{lemma}

\begin{proof}
	Just consider a $\mathcal{X}'$-exceptional cycle $\gamma$ above $b'$ and its direct image $f_*\gamma$ above $b$ and observe that points i) and ii) of Definition \ref{defrelex} are preserved through direct images.
\end{proof}

As an immediate consequence of Lemma \ref{lemmaexCD}, if $\mathcal X\to B$ is complete at $0$ and $X_0$ K\"ahler, then $X_0$ is $\mathcal X$-exceptional if and only $X_0$ is exceptional. To prove the statement, just apply Lemma \ref{lemmaexCD} to the following two cartesian diagrams of germs of families
\begin{equation}
	\label{CDdouble}
	\begin{tikzcd}
		(\mathcal{X}, X_0) \arrow[r] \arrow[d] \arrow[dr, phantom, "{\scriptscriptstyle \square}", very near start]
		& (\mathscr{K}_0, X_0) \arrow[d] \\
		(B,0) \arrow[r]
		&(K_0, 0)
	\end{tikzcd}
\qquad\text{ and }\qquad
	\begin{tikzcd}
		 (\mathscr{K}_0,X_0) \arrow[r] \arrow[d] \arrow[dr, phantom, "{\scriptscriptstyle \square}", very near start]
		& (\mathcal{X}, X_0)\arrow[d] \\
		(K_0,0) \arrow[r]
		&(B,0)
	\end{tikzcd}
\end{equation}
The left diagram, resp. right diagram, is given by completeness of the Kuranishi family, resp. of $\mathcal X\to B$. More generally, $\mathcal{X}$-exceptionality implies exceptionality since the left diagram in \eqref{CDdouble} is always verified.

As in the case of exceptional points, we have
\begin{proposition}
	\label{proprel}
	Let $\mathcal X\to B$ be a reduced $M$-deformation with K\"ahler fibers and reduced base $B$. Then, the closure of $\mathcal X$-exceptional points is a strict analytic subspace of $B$.
\end{proposition} 

\begin{proof}
	A straightforward adaptation of the proof of Theorem \ref{2ndmainthm} shows that the closure of $\mathcal X$-exceptional points is an analytic subspace of $B$. We have to prove it is strict. Assume the contrary. Then at least a complete irreducible component of $B$ contains an open and dense subset of $\mathcal X$-exceptional points. For simplicity, assume $B$ irreducible. 
	Let $\mathcal C$ be an $\mathcal X$-exceptional component at some point $b\in B$ and let $p$ the natural projection map from $\mathcal C$ to $B\times B$. Since there are only a finite number of components in $\mathscr{C}_{\mathcal X}$, we may assume that $\mathcal C$ is also exceptional for all the points in a neighborhood of $b$. The component $\mathcal C$ intersects $p^{-1}(x,x)$ in a completely singular component $\mathcal C_x$ for all $x$ in this neighborhood. Arguing as in the proof of Theorem \ref{2ndmainthm}, this implies that $p(\mathcal C)$ is an analytic subspace of $B\times B$ strictly containing the diagonal. Hence, at every point $x$ of $B$ it intersects $\{x\}\times B$ in a positive-dimensional subspace. And the cycles above $x$ of this intersection include graphs of biholomorphisms above a Zariski open subset. That means that one can find a path $(b_t)$ in $B$ ending at $x$ such that all $X_{b_t}$ are biholomorphic to $X_b$ for $t\not =1$. Either $X_{b}$ is isomorphic to $X_x$ or not. In the first case, we may find by Fischer-Grauert a continuous family of biholomorphisms $(f_t)_{t\in [0,1]}$ with $f_t$ sending $X_{b_t}$ isomorphically to $X_x$ and $f_1$ equal to the identity. Given any sequence \eqref{phin} between fibers of $\mathcal{X}$, we may thus compose the mappings $\phi_n$ with suitable $f_t$ to obtain a sequence of automorphisms of $X_x$ whose graphs converge to the same limit cycle as the graphs of the $\phi_n$. Contradiction with the fact that there exists a $\mathcal{X}$-exceptional cycle. Hence $X_{b}$ is not biholomorphic to $X_x$ through a biholomorphism $C^\infty$-isotopic to the identity. But that means that the restriction of $\mathcal{X}$ to the path $b$ is a jumping family with central fiber $X_x$ and generic one $X_b$. This forces the $h^0$ function to jump at $x$. Since it is the case for every $x\in B$, we see that $h^0$ jumps at every point of $B$, contradicting its property of upper semi-continuity.
%
%
\end{proof}

Such a result is of course completely false for exceptional points. Letting $f:B\to K_0$ lands in the subspace of exceptional points, then every point of $f^*\mathscr{K}_0\to B$ is exceptional.  

\subsection{Exceptional pairs}
\label{secpairs}
Let $X_0$ and $X'_0$ be two compact complex manifolds diffeomorphic to $M$. Let $K_0$, resp. $K'_0$, the Kuranishi space of $X_0$, resp. $X'_0$. As a second variation on the theme of exceptionality, we define

\begin{definition}
	\label{defexpair}
	 We say that $\{X_0,X'_0\}$ is an {\slshape exceptional pair} of $\mathscr{T}(M)$ if
	 \begin{enumerate}[\rm i)]
	 	\item $X_0$ and $X'_0$ are not biholomorphic
	 	\item There exists a sequence $\eqref{phin}$ with $(x_n)$, resp. $(y_n)$, sequence of $K_0$, resp. $K'_0$, converging to $X_0$, resp. $X'_0$, such that the graphs of $\phi_n$ converge to a singular cycle $\gamma$ in the cycle space of $X_0\times X'_0$.
	 \end{enumerate}  
	The cycle $\gamma$ is called {\slshape exceptional}.
\end{definition}
Let us make a few comments on Definition \ref{defexpair}. Firstly, since $X_0$ and $X'_0$ are not biholomorphic, there is no need to add that $\gamma$ is not the limit of a sequence of biholomorphisms between $X_0$ and $X'_0$. Secondly, as the notion of exceptional points, Definition \ref{defexpair} does not depend on $\{X_0,X'_0\}$ up to biholomorphisms smoothly isotopic to the identity of the pair; so the notion of exceptional pairs is intrinsic and this justifies that we speak of exceptional pairs of $\mathscr{T}(M)$. Thirdly, as in the case of relative exceptionality, we do not include in the notion of exceptional pairs the analogues of wandering and vanishing sequences. This is mainly because we focus on the K\"ahler case where such phenomena do not appear.

\begin{example}
	\label{exjumping}
	Let $\mathcal X\to B$ be a {\slshape jumping family}, that is there is a point $0\in B$ such that
	\begin{enumerate}[i)]
		\item For all $b\in B$ different from $0$, the fiber $X_b$ is biholomorphic to a fixed complex manifold, say $X_1$.
		\item The $0$-fiber $X_0$ is not biholomorphic to $X_1$.
	\end{enumerate}
	Then the pair $\{X_0,X_1\}$ is an exceptional pair. Just take as $(x_n)$ the image through a map $f :B\to K_0$ such that $f^*\mathscr{K}_0$ is locally isomorphic to $\mathcal X$ of a sequence of points in $B\setminus\{0\}$ converging to $0$; and for $(y_n)$ the constant sequence $X_0$.
\end{example}

Start with $(X_0,X'_0)$ and their Kuranishi spaces $(K_0,K'_0)$. As usual, we assume both Kuranishi spaces reduced, replacing them with their reduction if necessary. We consider the space of relative cycles of $\mathscr{K}_0\times\mathscr{K}'_0\to K_0\times K'_0$. Let $\mathscr{C}^{pair}$ be the union of the irreducible components of this relative cycle space that contains at least the graph of a biholomorphism between some fiber of $\mathscr{K}_0$ and some fiber of $\mathscr{K}'_0$ which is smoothly isotopic to the identity. Let $S$ denote the subset of singular cycles of $\mathscr{C}^{pair}$ and let $p$ denote the natural projection of $\mathscr{C}^{pair}$ to $K_0\times K'_0$.

The set of exceptional couples in $K_0\times K'_0$ is thus equal to
\begin{equation}
	\label{pair}
	\mathscr{P}:=\{(J,J')\in K_0\times K'_0\mid p^{-1}(J,J')\not=\emptyset\text{ and }p^{-1}(J,J')\subset S\}
\end{equation}

Analogously to Theorem \ref{2ndmainthm}, we have
\begin{proposition}
	\label{propexpair}
	The closure of the subset of exceptional couples is a strict analytic substack of $\mathscr{T}(M,V_K)\times \mathscr{T}(M,V_K)$.
\end{proposition}

\begin{proof}
We deduce from \eqref{pair} that $\mathscr{P}$ is a constructible set in $K_0\times K'_0$, hence its closure is an analytic set. Assume now that the closure contains a full component of $K_0\times K'_0$. Then a full component of $\mathscr{C}^{pair}$ consists of singular cycles since a dense subset of $K_0\times K'_0$ consists of non-isomorphic couples. This is in contradiction with its definition.
\end{proof}

However, notice that the projection onto $K_0$, resp. $K'_0$, may contain a full component of $K_0$, resp. $K'_0$. 

\begin{example}
	\label{exF2}
	Let $X_0$ be the product of projective lines $\mathbb P^1\times \mathbb P^1$ and let $X'_0$ be the second Hirzebruch surface. Then $K_0$ is a point since $\mathbb P^1\times \mathbb P^1$ is rigid; and $K'_0$ is (the germ of) a unit disk with $0$ encoding $\mathbb F_2$ and $t\not =0$ encoding $\mathbb P^1\times \mathbb P^1$ (this is a jumping family in the sense of Example \ref{exjumping}).
	
	The set of exceptional pairs is the pair formed by the unique point of $K_0$ and $0$. This is a strict analytic subset of the disk but not of the point.
\end{example}

\subsection{Non-Hausdorff points}
\label{secNH}
As a consequence of what preceeds, we obtain that the set of non-Hausdorff K\"ahler points is included in a strict analytic substack. 

To be more precise, first recall that the topological Teichm\"uller space - that is the quotient $\mathcal I(M)/\D$ endowed with the quotient topology - is also given as the geometric quotient of the Teichm\"uller stack, that is, given an atlas $\mathcal T_0\to \mathscr{T}(M)$ and given the associated groupoid $\mathcal T_1\rightrightarrows \mathcal T_0$, the quotient space of $\mathcal T_0$ by the equivalence relation generated by the morphisms encoded in $\mathcal T_1$. By extension we define

\begin{definition}
	\label{defNH}
	We say that $X_0$ and $X_1$ are {\slshape non-Hausdorff points in $\mathscr{T}(M)$}, or form a {\slshape non-Hausdorff pair of $\mathscr{T}(M)$} if they correspond to non-separated points $[J_0]$ and $[J_1]$ of $\mathcal I(M)/\D$.
\end{definition}

Using the local atlases \eqref{atlasneigh}, if a pair $\{[J_0],[J_1]\}$ of $\mathcal I(M)/\D$ is a pair of non-separated points of $\mathcal I(M)/\D$ then, setting $X_0:=(M,J_0)$, $X_1:=(M,J_1)$, and letting $K_0$, resp. $K_1$, be the Kuranishi space of $X_0$, resp. $K_1$, we may find a sequence \eqref{phin} with $(x_n)$ in $K_0$ converging to $X_0$ and $(y_n)$ in $K_1$ converging to $X_1$. In other words, two non-separated points of $\mathcal I(M)/\D$ define an exceptional pair of $\mathscr{T}(M)$.

As an immediate consequence to Proposition \ref{propexpair}, we thus have
\begin{corollary}
	\label{corNH}
	The set of couples of non-Hausdorff K\"ahler points in $\mathscr{T}(M)$ is contained in the set of exceptional couples of $\mathscr{T}(M, V_K)$, hence contained in a strict analytic substack of $\mathscr{T}(M, V_K)\times \mathscr{T}(M, V_K)$.
\end{corollary}

As an example, $\mathbb P^1\times\mathbb P ^1$ and $\mathbb F_2$ are non-Hausdorff points of $\mathscr{T}(\mathbb S^2\times\mathbb S^2)$, cf. Example \ref{exF2}, as well as the the pair $\{X_0,X_1\}$ in a jumping family, cf. Example \ref{exjumping}.

\section{Pathologies of the Teichm\"uller stack}
\label{secpathologies}

As explained in the Introduction \S \ref{intro}, making use of a Teichm\"uller stack rather than a Teichm\"uller space allows to pass easily from the point of view of moduli space to the point of view of families. They are two faces of the same mathematical object.

As a consequence, we want to explore in this Section the pathologies of the Teichm\"uller stack as pathologies of families and to relate them to pathologies of space. We notice that we already investigated a topological pathology in Section \ref{secNH}. But here we would like to search for analytic pathologies. As in \S \ref{secvarexc}, to avoid redundancies, we only deal with exceptionality but all can be immediatly adapted to $\Z$-exceptionality.

\subsection{Analytically non-separated, ambiguous and undistinguishable points}
\label{secambig}
Let us begin with some definitions.

\begin{definition}
	\label{defanaNH}
	Two distinct points $X_0$ and $X_1$ of the Teichm\"uller stack $\mathscr{T}(M)$ are {\slshape analytically non-separated} if there exist two reduced $M$-defor\-mations $\mathcal X_i\to B$ of $X_i$ (for $i=0,1$) above a reduced positive dimensional base $B$ such that $\mathcal X_0$ and $\mathcal X_1$ are isomorphic above $B\setminus\{0\}$.
\end{definition}

Taking sequences in $B\setminus\{0\}$ that converge to $0$, we immediatly obtain that $\{X_0,X_1\}$ is an exceptional pair; and a non-Hausdorff pair. However Definition \ref{defanaNH} is stronger. Observe that 
\begin{enumerate}[i)]
	\item Any couple $(X_0,X_1)$ formed by the base point, resp. the generic point, of a jumping family (cf. Examples \ref{exF2} and \ref{exjumping}) are analytically non-separated.
	\item We can always assume that $B$ is a disk.
\end{enumerate}

We pass now to the notions of analytic ambiguity and undistinguishability. In the following definitions, given $X_0$ and $X_1$, we set $X_i=(M,J_i)$ and let $V_i$ be an open neighborhood of $J_i$ in $\mathscr{I}(M)$. Finally, $V_i^*$ denotes $V_i$ minus the $\D$-orbit of $J_i$.

\begin{definition}
	\label{defambig}
	Two distinct points $X_0$ and $X_1$ of the Teichm\"uller stack $\mathscr{T}(M)$ are {\slshape formally ambiguous} if, for any choice of a small enough neighborhood $V_0$, resp. $V_1$, there exists a neighborhood $V_1$, resp. $V_0$, such that $\mathscr{T}(M, V_0^*)$ is equivalent to $\mathscr{T}(M, V_1^*)$, that is there exists a fully faithful and essentially surjective morphism from $\mathscr{T}(M, V_0^*)$ to $\mathscr{T}(M, V_1^*)$.\medskip\\
	 They are {\slshape analytically ambiguous} if 
	 \begin{enumerate}[\rm i)]
	 	\item they are formally ambiguous and the corresponding equivalence between $\mathscr{T}(M, V_0^*)$ and $\mathscr{T}(M, V_1^*)$ sends a reduced $V_0^*$-family to an {\slshape isomorphic} reduced $V_1^*$-family.
	 	\item \label{itemambig} There exists non-isomorphic non-isotrivial reduced $V_i$-deformations $\mathcal{X}_i\to B$ (for $i=0,1$) over a positive-dimensional connected base $B$ whose restrictions over $B^*$\footnote{\label{footnoteBstar} We assume that $B^*$ for $\mathcal{X}_0$ and $B^*$ for $\mathcal{X}_1$ are equal.} are images through the equivalence of point i), hence isomorphic.
	 \end{enumerate}
\end{definition}

\begin{remark}
	\label{rkambiguousrigid}
	Because of hypothesis \ref{itemambig}), $X_0$ and $X_1$ are not rigid. By a {\slshape rigid} manifold, we mean a compact complex manifold such that any sufficiently small deformation of it is $\D$-biholomorphic\footnote{ that is biholomorphic through a biholomorphism smoothly isotopic to the identity.} to it.
\end{remark}

In other words, $X_0$ and $X_1$ are formally ambiguous iff they have arbitrary small isomorphic punctured neighborhoods in $\mathscr{T}(M)$. To understand the difference with analytically ambiguous, notice that two compact surfaces of genus $g>1$ are always formally ambiguous, but never analytically ambiguous. Indeed, the Teichm\"uller stack of compact surfaces of fixed genus $g>1$ is a bounded domain in $\mathbb C^{3g-3}$; in particular, it is a manifold, every automorphism group $\Aun$ is the identity and two distinct points are not biholomorphic through a biholomorphism smoothly isotopic to the identity. Thus, a punctured neighborhood of some $X$ in this Teichm\"uller stack is isomorphic to a punctured neighborhood of $0$ in $\mathbb C^{3g-3}$ proving the first point. However, given $X_0$ and $X_1$ distinct, and disjoint neighborhoods $V_0$ and $V_1$ with $\mathscr{T}(M, V_0^*)$ equivalent to $\mathscr{T}(M, V_1^*)$ as a category, every complex structure encoded in $V_0^*$ is distinct from any complex structure encoded in $V_1^*$, hence the previous morphism cannot send a family to an isomorphic family.

Definition \ref{defambig} is coined to ensure the following lemma.

\begin{lemma}
	\label{lemmaambiganans}
	Ambiguous K\"ahler points are analytically non-separated.
\end{lemma}

\begin{proof}
	Consider the non-isomorphic non-isotrivial reduced $V_i$-deformations $\mathcal{X}_i\to B$ given by Definition \ref{defambig}. Note that $B^*$ is not empty otherwise the deformations $\mathcal{X}_i$ would be isotrivial. We can thus restrict both families to non-trivial ones over a disk. We denote them by $\mathcal{X}_i\to\mathbb{D}$ for simplicity. Since $X_i$ is K\"ahler, we may assume that $0$ is the only point of $\mathbb D$ encoding $X_i$ up to biholomorphisms smoothly isotopic to the identity, cf. the proof of Corollary \ref{interlemma}. So, in the associated $V_i^*$-family $\mathcal X_i^*\to \mathbb D^*$, the base is in fact $\mathbb D\setminus\{0\}$. By Definition \ref{defambig}, to $\mathcal X\to\mathbb D$ we may associate a deformation of $X_1$ over the disk such that the two families are isomorphic outside $0$. This isomorphism doesnot extend over $\mathbb{D}$ since the central fibers are not biholomorphic. This is exactly saying that $X_0$ and $X_1$ are analytically non-separated.
\end{proof}

\begin{remark}
	\label{rkambigns}
	Notice however that the line of arguments of the proof of Lemma \ref{lemmaambiganans} does not work if we only assume that $X_0$ and $X_1$ fulfill point i) of Definition \ref{defambig} and are non-rigid instead of being analytically ambiguous. Indeed, under this milder assumption, we may find a non-trivial deformation of $X_0$ over the disk and associate to its restriction over $\mathbb{D}^*=\mathbb{D}\setminus\{0\}$ a $V_1^*$-family over the punctured disk. But there is no reason it extends to a deformation of $X_1$ over the disk. 
\end{remark}

 Analogously to Definition \ref{defambig}, we have

\begin{definition}
	\label{defundistinguish}
	Two distinct points $X_0$ and $X_1$ of the Teichm\"uller stack $\mathscr{T}(M)$ are {\slshape formally undistinguishable} if, for any choice of a small enough neighborhood $V_0$, resp. $V_1$, there exists a neighborhood $V_1$, resp. $V_0$, such that $\mathscr{T}(M, V_0)$ is equivalent to $\mathscr{T}(M, V_1)$. We ask these equivalences to send $X_0$ to $X_1$.\medskip\\
	They are {\slshape analytically undistinguishable} if 
	\begin{enumerate}[i)]
		\item they are non rigid\footnote{ in the sense of Remark \ref{rkambiguousrigid}.}, 
		\item formally undistinguishable and the corresponding equivalence between $\mathscr{T}(M, V_0)$ and $\mathscr{T}(M, V_1)$ sends a reduced $V_0^*$-deformation of $X_0$ to an {\slshape isomorphic} reduced $V_1^*$-defor\-mation of $X_1$.
	\end{enumerate}
\end{definition}
That is, $X_0$ and $X_1$ have arbitray small isomorphic neighborhoods in $\mathscr{T}(M)$ and there is no way to distinguish them by looking locally at the Teichm\"uller stack if they satisfy Definition \ref{defundistinguish}. Any two compact surfaces of genus $g>1$ are indeed formally undistinguishable. We give now examples of analytically undistinguishable points.
\begin{example}
	\label{exampleF2undistinguish}
	A first example of analytically undistinguishable points is given by $\mathbb F_{2a}$ and $\mathbb F_{-2a}$ in the Teichm\"uller stack of $\mathbb S^2\times\mathbb S^2$. Indeed the flip $(x,y)\mapsto (y,x)$ of $\mathbb S^2\times\mathbb S^2$ defines an automorphism of $\mathbb P^1\times \mathbb P^1$ that exchanges $\mathbb F_{2a}$ and $\mathbb F_{-2a}$. 
\end{example}

\begin{example}
	\label{exampleK3undistinguish}
	Inseparable points in the Teichm\"uller space of K3 surfaces are indeed isomorphic with same periods, see Proposition 2.1 of Chapter 7 in \cite{Huybrechts}. So they are analytically undistinguishable.
\end{example}

Formally unidstinguishable points are of course formally ambiguous (compare Definitions \ref{defambig} and \ref{defundistinguish}). The relationship between analytically non-separated, ambiguous and undistinguishable is more subtle and goes as follows.

\begin{lemma}
	\label{lemmaambigvsundis}
	Let $X_0$ and $X_1$ be two points of $\T$. 
	\begin{enumerate}[\rm i)]
		\item Assume that $X_0$ and $X_1$ are analytically undistinguishable points.  Then they are analytically ambiguous.
		\item Assume that $X_0$ and $X_1$ are K\"ahler analytically undistinguishable points. Then they are analytically non-separated.
	\end{enumerate}
\end{lemma}

\begin{proof}
	We first prove i). Assume that $X_0$ and $X_1$ are analytically undistinguishable points. Then $X_0$ is in particular non rigid, hence there exists a deformation $\mathcal{X}_0\to B$ of $X_0$ with two non-isomorphic fibers. Its image $\mathcal{X}_1\to B$ through the equivalence of Definition \ref{defundistinguish} is thus non-isotrivial. Since $X_0$ and $X_1$ are not isomorphic, the families $\mathcal{X}_0$ and $\mathcal{X}_1$ satisfy point ii) of Definition \ref{defambig} and we are done. 
	
	Now ii) follows immediatly from i) and from Lemma \ref{lemmaambiganans}.
\end{proof}

In view of Lemma \ref{lemmaambigvsundis}, it is natural to ask for examples of analytically ambiguous but distinguishable points. Here is a potential one.

\begin{example}
	\label{exHK}
	Consider the case of Hyperk\"ahler manifolds. It is quite different from that of K3 surfaces, treated in Example \ref{exampleK3undistinguish}. Inseparable points of the Teichm\"uller space of Hyperk\"ahler structures on a fixed differentiable type are not always isomorphic, cf. \cite{Debarre} and \cite{Namikawa}. Hence they are not analytically undistinguishable. However, the global Torelli Theorem of \cite{VerbitskyTorelli} implies that they are analytically ambiguous in case they are isolated. We do not know however if such examples exist.
\end{example}


Undistinguishable points enjoy several interesting properties. 
 We have
 \begin{lemma}
 	\label{lemmaundis0}
 	If $X_0$ and $X_1$ are formally undistinguishable, then their Kuranishi spaces $K_0$ and $K_1$ are isomorphic as germs of analytic spaces.
 \end{lemma}

\begin{proof}
	Consider the equivalence $F$ between $\mathscr{T}(M, V_0)$ and $\mathscr{T}(M, V_1)$. It sends every cartesian diagram
	\begin{equation}
		\label{CDKur}
		\begin{tikzcd}
			(\mathcal{X}, X_0) \arrow[r] \arrow[d] \arrow[dr, phantom, "{\scriptscriptstyle \square}", very near start]
			& (\mathscr{K}_0, X_0) \arrow[d] \\
			(B,0) \arrow[r]
			&(K_0, 0)
		\end{tikzcd}
	\end{equation}
	to the cartesian diagram
	\begin{equation}
		\label{CDFDKur}
		\begin{tikzcd}
			(F(\mathcal{X}), X_1) \arrow[r] \arrow[d] \arrow[dr, phantom, "{\scriptscriptstyle \square}", very near start]
			& (F(\mathscr{K}_0), X_1) \arrow[d] \\
			(B,0) \arrow[r]
			&(K_0, 0)
		\end{tikzcd}
	\end{equation}
	Since $F$ is essentially surjective, every family of $\mathscr{T}(M, V_1)$ is isomorphic to a family in the image of $F$ and \eqref{CDFDKur} expresses that the image of the Kuranishi family $\mathscr{K}_0\to K_0$ of $X_0$ is complete for $X_1$. This implies that the dimension of the Zariski tangent space of $K_0$ at $0$ is greater than the dimension of that of $K_1$.
	
	Similarly, still because $F$ is essentially surjective, there exists a reduced deformation $\mathcal X_1$ of $X_1$ over $K_1$ such that $F(\mathcal X_1)$ is isomorphic to $\mathscr{K}_1$ hence semi-universal for $K_1$. Given any reduced deformation $\mathcal X\to B$ of $X_0$ then semi-universality yields
	\begin{equation}
		\label{CDFDKur2}
		\begin{tikzcd}
			(F(\mathcal{X}), X_1) \arrow[r] \arrow[d] \arrow[dr, phantom, "{\scriptscriptstyle \square}", very near start]
			& (F(\mathcal{X}_1), X_1) \arrow[d] \\
			(B,0) \arrow[r]
			&(K_1, 0)
		\end{tikzcd}
	\end{equation}
	It is the image through $F$ of a family morphism $\mathcal X\to\mathcal X_1$ over $B\to K_1$ because $F$ is fully faithful. But this implies that $\mathcal X_1\to K_1$ is complete for $X_0$. In particular, the dimension of the Zariski tangent space of $K_1$ at $0$ is greater than the dimension of that of $K_0$. So they are equal and $\mathcal X_1\to K_1$ is indeed semi-universal for $X_0$. As a consequence $K_0$ and $K_1$ are isomorphic as germs of analytic spaces at the base point.
\end{proof}
Let us deal now with rigid manifolds, still in the sense of Remark \ref{rkambiguousrigid}. Notice that the Kuranishi space of a rigid manifold is a point, but not always a reduced one, cf. \cite{Bauer}.
\begin{lemma}
	\label{lemmarigid}
	Let $X_0$ and $X_1$ be rigid complex manifolds.
	Then,
	\begin{enumerate}[\rm i)]
		\item $X_0$ and $X_1$ are formally ambiguous.
		\item $X_0$ and $X_1$ are formally undistinguishable if and only if they have isomorphic $\text{\rm Aut}^1$ group and isomorphic Kuranishi spaces.
		\item $X_0$ and $X_1$ are analytically separated.
	\end{enumerate}
\end{lemma}

\begin{proof}
	If $X_0$ and $X_1$ are rigid, given any reduced $M$-deformations $\mathcal X_i\to B$ of $X_i$ (for $i=0,1$), there exists an open neighborhood $V$ of $0$ in $B$ above which the restriction of $\mathcal X_i$ is isomorphic to $X_i\times V$. Thus both families cannot be isomorphic above $B\setminus\{0\}$ and $X_0$ and $X_1$ are analytically separated. This proves iii).
	
	Two rigid manifolds are formally ambiguous, since the corresponding neighborhoods $V_i^*$ are empty. This proves i). 
	
	To satisfy \ref{defundistinguish}, the existence of an isomorphism between small neighborhoods of $X_0$ and $X_1$ is needed. It must send $X_0$ to $X_1$ thus they must have same isotropy groups, that is same $\text{Aut}^1$ groups. Moreover, Lemma \ref{lemmaundis0}, they have isomorphic Kuranishi spaces.
	
	Conversely, the Teichm\"uller stack is locally isomorphic at a rigid point to the stack quotient of a $n$-uple point (its Kuranishi space) by its $\text{Aut}^1$ group. Hence, if $X_0$ and $X_1$ have same $\text{Aut}^1$ groups and same Kuranishi spaces, they have isomorphic neighborhoods. This proves ii).
\end{proof}
%
%

We now show that all the previous analytic pathologies on families concern only a strict analytic substack of $\mathscr{T}(M,V_K)$.

\begin{proposition}
	\label{propambig}
	Let $X_0$ and $X_1$ be distinct K\"ahler points of the Teichm\"uller stack. Assume that
	\begin{enumerate}
		\item[\rm i)] $X_0$ and $X_1$ are analytically undistinguishable,
		\item[\rm or ii)] $X_0$ and $X_1$ are analytically non-separated.
	\end{enumerate}
	Then $X_0$ and $X_1$ are non-Hausdorff points and form an exceptional pair.
	
	In particular, the set of analytically undistinguishable,
	or non-separated K\"ahler pairs is included in a strict analytic substack of $\mathscr{T}(M,V_K)\times \mathscr{T}(M,V_K)$.
\end{proposition}

\begin{proof}
	We already observed after Definition \ref{defanaNH} that analytically non-sepa\-rated points are non-Hausdorff. Thus this also true for analytically undistinguishable points thanks to Lemma \ref{lemmaambigvsundis}. 
	
	Finally, making use of Proposition \ref{propexpair} proves the last sentence.
\end{proof}
 The case of ambiguous points is unclear due to Remark \ref{rkambigns}.\\
 We sum up some of the properties in the following table.
{\setlength{\arrayrulewidth}{0.5mm}
	\renewcommand{\arraystretch}{1.5} 

\begin{table}[H]
	\label{table1}
	\begin{tabular}{|c|m{2.7cm}|m{2.7cm}|m{2.7cm}|}
		\hline
		\bfseries {Analytically} &\centering{\bfseries{Teichm\"uller}}  &\centering{\bfseries{Families}} & \bfseries{Examples} \\
		\hline
		non-separated & isomorphic punctured slices & existence of i\-so\-mor\-phic local punctured de\-for\-mations & jumping families\\
		\hline
		ambiguous & isomorphic punctured neighborhoods  & duality between local punctured deformations & see Example \ref{exHK}\\
		\hline
		undistinguishable & isomorphic neighborhoods & duality between local de\-for\-ma\-tions & Hirzebruch surfaces $\mathbb F_a$ and $\mathbb F_{-a}$\\
		\hline
	\end{tabular}
\vskip.3cm
\caption{Pathologies for $X_0$ and $X_1$ distinct and K\"ahler } 
\end{table}

\subsection{The local isomorphism property}
\label{seclocaliso}
Following \cite{ENS}, we say that a compact complex manifold $X_0$ has the local isomorphism property if any two pointwise isomorphic deformations of $X_0$ are locally isomorphic at the base point $0$.

In \cite{ENS}, we claim that $X_0$ has the local isomorphism property for deformations over a reduced base if and only if the function
\eqref{h0}
is constant on the fibers $X_t$ of the Kuranishi family $\mathscr{K}_0\to K_0$.

This is however not true and counterexamples exist on K3 surfaces as explained in \cite{Kirsch}. This comes from the fact that the number of chambers in the positive cone of a K3 surface may jump above at the base point of a deformation, see the main Lemma of \cite{BP} or \cite{Beauville}. Start with a K3 surface $X_0$ whose positive cone is subdivided into several chambers. Choose a $(-2)$-curve $C$ on $X_0$ such that the Hodge isometry of $H^2(X_0,\C)$ given by reflection around the hyperplane normal to the class of $C$, say $\sigma$, is non-effective, that is does not respect the K\"ahler chamber. Consider a deformation $\mathcal X\to B$ of $X_0$ whose generic fiber does not admit such a subdivision of its positive cone. Apply the Hodge isometry $\sigma$ fiberwise to the punctured family, giving another family isomorphic to the first one above the punctured base $B^*$. Then it can be checked that this new family extends as a deformation $\mathcal X'\to B$ of $X_0$. Now $\mathcal X$ and $\mathcal X'$ are not locally isomorphic at $0$, since a local isomorphism would induce an automorphism on the central fiber $X_0$ acting on $H^2(X_0,\C)$ as the non-effective Hodge isometry $\sigma$ we start with. 

Notice that in such counterexamples, we may assume the families to be reduced families (the construction process is local) but the two families are not pointwise isomorphic as reduced families. Indeed the central fiber of $\mathcal X$ and $\mathcal X'$ are common, hence identified through the identity, whereas the other fibers are identified through a biholomorphism acting as the reflection $\sigma$ on the local system of second cohomology groups. They are inseparable points of the Teichm\"uller space, cf. Example \ref{exampleK3undistinguish}.

\begin{definition}
	\label{defdiff0lip}
	We say that a compact complex manifold $X_0$ has the {\slshape local $\D$-isomorphism property} if any two pointwise $\D$-isomorphic reduced deformations of $X_0$ are locally $\D$-isomorphic at the base point $0$.
\end{definition}

We do not know of an example of a compact complex manifold $X_0$ not having the $\D$-local isomorphism property and such that \eqref{h0} is constant. However, we shall prove

\begin{theorem}
	\label{thmlip}
	Let $X_0$ be K\"ahler. Assume that $X_0$ has not the $\D$-local isomorphism property over reduced bases. Then,
	either the $h^0$-function \eqref{h0} is not locally constant, or $X_0$ is an exceptional point.
\end{theorem}
\noindent and its immediate corollary
\begin{corollary}
	\label{corlip}
	The subset of points of $\mathscr{T}(M,V_K)$ not having the $\D$-local isomorphism property over reduced bases is included in a strict analytic substack of $\mathscr{T}(M,V_K)$.
\end{corollary}

Of course, assuming part I of Conjecture \ref{mainconj}, then Theorem \ref{thmlip} implies that a K\"ahler point with locally constant $h^{0}$-function has the $\D$-local isomorphism property over reduced bases.

\begin{proof}[Proof of Corollary \ref{corlip}]
	Apply Theorem \ref{thmlip}. The conclusion follows then 
	from the fact that the $h^0$-function is upper semi-continuous for the Zariski topology on $K_0$.	
\end{proof}

\begin{proof}[Proof of Theorem \ref{thmlip}]
	Let $X_0$ be K\"ahler  and not having the $\D$-local isomorphism property over reduced bases. By \cite[p.513--514]{ENS}, $X_0$ has not the $\D$-local isomorphism property over disks. So let $\mathcal X\to\mathbb D$ and $\mathcal X'\to\mathbb D$ be two pointwise $\D$-isomorphic but not locally $\D$-isomorphic reduced deformations of $X_0$. Let $f:\mathbb D\to K_0$, resp. $g:\mathbb D\to K_0$ be holomorphic mappings such that $f^*\mathscr{K}_0$ is isomorphic to $\mathcal X$, resp. $g^*\mathscr{K}_0$ is isomorphic to $\mathcal X'$. As usual, $\mathscr{K}_0\to K_0$ is the Kuranishi family of $X_0$ and the existence of $f$ and $g$ comes from its semi-universality property, cf. Section \ref{Kurfamily}. Fix a sequence $(t_n)$ of $\mathbb D^*$ converging to zero. Set 
	\begin{equation}
		\label{xnyn}
		\forall n\in\N,\qquad x_n=f(t_n)\qquad\text{ and }\qquad y_n=g(t_n)
	\end{equation}
	Thus $(x_n)$ and $(y_n)$ are sequences of points in $K_0$ converging to $0$. Since $\mathcal X$ and $\mathcal X'$ are pointwise $\D$-isomorphic, we may choose a sequence $(\phi_n)$ in $\D$ satisfying \eqref{phin}. Up to passing to a subsequence, we may assume that the graphs of the $\phi_n$ all belong to the same irreducible component of $\mathscr{C}$ and converge in this cycle space to a $\gamma_0$. 
	
	Assume now that \eqref{h0} is constant and that $X_0$ is not exceptional. Then $\gamma_0$ is singular, otherwise its extension would induce a local isomorphism between the two families $\mathcal X$ and $\mathcal X'$. Indeed, assume that $\gamma_0$ is the graph of an automorphism $H$ of $X_0$, and still denote by $H$ its extension as an isomorphism of the Kuranishi family:
	\begin{equation}
		\label{CDiso}
		\begin{tikzcd}
			 (\mathscr{K}_0,X_0) \arrow[r,"H"] \arrow[d] \arrow[dr, phantom, "{\scriptscriptstyle \square}", very near start]
			& (\mathscr{K}_0, X_0)\arrow[d] \\
			(K_0,0) \arrow[r, "h"]
			&(K_0,0)
		\end{tikzcd}
	\end{equation}
		 then $H^{-1}\circ \phi_n$ tends to the identity uniformly in every chart and in every $L^2_k$-norm. In particular, since \eqref{h0} is constant, this implies that $H^{-1}\circ \phi_n$ fixes $x_n$ for $n$ big enough, see \cite{Kur3}. Hence $h$ induces a local isomorphism of $K_0$ making the following diagram commutative
	\begin{equation}
		\label{CDiso2}
		\begin{tikzcd}
			(\mathbb D,0) \arrow[r,"f"] \arrow[dr, "g"'] 
			& (K_0, 0)\arrow[d,"h"] \\
			&(K_0,0)
		\end{tikzcd}
	\end{equation}
	because it satisfies this diagram on the convergent sequence $(t_n)$ of the disk. Finally \eqref{CDiso} lifts as a commutative diagram of local isomorphisms of the corresponding families as stated. So $\gamma_0$ is singular.
	
	Since $X_0$ is not exceptional, the singular cycle $\gamma_0$ belongs to an $\Aun$-component of $\mathscr{C}_0$, hence is the limit of graphs of automorphisms $(H_n)$ of $X_0$. We need the following lemma.
	
	\begin{lemma}
		\label{lemmadefdomain}
		Assume \eqref{h0} is constant and $X_0$ is not exceptional. Then, there exists an open neighborhood $U$ of $0$ in $K_0$ such that any element of $\Aun$ admits an extension as an isomorphism of the Kuranishi family above $U$.
	\end{lemma}

	\begin{proof}[Proof of Lemma \ref{lemmadefdomain}]
We already proved this result for $\Azero$ in Lemma \ref{lemmah0constant}. Since $X_0$ is K\"ahler, $\Aun$ has a finite number of connected components, say $p$. Fix one element $g_i$ in each component. Reducing $U$ if necessary, we may assume that the $g_i$ admit a holomorphic extension above $U$. Now any element of $\Aun$ is a composition $g_i\circ f$ with $f\in\Azero$, hence admits an extension over $U$. 
	\end{proof}
	Then, reducing $K_0$ if necessary, we have a commutative diagram \eqref{CDiso} for every $(H_n,h_n)$, and all of them are defined above the whole $K_0$. Arguing as above, we see that, for $n$ big enough, the sequence $(h_n)$ stabilizes to some mapping $h$ which satisfies \eqref{CDiso2}, as well as the sequence $(H_n)$. Hence $\gamma_0$ is the graph of $H_n$ for $n$ big enough so is not singular. Contradiction that proves the Theorem.
\end{proof}

\subsection{Double and split points}
\label{secubiquit}
In Section \ref{secambig}, we deal with distinct points of the Teichm\"uller stack
admitting isomorphic neighborhoods (or punctured neighborhoods or slices in the milder versions). Here distinct points cannot be recognized by looking at their neighborhoods.

In this section, we deal with a phenomenon which is in a sense opposite: a single point of the Teichm\"uller stack admitting non-isomorphic neighborhoods or slices. Here a single point does not determine its neighborhood.

Analogously to Definition \ref{defanaNH}, we set

\begin{definition}
	\label{deflip}
	A point $X_0$ of the Teichm\"uller stack $\mathscr{T}(M)$ is a {\slshape double point} if there exist two reduced $M$-defor\-mations $\mathcal X_i\to B$ of $X_0$ (for $i=0,1$) above a reduced positive dimensional base $B$ such that $\mathcal X_0$ and $\mathcal X_1$ are isomorphic above $B\setminus\{0\}$ but non-isomorphic over $B$.
\end{definition}

We note that a double point $X_0$ has not the local $\D$-isomorphism property. Analogously to Definitions \ref{defambig}, we set

\begin{definition}
	\label{defsplit}
	A point $X_0$ of the Teichm\"uller stack $\mathscr{T}(M)$ is {\slshape split} if, for any choice of a small enough neighborhood $V_0$ of $X_0$, there exists a neighborhood $V_1$ of $X_0$ and an equivalence between $\mathscr{T}(M, V_0^*)$ and $\mathscr{T}(M, V_1^*)$ such that 
	\begin{enumerate}[i)]
		\item It sends every reduced $V_0^*$-deformation of $X_0$ to an {\slshape isomorphic} reduced $V_1^*$-defor\-mation of $X_0$.
		\item There exists non-isomorphic non-isotrivial reduced $V_i$-deformations $\mathcal{X}_i\to B$ (for $i=0,1$) over a positive-dimensional connected base $B$ whose restrictions over $B^*$ (cf. footnote \ref{footnoteBstar}) are images through the equivalence of point i), hence isomorphic.
	\end{enumerate}
\end{definition}


We note that Definition \ref{defsplit} is similar to the notion of {\it analytic} ambiguity. Formally split is obviously automatically satisfied so we drop the adjective analytic in Definition \ref{defsplit}.

Also, note that there is no analogue to undistinguishable points. Indeed, the natural definition would be that of a point $X_0$ such that for any choice of a small enough neighborhood $V_0$ of $X_0$, there exists a neighborhood $V_1$ of $X_0$ and an equivalence between $\mathscr{T}(M, V_0)$ and $\mathscr{T}(M, V_1)$ such that 
\begin{enumerate}[i)]
			\item It sends every reduced $V_0^*$-deformation of $X_0$ to an {\slshape isomorphic} reduced $V_1^*$-defor\-mation of $X_0$.
			\item There exists a $V_0$-deformation which is {\slshape not isomorphic} to its image.
	\end{enumerate}
Now, arguing as in Lemma \ref{lemmaundis0}, we would have the equivalence sending the Kuranishi family to itself. Hence it would send every reduced family to an isomorphic one, contradicting item ii).

\begin{example}
	\label{exF2split}
	We claim that the Hirzebruch surface $\mathbb F_2$ is split. Recall that its Kuranishi space is a jumping family with generic point $\mathbb P^1\times \mathbb P^1$. Hence, for $V_0$ small enough, the objects of $\mathscr{T}(M,V_0)$ are reduced families with every fiber isomorphic either to $\mathbb F_2$ or to $\mathbb P^1\times \mathbb P^1$; and the objects of $\mathscr{T}(M,V^*_0)$ are reduced families with every fiber isomorphic to $\mathbb P^1\times \mathbb P^1$, that is locally trivial $\mathbb P^1\times \mathbb P^1$-bundles by Fischer-Grauert Theorem. 
	
	Consider the stack morphism induced by the map $z\mapsto z^2$ of the unit disk. It sends the pull-back of the Kuranishi family by some mapping $f:B\to\mathbb D$ to the pull-back of the Kuranishi family by $z\in B\mapsto (f(z)^2)\in\mathbb D$.
	It is an equivalence on $V_0^*$, that is on locally trivial $\mathbb P^1\times \mathbb P^1$-bundles so $\mathbb F_2$ is split. Nevertheless, it is not essentially surjective on $V_0$ since it maps reduced deformations of $\mathbb F_2$ to reduced deformations of $\mathbb F_2$ {\itshape with Kodaira-Spencer map zero at the base point}. Thus, no family in the image is isomorphic to the Kuranishi family.
\end{example}

Once again, all these pathologies only occur on a strict analytic substack of $\mathscr{T}(M,V_K)$.
\begin{theorem}
	\label{thmpathosubstack}
	Let $X_0$ be a point of $\mathscr{T}(M,V_K)$. Then, 
	\begin{enumerate}[\rm i)]
	\item We have
	\begin{equation}
		\label{suiteimpli}
		\begin{aligned}
		X_0&\text{ split }\Longrightarrow X_0\text{ double }\\
		&\Longrightarrow X_0\text{ is exceptional or \eqref{h0} is not constant.}
		\end{aligned}
	\end{equation}
	\item In particular, the subset of split, resp. double points of $\mathscr{T}(M,V_K)$ is included in a strict analytic substack of $\mathscr{T}(M,V_K)$.
	\end{enumerate}
\end{theorem}

\begin{proof}
	The first implication are obvious from the definitions. Then we already noticed that a double point has not the local $\D$-isomorphism property. Hence we may apply Theorem \ref{thmlip} and Corollary \ref{corlip}.
\end{proof}

 We sum up some of the properties in the following table, to be compared with table \ref{table1}.
{\setlength{\arrayrulewidth}{0.5mm}
	\renewcommand{\arraystretch}{1.5} 
	
	\begin{table}[H]
		\label{table2}
		\begin{tabular}{|m{2.7cm}|m{2.7cm}|m{2.7cm}|m{2.7cm}|}
			\hline
			\bfseries {} &\centering{\bfseries{Teichm\"uller}}  &\centering{\bfseries{Families}} & \bfseries{Examples} \\
			\hline
			\centering{double} & isomorphic punctured slices & existence of i\-so\-mor\-phic local punctured de\-for\-mations & Base point of a jumping family\\
			\hline
			\centering{split} & isomorphic punctured neighborhoods  & duality between local punctured deformations & Hirzebruch surface $\mathbb F_2$\\
			\hline
		\end{tabular}
		\vskip.3cm
		\caption[withinsection]{Pathologies for $X_0$ K\"ahler } 
	\end{table}
 
 \section{Cartography of the Teichm\"uller stack of K\"ahler structures}
 \label{seccarto}
 In this final section, we gather all the previous results to describe the geography of the Teichm\"uller stack, that is to give a geometric picture of the different strata of points (normal, exceptional, ...) in the Teichm\"uller stack. We first deal with the set of K\"ahler points for which this cartography is much more precise. Then we deal with the general case. At the end, we discuss the notion of holonomy points and their repartition and add a few remarks on the geography of bad points in $\T$ vs. in GIT quotients.
 
\subsection{K\"ahler points}
\label{secKahlercartography}

\begin{theorem}[Structure Theorem for K\"ahler points]
	\label{thmstructureKahler}
	Let $M$ be a connected, compact, oriented $C^\infty$ manifold admitting complex K\"ahler structures. Then the Teichm\"uller stack of K\"ahler points $\mathscr{T}(M,V_K)$ admits the following structure:
	\begin{enumerate}[\rm i)]
		\item Jumping points form a strict analytic substack $\mathscr{J}$. At a generic jumping point $X_0$, we have that $\mathscr{T}(M,V_K)$ is locally isomorphic to the Kuranishi stack, hence locally homeomorphic to the quotient of the Kuranishi space by an equivalence relation induced by $\Aun$. It is however far from being a manifold or an orbifold: the equivalence classes are analytic submanifolds of $K_0$ but that of $X_0$ is the base point, whereas others are positive dimensional, hence it is not locally Hausdorff.
		\item If non-empty - see Conjecture {\rm \ref{mainconj}, I} -, the closure of exceptional points form another strict analytic substack $\mathscr{E}$. At an exceptional point, we have that the Kuranishi stack only admits a finite \'etale projection onto $\mathscr{T}(M,V_K)$. At a generic exceptional point, $\mathscr{T}(M,V_K)$ is locally homeomorphic to a finite quotient of its Kuranishi space, but is not an orbifold.
		\item These two analytic substacks may intersect. Generic points refered to in points {\rm i)} and {\rm ii)} are points which are not in the intersection $\mathscr{J}\cap\mathscr{E}$. Points in the intersection combine the two pathologies: $\mathscr{T}(M,V_K)$ is neither locally Hausdorff nor isomorphic to the Kuranishi stack.
		\item The complementary Zariski open substack $\mathscr{O}$ contains the open substack of normal K\"ahler points $\mathscr{T}(M,\mathcal N_K)$. At a normal point, we have that $\mathscr{T}(M,\mathcal N_K)$ is locally isomorphic to the Kuranishi stack and locally homeomorphic to an orbifold. Moreover, there exists a global analytic morphism from $\mathscr{T}(M,\mathcal N_K)$ to a finite \'etale analytic stack with same topological quotient.
		\item The left points of $\mathscr{O}$ are points whose Kuranishi space is not reduced and admitting an automorphism of $\Azero$ with no extension as a local isomorphism of the Kuranishi family inducing the identity on the base. 
	\end{enumerate}
	The previous cartography also holds for the $\Z$-Teichm\"uller stack of K\"ahler points $\mathscr{T}^\Z(M,V_K)$ with the obvious changes in the statements.
\end{theorem}
Some additional remarks complete the statement of Theorem \ref{thmstructureKahler}. They also hold in the case of the $\Z$-Teichm\"uller stack of K\"ahler points $\mathscr{T}^\Z(M,V_K)$ with the obvious changes in the statements.
\begin{enumerate}[\rm a)]
	\item Recall that the analytic substacks of jumping points, resp. closure of exceptional points, are given locally by analytic subsets of Kuranishi spaces that glue when identifying two distinct Kuranishi spaces via a local isomorphism. Hence these abstract statements on stacks can be seen locally as classical statements about the geometry of Kuranishi spaces. However, this point of view, though more concrete and geometric, gets rid of the global behaviour, which is the crux of \cite{LMStacks} and of this paper.
	\item The construction and the nature of a local moduli space in the classical sense (i.e. a set with some structure such that every isomorphism class of complex structures on $M$ close to $X_0$ is encoded in a unique point of it) can be deduced from Theorem \ref{thmstructureKahler}. It is always a quotient of the Kuranishi space. If $X_0$ is normal K\"ahler, it is a finite quotient fixing the base point, hence a complex orbifold; indeed it is the quotient of $K_0$ by $\text{Map}^1(X_0)$. If $X_0$ is exceptional K\"ahler and not jumping, or in the closure of exceptional points but not jumping it is also a finite quotient of $K_0$ but not given by a proper group action so it is not a complex orbifold. If it is jumping non exceptional point, it is given as the leaf space of the foliation of $K_0$ described in \cite{ENS}, see point \ref{itemfoliation}). Finally if it is jumping and in the closure of exceptional points, it is a finite quotient of the leaf space of this foliation.
	\item Points where $\mathscr{T}(M,V_K)$ is locally isomorphic to the Kuranishi stack are points where the action of $\D$ onto $V_K$ is proper. The neighborhood of such a point only depends on two data, both encoding in $X_0$: the Kuranishi space $K_0$ and the extension of $\Aun$ to the Kuranishi family. 
	\item Recall that, by Theorem \ref{thmpathosubstack}, split and double points are jumping or exceptional points. Hence, the worst pathologies of the Teichm\"uller stack occur in these two substacks. 
	\item\label{itemfoliation} The foliation of $K_0$ of \cite{ENS} is given as follows. First decompose $K_0$ into strata 
	\begin{equation*}
		(K_0)_a:=\{J\in K_0\mid h^0(J)\leq a\}
	\end{equation*}
	Each non-empty $(K_0)_a$ is a Zariski open set of $K_0$ and the differences $(K_0)_{a+1}\setminus (K_0)_a$ are analytic subsets of $K_0$. Then each stratum admits a holomorphic regular foliation\footnote{\ The leaves are complex manifolds but the transversals may be singular, see the discussion in \cite{ENS}.} whose dimension is given on $(K_0)_a$ by the difference between $a$ and the minimal value of $h^0$ on $K_0$.
\end{enumerate}

\begin{proof}
	Theorem \ref{thmstructureKahler} gathers results proved in the previous sections. Point i) is a mixing of the well known upper semi continuity property of $h^0$ for the Zariski topology, of the foliated structure of the Kuranishi space proved in \cite{ENS} and of Corollary \ref{interlemma}. Point ii) is a rephrasing of Corollary \ref{5thcor} and Theorem \ref{thmfinite2} and point iii) follows from the two previous ones. Point iv) is an adaptation of Theorem \ref{thmnormal}, taking into the finiteness properties in the K\"ahler setting. Point v) is a rephrasing of the gap between normal and non-exceptional and non-jumping points when $K_0$ is not reduced.
\end{proof}

\begin{example}
	\label{exampleHirz}
	Take $M=\mathbb S^2\times\mathbb S^2$ corresponding to the complex structures of Hirzebruch surfaces $\mathbb{F}_{2a}$ for $a\in\mathbb Z$, cf. Remark \ref{topgenericity}. We note that $\mathscr{T}(M)$ is not an analytic stack but an inductive limit of analytic stacks, see \cite{ENS}. Now, all points are K\"ahler. There is no exceptional points, but a normal point, $\mathbb F_0$, that is $\mathbb P^1\times\mathbb P^1$, and all others points are jumping points. The fact that normal points fill a Zariski open substack can be seen by looking at the Kuranishi space of $\mathbb{F}_{2a}$ which always contains a Zariski open subset of points corresponding to $\mathbb P^1\times\mathbb P^1$.
\end{example}

\subsection{The general case}
\label{seccartographygeneral}

\begin{theorem}[Structure Theorem -- the general case]
	\label{thmstructuregeneral}
	Let $M$ be a connected, compact, oriented $C^\infty$ manifold admitting complex structures. Then the Teichm\"uller stack $\mathscr{T}(M)$ admits the following structure:
	\begin{enumerate}[\rm i)]
		\item Jumping points form a strict analytic substack $\mathscr{J}$. At a generic jumping point $X_0$, we have that $\mathscr{T}(M)$ is locally isomorphic to the Kuranishi stack, hence locally homeomorphic to the quotient of the Kuranishi space by an equivalence relation induced by $\Aun$. It is however far from being a manifold or an orbifold: the equivalence classes are analytic submanifolds of $K_0$ but that of $X_0$ is a the base point or at worst a sequence of points accumulating onto it, whereas others are positive dimensional, hence it is not locally Hausdorff.
		\item  The closure of exceptional points form a substack $\mathscr{E}$ of special importance. At an exceptional point, we have that the Kuranishi stack only admits an \'etale projection onto $\mathscr{T}(M,V_K)$. At a generic exceptional point, $\mathscr{T}(M,V_K)$ is locally homeomorphic to an at most discrete quotient of its Kuranishi space, but is not a group quotient. Exceptional points divides into points with an exceptional cycle, vanishing points and wandering points\footnote{ see however part III of Conjecture \ref{mainconj}}. 
		\item These two substacks may intersect. Generic points refered to in points {\rm i)} and {\rm ii)} are points which are not in the intersection $\mathscr{J}\cap\mathscr{E}$. Points in the intersection combine the two pathologies: $\mathscr{T}(M,V_K)$ is neither locally Hausdorff nor isomorphic to the Kuranishi stack.
		\item The complementary open substack $\mathscr{O}$ contains the open substack of normal K\"ahler points $\mathscr{T}(M,\mathcal N_K)$. At a normal point, we have that $\mathscr{T}(M)$ is locally isomorphic to the Kuranishi stack and locally homeomorphic to the quotient of $K_0$ by the action of the discrete group $\text{Map}^1(X_0)$ that fixes the base point. Moreover, there exists a global analytic morphism from $\mathscr{T}(M,\mathcal N_K)$ to a \'etale analytic stack with same topological quotient.
		\item The left points of $\mathscr{O}$ are points whose Kuranishi space is not reduced and admitting an automorphism of $\Azero$ with no extension as a local isomorphism of the Kuranishi family inducing the identity on the base. 
	\end{enumerate}
	The previous cartography also holds for the $\Z$-Teichm\"uller stack  $\TZ$ with the obvious changes in the statements.
\end{theorem}

We state Theorem \ref{thmstructuregeneral} in parallel to Theorem \ref{thmstructureKahler}. It is important to notice that points i) and iii) are identical in both Theorems whereas the statements in the other two points are much weaker in the general case. Indeed,

\begin{enumerate}[\rm a)]
	\item In point ii), the closure of exceptional points is just a substack, a priori not an analytic substack. In other words, it is given locally by  (closed) subsets of Kuranishi spaces that glue when identifying two distinct Kuranishi spaces via a local isomorphism, but not by analytic subsets. Conjecture \ref{mainconj}, II asserts that nothing more precise can be said. Also exceptional points are a priori of three types, following Proposition \ref{propex}, compare with Corollary \ref{5thcor}. Part III of Conjecture \ref{mainconj} asserts that every exceptional point is wandering but at this stage we cannot prove anything precise about the geography of each type of exceptional point.
	\item Similarly, in point iv), the Zariski open substack of Theorem \ref{thmstructureKahler} is simply an open substack in Theorem \ref{thmstructuregeneral}. And the finite quotient is replaced with a discrete one.
	\item Here again, the construction and the nature of a local moduli space in the classical sense can be deduced from Theorem \ref{thmstructuregeneral} and this gives a good idea of what we lost in the statement of the general case. The local moduli space is always a quotient of the Kuranishi space. If $X_0$ is normal, it is an at most discrete quotient by $\text{Map}^1(X_0)$ fixing the base point. So it is not always a complex orbifold. If $X_0$ is exceptional and not jumping, or in the closure of exceptional points but not jumping it is also a discrete quotient of $K_0$ but not given by a proper group action. If it is jumping non exceptional point, it is given as the leaf space of the foliation of $K_0$ described in \cite{ENS}. Finally if it is jumping and in the closure of exceptional points, it is a discrete quotient of the leaf space of this foliation.
	\item As in Theorem \ref{thmstructureKahler}, points where $\mathscr{T}(M)$ is locally isomorphic to the Kuranishi stack are points where the action of $\D$ onto $V_K$ is proper. The neighborhood of such a point only depends on two data, both encoding in $X_0$: the Kuranishi space $K_0$ and the extension of $\Aun$ to the Kuranishi family. 
	\item Recall that, by Theorem \ref{thmpathosubstack}, split and double points are jumping or exceptional points. Hence, the worst pathologies of the Teichm\"uller stack occur in these two substacks. 
\end{enumerate}

\begin{proof}
	This is completely similar to the proof of Theorem \ref{thmstructureKahler}. Point i) is a mixing of the well known upper semi continuity property of $h^0$ for the Zariski topology, of the foliated structure of the Kuranishi space proved in \cite{ENS} and of Corollary \ref{interlemma2}. Point ii) is a rephrasing of Proposition \ref{propex} and Theorem \ref{finitethm} and point iii) follows from the two previous ones. Point iv) is essentially Theorem \ref{thmnormal} and point v) just makes explicit the gap between normal and non-exceptional and non-jumping points when $K_0$ is not reduced.
\end{proof}

\begin{example}
	\label{exHopf3}
	We go back to Hopf surfaces. We already saw in Example \ref{exHopf2} that the function $h^0$  varies from $2$ to $4$ and that (a connected component of) the normal Teichm\"uller stack identifies with the bounded domain $\mathbb D^*\times\mathbb D$ of $\mathbb C^2$. For the Teichm\"uller stack, we make use of the results of C. Fromenteau \cite{clement}. 
	
	Let $\mathbb M$ be the product $\text{GL}_2^c(\C)\times\C$, where the superscript $c$ stands for contracting, that is $\text{GL}_2^c(\C)$ only contains invertible matrices with eigenvalues of modulus strictly less than one. Denoting by $\lambda_1$ and $\lambda_2$ the eigenvalues of a matrix, we define, for each $n\in\mathbb N^*$, the $n$-resonances detection function
	\begin{equation}
		\label{Rn}
		(M,t)\in\mathbb{M}\longmapsto R_n(M):= (\lambda_1^n-\lambda_2)(\lambda_2^n-\lambda_1)\in\C
	\end{equation}
	Observe that $R_n$ is holomorphic as a symmetric function of eigenvalues. The zero set of $R_n$ is exactly the analytic subset of matrices with a resonance of order $n$.
	
	Let $S_n$ be the open subset of $\mathbb{M}$ consisting of couples $(M,t)$ with $M$ non-resonant or resonant of order $n$. By definition, the union of all $S_n$ is $\mathbb{M}$. Notice that
	\begin{equation}
		\label{SHopf}
		S:=S_i\cap S_j\qquad i\not = j
	\end{equation}
	is independent of the choice of $i\not =j$ and corresponds to matrices with no resonances. On $S_n$, consider the $\mathbb Z$-action on $\mathbb{M}\times\C^2\setminus\{0\}$ generated by
	\begin{equation}
		\label{ZactionHopf}
		\left ((M,t),\begin{pmatrix}
			z_1\\
			z_2
		\end{pmatrix}\right )\longmapsto \left ((M,t),M\begin{pmatrix}
		z_1\\
		z_2
		\end{pmatrix}
		+t\begin{pmatrix}
			z_2^n\\
			z_1^n
		\end{pmatrix}
		\right )
	\end{equation}
	This defines a reduced family $\mathcal{X}_i$ of Hopf surfaces above $S_i$.
	When $M$ has no resonance of order $n$, the contracting biholomorphism of \eqref{ZactionHopf} is equivalent to a linear diagonal one, and the corresponding fiber is biholomorphic to a linear diagonal Hopf surface, regardless of the value of $t$. However, if $R_n(M)$ is zero, and $t$ is not zero, then the contracting biholomorphism of \eqref{ZactionHopf} cannot be linearized but is equivalent to \eqref{normalform}. In particular, we can glue $\mathcal X_i$ and $\mathcal X_j$ above $S$ for every $i\not =j$, thus producing a family of Hopf surfaces over $\mathbb{M}$. This family is complete at every point, making of $\mathbb{M}$ a connected atlas of a connected component of $\mathscr{T}(\mathbb S^3\times\mathbb S^1)$.
	   
	 The associated groupoid structure is described as follows. Let $G$ be the Lie group biholomorphic to $\text{GL}_2(\C)\times\C$ as a complex manifold but with the following product rule
	\begin{equation}
		\label{product}
		(A,t)\ast (B,s)=(AB, t+s\det A)
	\end{equation}
	Then one may define 
	\begin{enumerate}[a)]
		\item a holomorphic action $\cdot$ of $G$ onto $\mathbb{M}$.
		\item a holomorphic injection $\imath$ of $\mathbb{M}$ into $G$
	\end{enumerate}
	such that  the Lie groupoid $(G\times \mathbb{M})/\mathbb{Z}\rightrightarrows \mathbb{M}$ is the desired groupoid. Here the $\mathbb{Z}$-action is defined as
	\begin{equation}
		\label{Zaction}
		(p,g,m)\in\mathbb{Z}\times G\times \mathbb{M}\longmapsto (\imath (m)^pg, m)
	\end{equation} 
	and the source and target maps are the projections of the maps\footnote{\label{Zgerbe}\ The action groupoid $G\times \mathbb{M} \rightrightarrows \mathbb{M}$ with source and target maps defined in \eqref{onceagain} is an atlas for the stack of reduced $\mathbb S^3\times \mathbb S^1$-deformations admitting a covering $\C^2\setminus \{(0,0)\}$-deformation plus a choice of a base point in the covering family. Together with $\imath$, this forms a gerbe with band $\mathbb{Z}$.} 
	\begin{equation}
		\label{onceagain}
		(g,m)\in G\times M\longmapsto m\qquad\text{ and }\qquad (g,m)\mapsto m\cdot g
	\end{equation}
	\begin{remark}
		\label{rkHopf}
		In \cite{Dabro}, a connected family of Hopf surfaces containing a copy of every Hopf surface and complete at each point, that is a connected atlas of a connected component of $\mathscr{T}(\mathbb S^1\times\mathbb S^3)$, is also constructed. However, the associated groupoid has a more complicated structure than \eqref{Zaction}. The simple form of \eqref{Zaction} is used in a crucial way in \cite{clement} to compute some de Rham cohomology groups of $\mathscr{T}(\mathbb S^1\times\mathbb S^3)$, showing in particular the existence of a very particular class in dimension $2$ that plays the role of the Euler class of the $\mathbb Z$-gerbe of footnote \ref{Zgerbe}.
	\end{remark}
	We may now describe the geography of $\mathscr{T}(\mathbb S^1\times\mathbb S^3)$ by looking at the geometry of $\mathbb{M}=\text{GL}_2^c(\C)\times\C$. Recall that there is no exceptional point.
	\begin{enumerate}[a)]
		\item \label{itemh0=4} The function $h^0$ equals $4$ on the analytic subset $\{R_1=0\}\cap \{t=0\}$ of couples $(\lambda Id,0)$.
		\item \label{itemh0=3} It equals $3$ on the countable union of analytic subsets $\{R_p=0\}\cap \{t=0\}$ of couples $(M,0)$ with eigenvalues $(\lambda^p,\lambda)$ for $p>1$.
		\item It is equal to $2$ at every other point so the set of jumping points is described in \ref{itemh0=4}) and \ref{itemh0=3}).
		\item \label{itemsplit} Every point with $h^0$ equal to $3$ or $4$ is a double point but is not split.
		\item The subset of normal points is equal to
		\begin{equation}
			\mathbb{M}\setminus\cup_{p\geq 1}(\{R_p=0\}\cap \{t=0\})
		\end{equation}
		Since $\text{Map}^1$ is equal to the identity for all normal points, they are all manifold points.
		\item The holomorphic map
		\begin{equation}
			(M,t)\in\mathbb{M}\longmapsto f(M):=(\det M,(\text{Tr }M)/2)\in\mathbb D^*\times\mathbb D
		\end{equation}
		descends as the mapping from the open substack of normal points to the normal Teichm\"uller stack, which is thus identified with $\mathbb D^*\times\mathbb D$.
	\end{enumerate}
	Every statement in the list above is clear except for item \ref{itemsplit}). It can be proven as follows.
	
	\begin{proof}[Proof of item \ref{itemsplit})]
	For $p\geq 1$, the family $\mathcal{X}_{\lambda,p}\to\mathbb{D}$ given by the quotient of $\C^2\setminus\{0\}\times\mathbb{D}$ by the group generated by 
	\begin{equation}
		\label{eqHopfnonsplit1}
		(z,w,t)\longmapsto (\lambda^p z+tw^p,\lambda w,t)
	\end{equation}
	is a jumping family with base point $\left (\begin{smallmatrix}
		\lambda^p &0\\
		0 &\lambda
	\end{smallmatrix}\right )$
	\footnote{ We identify a contracting matrix $A$ and the Hopf surface given as the quotient of $\C^2\setminus\{0\}$ by the group generated by $A$.}
	and jumping point $\left (\begin{smallmatrix}
		\lambda^p &1\\
		0 &\lambda
	\end{smallmatrix}\right )$ showing that any point with $h^0$ equal to $3$ or $4$ is double. Assume now that $\lambda Id$ is a split point. We recall that $K_0$ is a neighborhood $U$ of of $\lambda Id$ in $\text{GL}_2^c(\C)$ and $K_0^*$ is the corresponding punctured neighborhood $U^*$. The Kuranishi family $\mathscr{K}_0\to K_0$ is obtained as the quotient of $\mathbb{C}^2\setminus\{0\}\times U$ by the group $\Z$ acting on the fiber $\mathbb{C}^2\setminus\{0\}\times \{A\}$ as the group generated by $A$ and $\mathscr{K}_0^*$ is obtained by remowing the fiber at $\lambda Id$. Since we assume that it is a split point, there exists a stack isomorphism  of $\mathscr{T}(M,V^*)$ that sends $\mathscr{K}_0^*\to K_0^*$ isomorphically to itself. Thus there exists an isomorphism $\Phi$ satisfying
	\begin{equation*}
		\label{cdHopfnonsplit}
		\begin{tikzcd}
			U^* \arrow[r,"\Phi"]\arrow[d,"f"'] &U^*\arrow[d,"f"]\\
			\mathbb{D}^*\times\mathbb{D}\arrow[r,"Id"]&\mathbb{D}^*\times\mathbb{D}
		\end{tikzcd}
	\end{equation*}
	that lifts to an isomorphism $\Psi$ of the family $\mathscr{K}_0^*$.
	Recall that two Hopf surfaces $A$ and $B$ are isomorphic if and only if the matrices $A$ and $B$ are conjugated. From this observation, we deduce the existence of a holomorphic mapping
	\begin{equation}
		\label{eqHopfnonsplitP}
		A\in U^*\longrightarrow P_A\in\text{GL}_2(\C)
	\end{equation}
	such that
	\begin{equation}
		\label{eqHopfnonsplitPhietPsi}
		\Phi(A)=P_A A P_A^{-1}\qquad\text{ and }\qquad \Psi([z,w],A)=([P_A(z,w)],\Phi(A))
	\end{equation}
	By Hartogs Theorem, $\Phi$ and $P_A$ extend holomorphically to $\lambda Id$, hence $\Psi$ also so the stack isomorphism of $\mathscr{T}(M,V^*)$ extends to a stack isomorphism of $\mathscr{T}(M,V)$ sending any family to an isomorphic one. This contradicts Definition \ref{defsplit}, proving that $\lambda Id$ is non split.
	
	The case of $\left (\begin{smallmatrix}
		\lambda^p &0\\
		0 &\lambda
	\end{smallmatrix}\right )$ is similar with $K_0^*$ being equal to a small punctured neighborhood of $(\lambda^p,\lambda, 0)$ in $\C^3$. Indeed, we deduce from \cite[Thm 2]{Wehler} the following facts
	\begin{enumerate}[i)]
		\item $(\alpha,\beta,t)$ and $(\alpha',\beta',t')$ encode isomorphic surfaces if and only if $\alpha=\alpha'$, $\beta=\beta'$ and both $t$ and $t'$ are either zero or non-zero. An isomorphism of Kuranishi families thus induces an isomorphism
		\begin{equation}
			\label{eqHopfnonsplit2}
			(\alpha,\beta,t)\in K_0^*\longmapsto (\alpha,\beta,tf(\alpha,\beta,t))\in K_0^*
		\end{equation}
		with $f:K_0^*\to\C$ holomorphic and non-vanishing for $t\not =0$.
		\item The lift of \eqref{eqHopfnonsplit2} to the universal covering $\C^2\setminus\{0\}\times K_0^*$ of the Kuranishi family is induced by a holomorphic mapping
		\begin{equation}
			\label{eqHopfnonsplit3}
			F\ :\ K_0^*\longrightarrow \{(z,w)\in\C^2\mapsto (az+bw^p,cw)\mid ac\not =0\}
		\end{equation}
			\end{enumerate}
		Since $K_0^*$ is a punctured nieghborhood of a point in $\C^3$, it follows from Hartogs Theorem that both $f$ and $F$ extend holomorphically to $K_0$ yielding an isomorphism between the full Kuranishi families.
	\end{proof}
\end{example}

\subsection{Holonomy points}
\label{secholonomy}
It is interesting to revisit Theorems \ref{thmstructureKahler} and \ref{thmstructuregeneral} through the concept of holonomy points. 
\begin{definition}
	\label{defholonomypoint}
	A point $X_0$ of $\mathscr{T}(M)$ is called a {\slshape holonomy point} if 
	\begin{enumerate}[\rm i)]
		\item It is not exceptional.
		\item There exists $f\in\Aun$ which does not admit an extension as an isomorphism of the germ of Kuranishi family at $X_0$ inducing the identity on the base.
	\end{enumerate}
	Replacing exceptional with $\Z$-exceptional and $\Aun$ with $\AZ$ gives rise to a {\slshape $\Z$-holonomy point}.
\end{definition}
Closely related is the notion of holonomy group. To define it, note the following lemma.
\begin{lemma}
	\label{lemmahol}
	Let $E^1(X_0)$ be the subset of $\Aun$ consisting of elements $f\in\Aun$ that admits an extension as an isomorphism of the germ of Kuranishi family at $X_0$ inducing the identity on the base. 
	
	Then $E^1(X_0)$ is a normal Lie subgroup of $\Aun$.
	
	Replacing $\Aun$ with $\AZ$, one obtains a normal Lie subgroup $E^\Z(X_0)$ of $\AZ$.
\end{lemma}

\begin{proof}
	The composition of two extensions inducing the identity on the base, and the inverse of such an extension, still induce the identity on the base. Closedness is immediate. Let $f\in E^1(X_0)$ and $g\in\Aun$. Let $F$ be an extension of $f$ as an isomorphism of the germ of Kuranishi family at $X_0$ inducing the identity on the base. Let $G$ be an extension of $g$ as an isomorphism of the germ of Kuranishi family at $X_0$. Then $G$ induces some map $h$ on the germ of Kuranishi space at $0$ and $h$ has no reason to be the identify. Now, $G\circ F\circ G^{-1}$ is an extension of $g\circ f\circ g^{-1}$ that induces $h\circ Id\circ h^{-1}$, that is the identity, on the base.
\end{proof}
We then define.
\begin{definition}
	\label{defholgroup}
	Let $X_0$ be a point of the Teichm\"uller stack. The {\slshape holonomy group} $\text{Hol}(X_0)$ of $X_0$ is defined as the quotient group $\Aun/E^1(X_0)$. The {\slshape $\Z$-holonomy group} $\text{Hol}^\Z(X_0)$ of $X_0$ is defined as the quotient group $\AZ/E^\Z(X_0)$. 
\end{definition}
Thus, a ($\Z$)-holonomy point is a point
\begin{enumerate}[i)]
	\item that is not  ($\Z$)-exceptional,
	\item and has a non-trivial  ($\Z$)-holonomy group.
\end{enumerate}
To understand why we exclude  ($\Z$)-exceptional points, recall that they correspond to points where the $\D$-orbits, resp. the $\DZ$-orbits in $K_0$ are not controlled by $\Aun$, resp. $\AZ$. More precisely,

\begin{lemma}
	\label{lemmaexandhol}
	A point $X_0$ of $\T$ is not exceptional if and only if it satisfies the following property:\\
	For $V$ small enough, two distinct points $J_1$ and $J_2$ of $K_0$ belong to the same $\D$-orbit if and only if there exists a local isomorphism of the Kuranishi family $\mathscr{K}_0\to K_0$ sending $J_1$ to $J_2$ and acting as an automorphism of the central fiber $X_0$.
	
	The same characterization holds for $X_0$ being not $\Z$-exceptional if we replace $\D$ with $\DZ$.
\end{lemma}

\begin{proof}
	Assume $K_0$ satisfies the property of Lemma \ref{lemmaexandhol}. Then, to any sequence \eqref{phin}, corresponds a sequence of local isomorphisms of the Kuranishi family fixing the central fibers, hence a sequence of automorphisms of $X_0$. By Lemma \ref{lemmaA1notex}, the sequence \eqref{phin} is a sequence of morphisms of the Kuranishi stack and $X_0$ is not exceptional.
	
	Conversely, if $X_0$ is not exceptional, then, for $V$ small enough, every morphism between two points $J_1$ and $J_2$ of $K_0$ is $(V,\mathcal{D}_1)$-admissible, hence is the evaluation at $J_1$ of a local isomorphism of the Kuranishi family $\mathscr{K}_0\to K_0$ sending $J_1$ to $J_2$ and acting as an automorphism of the central fiber $X_0$.
\end{proof}

Thus, with Definition \ref{defholonomypoint}, holonomy points are points such that non-trivial repetitions in the Kuranishi space\footnote{ that is pairs of distinct points in the Kuranishi space that encode the same complex structure up to $C^\infty$-isotopy.} exist but can be determined by computing the extensions of the elements of $\Aun$ as isomorphisms of the germ of Kuranishi family. Hence they are induced by the central fiber and not by the geometry of the $\D$-orbits in $\I$.

\begin{remark}
	\label{rkholpoints}
	Examples $\mathcal{X}_{a,b}$ of \cite{Aut1} have holonomy group $\mathbb{Z}_a$, since an arbitrary small deformation of them obtained by moving generically the $a$-th roots of unity has no $\Aun$-mapping class group. In the same way, Examples $\hat{\mathcal{X}}_P$ (with $P$ in the $t_a$-fiber) have $\Z$-holonomy group $\mathbb{Z}^{4a}$ by Theorem \ref{thmexampleAut1Z}. However, we do not know if these examples are ($\Z)$-holonomy points. To answer this question, we should decide whether they are exceptional or not. This supposes to know all their small deformations. Notice also that, if $\hat{\mathcal{X}}_P$ is not exceptional, this would be an example of a non-exceptional point in the closure of the set of exceptional points by Theorem \ref{thmexampleexc} and Corollary \ref{thmrepartitionZex}.
\end{remark}

Then, we may characterize generic jumping points by their holonomy group.

\begin{proposition}
	\label{propholjumping}
	Let $X_0$ be a point of $\T$. Then,
	$X_0$ is a non-exceptional jumping point if and only if it is a holonomy point with non-discrete holonomy group $\text{\rm Hol}(X_0)$.
\end{proposition}

\begin{proof}
	Jumping points have an automorphism group whose dimension is strictly greater than $E^1(X_0)$, hence their holonomy group is non-trivial and positive-dimensional.
	
	Conversely, if $\text{Hol}(X_0)$ is not discrete, it is a positive dimensional Lie group, hence taking a non-trivial element in its Lie algebra, the exponential flow of this element induces a $1$-dimensional submanifold of $K_0$ all of whose points correspond to the same complex structure up to $\D$-action. Hence the foliation of $K_0$ is not trivial and the $h^0$ function jumps.
\end{proof}

Let us analyze what happens at normal points. Let $X_0$ be a normal point. Let $\mathcal{E}^1$ be the subset of elements $f$ in $\mathcal{A}_1$ such that $s(f)=t(f)$. Since $X_0$ is normal, $\mathcal{E}^1$ contains $\mathcal{A}_0$. We may mimic Section \ref{secTSnormal} and define
an \'etale quotient groupoid $\mathcal{A}_1/\mathcal{E}^1\rightrightarrows K_0$. Its stackification over the analytic site describes classes of reduced $(M,V)$-families up to $\mathcal{E}^1$-equivalence. 

This stack can be defined over the full open set $\mathcal N$ of normal points. In this context, a reduced $(M,\mathcal N)$-family is $\mathcal{E}^1$-equivalent to a trivial family if it can be decomposed as local pull-back families glued by a cocycle in $\mathcal{E}^1$, cf. the proof of Theorem \ref{finitethm}; and an isomorphism of a reduced $(M,\mathcal N)$-family is $\mathcal{E}^1$-equivalent to the identity, if it is given by local $\mathcal{E}^1$-sections once decomposed as local pull-back families. We set
\begin{definition}
	\label{defholonomystack}
	The stack over the analytic site of $\mathcal{E}^1$-equivalence classes of reduced $(M,\mathcal N)$-families is called {\slshape the holonomy normal Teichm\"uller stack} and denoted by $\mathscr{H}\mathscr{T}(M)$.
\end{definition}
This is based on the result
\begin{lemma}
	\label{lemmaquotientgroupoidhol}
	Assume $X_0$ is normal. Then, the quotient space $\mathcal{A}_1/\mathcal{E}^1$ is an analytic space and the morphism $s$, resp. $t :\mathcal{A}_1\to K_0$, descends as an \'etale morphism from $\mathcal{A}_1/\mathcal{E}^1$ to $K_0$.
\end{lemma}
\noindent whose statement and proof are identical to those of Lemma \ref{lemmaquotientgroupoid}.
We then have

\begin{theorem}
	\label{thmnormalhol}
	The normal holonomy Teichm\"uller stack $\mathscr{HT}(M)$ satisfies the following properties
	\begin{enumerate}[\rm i)]
		\item It is an analytic \'etale stack with atlas an (at most) countable union of Kuranishi spaces.
		\item The isotropy group of $\mathscr{NT}(M)$ at $X_0$ is the discrete group $\text{\rm Hol}(X_0)$. It is finite if $X_0$ is in Fujiki class $(\mathscr{C})$.
		\item \label{itemfinite}Assume that $\text{\rm Hol}(X_0)$ is finite. Then, we may assume that it acts effectively on $K_0$ and the normal holonomy Teichm\"uller stack is locally isomorphic at $X_0$ to the effective orbifold $[K_0/\text{\rm Hol}(X_0)]$.
		\item In particular, if $X_0$ is a normal point with no holonomy, then the normal holonomy Teichm\"uller stack is locally isomorphic to $K_0$ at $X_0$.
		\item \label{itemgerbe} There is a natural \'etale morphism from $\mathscr{NT}(M)$ to $\mathscr{HT}(M)$. Moreover, these two stacks are associated to the same topological quotient space.
		\item Locally at $X_0$, the morphism from $\mathscr{NT}(M)$ to $\mathscr{HT}(M)$ makes of it a gerbe with band $E^1(X_0)/\Azero$.
	\end{enumerate}
\end{theorem}

Of course all this applies to the $\Z$-Teichm\"uller stack and we may thus define a {\slshape $\Z$-normal holonomy Teichm\"uller stack $\mathscr{HT}^\Z(M)$} that satisfies the properties listed in Theorem \eqref{thmnormalhol} with the obvious changes.
\begin{remark}
	\label{rknormalhol}
	Theorem \ref{thmnormalhol} has of course to be compared with Theorem \ref{thmnormal}. The main difference appears in item \ref{itemfinite}): assuming finiteness of the holonomy group of $X_0$, the local form of $\mathscr{HT}(M)$ is an {\slshape effective} orbifold whereas, assuming finiteness of the $\Aun$ mapping class group, the local form of $\mathscr{NT}(M)$ is a {\slshape possibly non-effective} orbifold. The explanation is given in item \ref{itemgerbe}): the $\text{\rm Hol}(X_0)$-action is the effective action induced by the $\text{Map}^1(X_0)$-action.
	
	Hence, the normal holonomy Teichm\"uller stack is the closest to a geometric moduli space. But it forgets about part of the automorphism group $\Aun$. 
\end{remark}
\begin{proof}
	Point i) follows from Lemma \ref{lemmaquotientgroupoidhol} and the first of point ii) from the mere definition of $\text{\rm Hol}(X_0)$. If $X_0$ is in Fujiki class $(\mathscr{C})$, its automorphism group is finite, hence also its holonomy group. Point iii) can be proved along the same lines that the corresponding statement in Theorem \ref{thmnormal}. An automorphism $f$ of $\Aun$ acts as the identity on $K_0$ if and only if its extension $\text{Hol}_f$ is the identity. But this means that the class of $f$ in $\text{\rm Hol}(X_0)$ is the class of the identity, showing effectiveness and finishing the proof of point iii). If $X_0$ has no holonomy, the effective orbifold chart of point iii) is a local isomorphism with $K_0$, proving iv). Since $\Azero$ is included in $E^1(X_0)$ at a normal point, there is a forgetful functor from $\mathscr{NT}(M)$ to $\mathscr{HT}(M)$. Its fiber at $X_0$ is given by the discrete group $E^1(X_0)/\Azero$ which acts trivially on $K_0$ making of $\mathscr{NT}(M)$ to $\mathscr{HT}(M)$ a gerbe with band $E^1(X_0)/\Azero$. This proves iv) and v).
\end{proof}

\subsection{Teichm\"uller stack and GIT quotients}
\label{secGIT}
In this last subsection, we want to say a few words in the case of the Teichm\"uller stack being isomorphic to a quotient stack $[X/G]$ with $X$ affine, resp. projective, and $G$ reductive, see \cite{Theo} for an example. In such a situation, we can also form the GIT quotient $X\git G$, which has its own geography of stable points, resp. unstable, semistable and stable points. We would like to compare both type of quotients. We note that a thorough study of this question is done in \cite{AHLH} in an algebraic context. Here, we content ourselves with some naive geometric remarks. We refer to \cite{Hoskins} for basics on GIT theory.

Let us start with $X$ being affine. Then, the algebra of $G$-invariant functions on $X$ is finitely generated and the associated affine scheme $X\git G$ comes equipped with a natural map $X\to X\git G$ whose fibers are the closure of the orbits. This good quotient is geometric if and only all orbits are closed. The stable points, that is the points with closed orbits and finite stabilizers, form a Zariski open subset $X^s$ of $X$ and the restriction of $X\to X\git G$ to $X^s$ is a geometric quotient. Note however that $X^s$ may be empty. Note also that there may exist an open subset of $X$ bigger than $X^s$ such that the restriction of $X\to X\git G$ to it is a geometric quotient. 

\begin{proposition}
	\label{propaffgeoquotient}
	Let $X$ be an affine scheme and $G$ a reductive group acting rationally on it. Assume that $\T$ is isomorphic to the quotient stack $[X/G]$. Then $X\git G$ is not homeomorphic to the orbit space $X/G$, hence to the geometric quotient  of $\T$ if and only if one the following equivalent conditions are satisfied
	\begin{enumerate}[\rm i)]
		\item There exist a double point $X_0$ and a jumping family based at $X_0$.
		\item There exists an injective morphism from the quotient stack $[\C/\C^*]$ (with $\C^*$ acting multiplicatively on $\C$) to $\T$.
	\end{enumerate}
	Moreover, the subset $X^s$ is included in the subset of points $X_0$ of $\T$ with finite $\Aun$.
\end{proposition}
 
\begin{proof}
	Since $\T$ is isomorphic to $[X/G]$, then $X$ is an atlas of $\T$, hence comes equipped with a family of reduced $M$-deformations $\mathcal{X}\to X$ which is complete at any point. Now, $X\git G$ is homeomorphic to the orbit space $X/G$ if and only if all $G$-orbits are closed. If this is not the case, then the main Theorem of affine GIT asserts that there exists $x\in X$ and $\C^*$ in $G$ suth that the limit of $g\cdot x$ when $g\in\C^*$ tends to zero is a point $y$ not belonging to the $x$-orbit. In other words, this gives a holomorphic mapping from $\C$ to $X$ such that $\C^*$ lands in the $x$-orbit and $0$ in the distinct $y$-orbit. Pulling-back $\mathcal{X}\to X$ through this morphism gives a jumping family. Its base point $0$ is a double point of $\T$. At the same time, this morphism descends as an injective morphism from $[\C/\C^*]$ to $\T$. Conversely, if there exists a jumping family, we may assume that its base is the unit disk $\mathbb D$ and that it is obtained by pull-back from $\mathcal{X}\to X$ along a non-constant map $\mathbb D\to X$ with $\mathbb D^*$ landing in a single $G$-orbit and $0$ in a distinct one, proving the existence of a non-closed $G$-orbit. Applying what we just proved, this shows the existence of an injective morphism from $[\C/\C^*]$ to $\T$. If we assume the existence of such a morphism, arguing as above, we obtain a jumping family.
	
	Finally, since $\T$ is isomorphic to $[X/G]$, then the stabilizer of a point is the automorphism group $\Aun$ of the corresponding compact complex manifold.
\end{proof}

Assume from now on that $X\subset\mathbb{P}^n$ is projective and choose a linearization of the action, that is a lift of the $G$-action to $\C^{n+1}$. The points $[z]\in X$ such that $G\cdot z$ does not accumulate onto $0$ in $\C^{n+1}$ form a Zariski open subset $X^{ss}$ of $X$ called the set of semistable points. Then, the algebra of $G$-invariant functions on the affine cone over $X$ is finitely generated and the projective scheme $X\git G$ associated to its projectivization comes equipped with a natural map $X^{ss}\to X\git G$ whose fibers are the closure of the orbits. This good quotient is geometric if and only all orbits are closed in $X^{ss}$. The stable points, that is the points with closed orbits in $X^{ss}$ and finite stabilizers, form a Zariski open subset $X^s$ of $X^{ss}$ and the restriction of $X^{ss}\to X\git G$ to $X^s$ is a geometric quotient. As in the affine case, $X^s$ may be empty and there may exist an open subset of $X$ bigger than $X^s$ such that the restriction of $X^{ss}\to X\git G$ to it is a geometric quotient. 

The main difference with the affine case is the fact that unstable orbits must be thrown away before taking the GIT quotient. In particular, $X\git G$ cannot be homeomorphic to the geometric quotient of $\T$ in presence of unstable points. Now, the notion of unstable/semistable points depends on the choice of a linearization and is strongly related to the projectivity of the quotient. Indeed, considering the affine cone $\tilde X\subset \C^{n+1}$ above $X\subset\mathbb P^n$, then the GIT quotient $X\git G$ is the projectivization of the affine GIT quotient $\tilde X\git G$ minus zero. This is the projectivization of $\tilde X^{ss}\git G$ where $z$ belongs to $\tilde X^{ss}$ if $G\cdot z$ does not accumulate onto $0$ in $\C^{n+1}$. For if $z$ does not belong to $\tilde X^{ss}$, its $G$-orbit accumulates onto zero and it is sent to $0$ in $\tilde X\git G$, preventing from projectivizing the whole quotient $\tilde X\git G$.

As a consequence, unstable points are not intrinsic and have no clear geometric meaning in the Teichm\"uller stack. Hence, we modify our setting and assume from now on that $\T$ is isomorphic to the quotient stack $[X^{ss}/G]$. In this new setting, Proposition \ref{propaffgeoquotient} can be easily adapted.

\begin{proposition}
	\label{propprojgeoquotient}
	Let $X$ be a projective scheme and $G$ a reductive group acting rationally on it. Choose a linearization and assume that $\T$ is isomorphic to the quotient stack $[X^{ss}/G]$. Then $X\git G$ is not homeomorphic to the orbit space $X^{ss}/G$, hence to the geometric quotient  of $\T$ if and only if one the following equivalent conditions are satisfied
	\begin{enumerate}[\rm i)]
		\item There exist a double point $X_0$ and a jumping family based at $X_0$.
		\item There exists an injective morphism from the quotient stack $[\C/\C^*]$ (with $\C^*$ acting multiplicatively on $\C$) to $\T$.
	\end{enumerate}
	Moreover, 
	the subset $X^s$ is included in the subset of points $X_0$ of $\T$ with finite $\Aun$.
\end{proposition}

\begin{proof}
Apply Proposition \ref{propaffgeoquotient} to $\tilde X\to\tilde X\git G$ restricted to $\tilde X^{ss}$.
\end{proof}

\end{document}